\documentclass[12pt]{amsart}
\usepackage{amsmath}
\usepackage{amssymb}
\usepackage{amscd}
\usepackage{graphicx}
\usepackage{bm}
\newtheorem{theorem}{Theorem}[section]
\newtheorem{lemma}[theorem]{Lemma}
\newtheorem{proposition}[theorem]{Proposition}
\newtheorem{corollary}[theorem]{Corollary}

\theoremstyle{definition}
\newtheorem{definition}[theorem]{Definition}

\numberwithin{equation}{section}

\newtheorem{remark}[theorem]{Remark}

\title[Almost duality for Saito structure]
{Almost duality for Saito structure and complex reflection groups}
\author{Yukiko Konishi}
\address{ 
Department of Mathematics, 
Kyoto University,
Kyoto 606-8502, Japan 
}
\email{konishi@math.kyoto-u.ac.jp}
\author{Satoshi Minabe}
\address{Department of Mathematics, 
Tokyo Denki University, 120-8551 Tokyo,  Japan}
\email{minabe@mail.dendai.ac.jp}
\author{Yuuki Shiraishi}
\address{ 
Department of Mathematics, 
Kyoto University,
Kyoto 606-8502, Japan 
}
\email{yshiraishi@math.kyoto-u.ac.jp}
\begin{document}
\maketitle
\begin{abstract}
We reformulate Dubrovin's almost duality of the Frobenius structure
to the setting of the Saito structures without metric. 
Then we formulate and study
the existence and uniqueness
problem of the natural Saito structure on the orbit spaces of
finite complex reflection groups 
from the viewpoint of the almost duality.
We give a complete answer 
to the problem for the irreducible groups.
\end{abstract}

\section{Introduction}
\subsection{Almost duality for Frobenius structures}
Dubrovin introduced the Frobenius structure in \cite{Dubrovin1993}. 
Later he defined the almost Frobenius structure 
in \cite{Dubrovin2004}.
These are both structures on the tangent bundles of manifolds
consisting of flat metrics (i.e. nondegenerate bilinear forms 
whose Levi--Civita connections are flat), 
multiplications
and nonzero vector fields. 
However,  they must satisfy slightly different conditions.
Dubrovin showed that given a Frobenius structure,
one can construct an almost Frobenius structure
and vice versa.
He called it the almost duality in the title of the article 
\cite{Dubrovin2004}.
Later,  Arsie and Lorenzoni introduced the notion of bi-flat F-manifold
\cite{Arsie-Lorenzoni2013}. 
They showed, in semisimple case, that it can be regarded 
as an extension of the almost duality to the Frobenius structure without metric \cite[\S 4]{Arsie-Lorenzoni2016}.  

\subsection{Complex reflection groups}
Let $V$ be a complex vector space of finite dimension $n$.
A finite complex reflection group (or a unitary reflection group) 
$G\subset GL(V)$ is 
a finite group generated by pseudo-reflections on $V$.
Reducible finite complex reflection groups are 
direct products of irreducible ones,
and irreducible finite complex reflection groups were 
classified by Shephard and Todd in 1954 \cite{ShephardTodd}.
The number of minimal generators for
the irreducible groups are either $n$ or $n+1$
and 
those with $n$ minimal generators are said to be well-generated.
They are also called duality groups because 
there is a duality between the degrees and the codegrees of such 
groups.
The duality groups include 
the finite Coxeter groups.
See \cite{LehrerTaylor} and \cite{OrlikTerao} for 
finite complex reflection groups.

\subsection{Frobenius structures on the orbit spaces of Coxeter groups}
In 1979, Saito \cite{Saito1993} showed the existence of flat coordinates on the orbit
spaces of finite Coxeter groups.  
See also Saito--Yano--Sekiguchi \cite{SaitoSekiguchiYano}.
In 1993, Dubrovin \cite{Dubrovin1993} constructed Frobenius structures on them. 
See also \cite{Dubrovin1998}. 
Later Dubrovin  \cite{Dubrovin2004} generalized the construction to the 
Shephard groups which include the Coxeter groups and 
form a proper subclass of the duality groups. 
These constructions may be regarded as applications
of the almost duality.

\subsection{Saito structures on the orbit spaces of the duality groups}
The Saito structure without metric introduced by Sabbah \cite{Sabbah}
is a structure weaker than the Frobenius structure,
consisting of a torsion-free flat connection and a multiplication
on the tangent bundle
together with a nonzero vector field.
We call it  
Saito structure for short throughout this paper. 
If a Saito structure admits a compatible metric,
it makes a Frobenius structure.
Recently,  Kato, Mano and Sekiguchi showed that 
there exist  certain polynomial Saito structures
on the orbit spaces of the duality groups\footnote{
In the case of Shephard groups which are not Coxeter groups, 
the natural Saito structures do not coincide  in general with the underlying
Saito structures of Dubrovin's Frobenius structures.}
\cite{KatoManoSekiguchi2015}. See also \cite{KatoManoSekiguchi2014}.
In \cite[\S 5]{Arsie-Lorenzoni2016},
Arsie and Lorenzoni studied the same polynomial
Saito structures based on their theory of bi-flat $F$-manifold
\cite{Arsie-Lorenzoni2013} and
computed many examples.

\subsection{Aim and results of the paper}
The aim of this article is to formulate and study
the existence and uniqueness
problem of the ``natural'' Saito structure 
for finite complex reflection groups 
from the viewpoint of the almost duality.
We give a complete answer 
to the problem for the irreducible groups.
Especially, it contains another proof of
Kato--Mano--Sekiguchi's result (i.e. the case of duality groups).

The paper consists of two parts.
The first part is devoted to formulating the almost duality
for the Saito structure.
We introduce the almost Saito structure in \S \ref{sec:ASS}
and show the duality between the Saito structure
and the almost Saito structure in \S \ref{sec:duality}
(Theorem \ref{thm:duality}).
In \S \ref{relationship-1}, 
we recall the definition of Frobenius manifold
and Dubrovin's almost duality
and explain the relationship with \S \ref{sec:ASS}
and \S \ref{sec:duality}.
In \S \ref{relationship-2},
we recall the definition of 
Arsie--Lorenzoni's bi-flat $F$-manifold 
\cite{Arsie-Lorenzoni2013} 
and show that 
the almost duality for the Saito structure
is a notion equivalent to
the bi-flat $F$-manifold.
\S \ref{matrix-rep} is devoted to matrix representations
of (almost) Saito structures.

In the second part,  we apply the formulation to
the orbit spaces of finite complex reflection groups.
In \S \ref{orbit-space}
we characterize,   using the almost duality, 
the Saito structures we are looking for.
We call it the natural Saito structure
because it comes from the trivial connection on $TV$.
In \S \ref{monic-degreen}, we show that unique natural Saito structures
exist for 
a certain class of finite complex reflection groups 
including the duality groups 
(Theorem \ref{main-theorem}, Corollary \ref{main-corollary}).
For the duality groups,
they coincide with
the Saito structures
studied by Kato--Mano--Sekiguchi 
\cite{KatoManoSekiguchi2015}
and
by Arsie--Lorenzoni \cite[\S 5]{Arsie-Lorenzoni2016}.
We also show that four irreducible groups
do not admit natural Saito structures (\S \ref{example-nonexistence}).
For the remaining irreducible groups, 
it turns out that 
all natural Saito structures are those induced from 
branched covering maps from the orbit spaces
of some duality groups (\S \ref{covering-case}, \S \ref{example-covering1}, \S \ref{example-covering2}).

In appendix \S \ref{appendix:tables}, the natural Saito structure for the rank two cases are listed. 
\S \ref{appendix:proofs} and \S \ref{appendix-C} contain  
proofs of some technical results.

\subsection{Conventions}
Throughout this paper, 
a manifold means a complex manifold. 
For a manifold $M$,
$TM$ denotes the holomorphic tangent bundle and
$\mathcal{T}_M$ its sheaf of local sections.
We write $x\in \mathcal{T}_M$ to mean that
$x$ is a local section of $TM$. 
A multiplication on $TM$ means an $\mathcal{O}_M$-bilinear 
map $\mathcal{T}_M\times \mathcal{T}_M \to \mathcal{T}_M$.

\subsection{Acknowledgements}
Y. K. and S. M. thank 
Jiro Sekiguchi for sending  the manuscripts of their articles,
and for helpful discussions.
The authors also thank
Hiroshi Iritani, Ikuo Satake and Atsushi Takahashi
for valuable discussion and comments.
We are grateful to anonymous referees for many 
useful comments, especially for pointing out wrong statements
in Proposition \ref{uniqueness1} and Lemma \ref{uniqueness2}.    

The work of Y.K. is supported in part by Grant-in-Aid for Challenging Exploratory Research 26610008 
and JSPS KAKENHI Kiban-S 16H06337.
The work of S.M. is supported in part by Grant for Basic Science Research
Projects from The Sumitomo Foundation and JSPS KAKENHI Grand number JP17K05228.
Y.S. is supported by
Research Fellowships of Japan Society
for the Promotion for Young Scientists.

\section{The Almost  Saito structure}\label{sec:ASS}
In this section, we first recall the definition of 
the Saito structure in \S \ref{sec:SS}
and  give a definition of the almost Saito structure in
\S \ref{subsec:ASS}.
In \S \ref{subsec:regularity}, we introduce 
the regular almost Saito structure
which is characterized by the property that
the multiplication is completely determined by 
the connection and the vector field $e$.
The regularity will play an important role
in application to the finite complex reflection groups.
In \S \ref{sec:two-par}, we explain that 
an almost Saito structure is always accompanied with 
a two-parameter family of almost Saito structures.
Roughly speaking,
one parameter $\lambda$ corresponds to 
the ``twist'' of the multiplication
and 
another parameter $\nu$ corresponds to the shift of 
the parameter $r$.
\subsection{Saito structure}\label{sec:SS}
The following notion was introduced by Sabbah in \cite{Sabbah}.
\begin{definition}
\label{def-Saito-structure}
A Saito structure (without a metric) on a manifold $M$ 
consists of 
\begin{itemize}
\item a torsion-free flat connection $\nabla$ on $TM$,
\item an associative commutative  multiplication $\ast$ 
on $TM$ with a unit $e\in \Gamma(M,\mathcal{T}_M)$,
\item a vector field $E\in \Gamma(M,\mathcal{T}_M)$ 
called {\it the Euler vector field},  
\end{itemize}
satisfying the following conditions.
\begin{eqnarray}
\nonumber (\bf{SS1})
&&\nabla_x (y\ast z)-y\ast \nabla_x \,z-\nabla_y(x\ast z)+x\ast \nabla_y \,z=[x,y]\ast z
\quad (x,y,z\in \mathcal{T}_M)~,
\\
\nonumber (\bf{SS2})
&&[E,x\ast y]-[E,x]\ast y-x\ast[E,y]=x\ast y\quad(x,y\in\mathcal{T}_M)~,
\\
\nonumber (\bf{SS3})
&&\nabla e=0~,
\\
\nonumber (\bf{SS4})
&&
\nabla_x\nabla_y E-\nabla_{\nabla_x y}E=0\quad (x,y\in \mathcal{T}_M)~.
\end{eqnarray}
\end{definition} 

\begin{remark}\label{E-shift}
If $(\nabla,\ast,E)$ is a Saito structure with a unit $e$,
then $(\nabla,\ast ,E-\lambda e)$ ($\lambda\in \mathbb{C}$)
is also a Saito structure.
Therefore we may be able to think that 
a Saito structure is always accompanied with a one-parameter family.
\end{remark}

We list a few formulas which will be useful later.
\begin{lemma} \label{tec1}
Let $(\nabla,\ast, E) $ be a Saito structure on $M$
with a unit $e$. Then the following holds.
For $x,y,z\in \mathcal{T}_M$, 
\begin{eqnarray}
\label{SS1}
&&[e,y\ast z]=y\ast [e,z]+[e,y]\ast z~,
\\
\label{SS2}
&& [E,e]=-e,\quad [e,E\ast x]=E\ast[e,x]+x~,\\
&& 
\label{SS3}
\nabla_{y\ast z} \,E-y\ast \nabla_z\, E-\nabla_y(E\ast z)
+E\ast \nabla_y\,z+y\ast z=0~,
\\
&&\label{SS4}
\big(E\ast \nabla_x(y\ast z)-y\ast \nabla_x (E\ast z)
+y\ast \nabla_{x\ast z}\,E
\big)
-(x\leftrightarrow y)=E\ast[x,y]\ast z~.
\end{eqnarray}
\end{lemma}
\begin{proof}
\eqref{SS1}:
substituting $x=e$ into ({\bf SS1}),
\begin{equation}\nonumber
\nabla_e(y\ast z)-y\ast \nabla_e \,z=[e,y]\ast z
\end{equation}
Using the torsion freeness of $\nabla$ and ({\bf SS3}),
we obtain
\begin{equation}\nonumber
[e,y\ast z]-y\ast [e,z]=[e,y]\ast z~.
\end{equation}
\eqref{SS2}: 
the first equation immediately follows from ({\bf SS2}) if we substitute $x=y=e$.
Substituting $y=E$ into \eqref{SS1}, 
we obtain the second equation:
\begin{equation}\nonumber
[e, E\ast z]=E\ast [e,z]+[e,E]\ast z=E\ast [e,z]+z.
\end{equation}
\eqref{SS3}: substituting $x=E$ into ({\bf SS1}),
\begin{equation}\nonumber
\nabla_E(y\ast z)-y\ast \nabla_E\,z-\nabla_y (E\ast z)
+E\ast \nabla_y\,z-[E,y]\ast z=0~.
\end{equation}
Changing $(x,y)$ to $(y,z)$ in ({\bf SS2}),
\begin{equation}\nonumber
[E,y\ast z]-[E,y]\ast z-y\ast [E,z]-y\ast z=0~.
\end{equation}
Subtracting, we obtain \eqref{SS3}.
\\
\eqref{SS4}: 
\begin{equation}\label{a1}
\begin{split}
0&=x\ast \eqref{SS3}-(x\leftrightarrow y)
=x\ast\big( \nabla_{y\ast z}\,E-\nabla_y (E\ast z)+E\ast \nabla_y z\big)
-(x\leftrightarrow y)
~.
\end{split}
\end{equation}
Then we have
\begin{equation}\nonumber
\begin{split}
0&=\eqref{SS3}-E\ast ({\bf SS1})
\\
&=\big(x\ast \nabla_{y\ast z}\,E-x \ast \nabla_y (E\ast z)+E\ast \nabla_y(x\ast z) \big)-(x\leftrightarrow y)+E\ast [x,y]\ast z~.
\end{split}
\end{equation}
\end{proof}

Now assume that $(\nabla,\ast,E)$ is a Saito structure on a manifold $M$ 
with a unit $e$.  
 If one chooses a nonzero constant $c\in \mathbb{C}^*$, 
one can construct a new multiplication $\ast'$ by
$x\ast'y =c x\ast y$ ($x,y\in \mathcal{T}_M$). 
Then  it is easy to see that
$(\nabla,\ast', E)$ is also a Saito structure on $M$ with a unit
$c^{-1}e$.
Therefore we introduce the following
equivalence relation.

\begin{definition}\label{def:equivalence-SS}
Two Saito structures $(\nabla,\ast,E)$ and 
$(\nabla',\ast',E')$ on $M$ are said to be equivalent
if $\nabla=\nabla'$, $E=E'$ and if there exists a nonzero constant
$c\in \mathbb{C}^*$ such that
$$
x\ast' y=c x\ast y \quad (x,y\in \mathcal{T}_M)~.
$$
\end{definition}

\subsection{The Almost Saito structure}\label{subsec:ASS}
\begin{definition}
\label{def-ASS}
An almost Saito structure  on a manifold $N$ 
with parameter $r\in \mathbb{C}$
consists of 
\begin{itemize}
\item a torsion-free flat connection $\boldsymbol{\nabla}$ on $TN$,
\item an associative commutative multiplication 
$\star$ 
on $TN$ with a unit $E\in \Gamma(N,\mathcal{T}_N)$, 
\item a nonzero vector field  $e\in \Gamma(N,\mathcal{T}_N)$
\end{itemize}
satisfying the following conditions.
\begin{eqnarray}
\nonumber ({\bf ASS1})
&&
\boldsymbol{\nabla}_x (y\star z)-y\star \boldsymbol{\nabla}_x \,z
-\boldsymbol{\nabla}_y(x\star z)+x\star \boldsymbol{\nabla}_y \,z=[x,y]\star z
\quad (x,y,z\in \mathcal{T}_N)~.
\\
\nonumber ({\bf ASS2})
&& 
[e,x\star y]-[e,x]\star y-x\star[e,y]+e\star x\star y=0\quad(x,y\in\mathcal{T}_N)~.
\\
\nonumber ({\bf ASS3})
&&
\boldsymbol{\nabla}_x E=r x  \quad (x\in \mathcal{T}_N)~.
\\
\nonumber ({\bf ASS4})
&&
\boldsymbol{\nabla}_x\boldsymbol{\nabla}_y \,e
-\boldsymbol{\nabla}_{\boldsymbol{\nabla}_x y}\,e+
\boldsymbol{\nabla}_{x\star y} \,e=0
\quad (x,y\in \mathcal{T}_N)~.
\end{eqnarray}
\end{definition}

Let us list some formulas.
\begin{lemma}\label{tec2}
Let $(\boldsymbol{\nabla},\star,e)$ be 
an almost Saito structure on $N$ with parameter $r$ and a unit 
$E$.
Then the following holds. For $x,y,z\in \mathcal{T}_N$,
\begin{eqnarray}
\label{ASS1}
&& [E,y\star z]=y\star [E,z]+[E,y]\star z~,
\\\label{ASS2}
&& [e,E]=e~,\qquad
[E,e\star z]=e\star[E,z]-e\star z~,
\\\label{ASS3}
&&
\boldsymbol{\nabla}_x(e\star y)
-e\star \boldsymbol{\nabla}_x\,y
-\boldsymbol{\nabla}_{x\star y}\,e
+x\star\boldsymbol{\nabla}_y\,e
+e\star x\star y=0~.
\\\label{ASS4}
&&
\big(
e\star \boldsymbol{\nabla}_x(y\star z)
-y\star \boldsymbol{\nabla}_x (e\star z)+y\star \boldsymbol{\nabla}_{x\star z}\,e
\big)-(x\leftrightarrow y)=e\star [x,y]\star z~.
\end{eqnarray}
\end{lemma}
\begin{proof}
\eqref{ASS1}: substituting $x=E$ into ({\bf ASS1}),
\begin{equation}\nonumber
\boldsymbol{\nabla}_E(y\star z)-y\star \boldsymbol{\nabla}_E\,z
=[E,y]\star z~.
\end{equation}
Using the torsion freeness of $\boldsymbol{\nabla}$ 
and ({\bf ASS3}), we obtain
\begin{equation}\nonumber
[E,y\star z]-y\star[E,z]=[E,y]\star z~.
\end{equation}
\eqref{ASS2}: 
the first equation immediately follows from ({\bf ASS2})
if we substitute $x=y=E$.
Substituting $y=e$ into \eqref{ASS1}, we obtain the second equation:
\begin{equation}\nonumber
[E,e\star z]=e\star [E,z]+[E,e]\star z=e\star [E,z]-e\star z~.
\end{equation}
\eqref{ASS3}: 
changing $(x,y,z)$ to $(e,x,y)$ in ({\bf ASS1}), we have
\begin{equation}\nonumber
\boldsymbol{\nabla}_e(x\star y)
-x\star \boldsymbol{\nabla}_e\,y
-\boldsymbol{\nabla}_x(e\star y)
+e\star \boldsymbol{\nabla}_x\,y-[e,x]\star y=0~.
\end{equation}
Subtracting this equation from ({\bf ASS2}),  we obtain \eqref{ASS3}.
\\
\eqref{ASS4} is obtained by
\begin{equation}\nonumber
\begin{split}
e\star ({\bf ASS1})
+x\star \text{(change $(x,y)$ to $(y,z)$ in \eqref{ASS3})}
-y\star \text{(change $(x,y)$ to $(x,z)$ in \eqref{ASS3})}~.
\end{split}
\end{equation}
\end{proof}

Let $c\in \mathbb{C}$ be a nonzero constant.
If $(\boldsymbol{\nabla},\star,e)$ is an 
almost Saito structure on a manifold $N$ with parameter $r$,
$(\boldsymbol{\nabla},\star, ce)$ is also an almost  Saito structure on $N$ with parameter $r$.
So we introduce the following
equivalence relation.
\begin{definition}\label{def:equivalence-ASS}
Two almost Saito structures $(\boldsymbol{\nabla},\star,e)$ and 
$(\boldsymbol{\nabla}',\star',e')$ on $N$ are said to be equivalent
if $\boldsymbol{\nabla}=\boldsymbol{\nabla}'$,
$\star=\star'$ and if there exists a nonzero constant
$c\in \mathbb{C}$ such that $e'=ce$.
\end{definition}

\subsection{The Regular almost Saito structure}
\label{subsec:regularity}
Let $N$ be a manifold.
Given a pair  $(\boldsymbol{\nabla},e)$
consisting of a connection $\boldsymbol{\nabla}$
on $TN$ and a vector field $e\in \Gamma(N,\mathcal{T}_N)$,
define
$\mathcal{Q}\in \mathrm{Hom}_{\mathcal{O}_N}(\mathcal{T}_N,\mathcal{T}_N)$ by
\begin{equation}\nonumber
\mathcal{Q}(x)=\boldsymbol{\nabla}_x \,e\quad (x\in \mathcal{T}_N)
\end{equation}
We say that the pair $(\boldsymbol{\nabla},e)$ is {\it regular}
if $\mathcal{Q}$  is an isomorphism.

\begin{lemma}
Let $e$ be a vector field on a manifold $N$ and 
let $\boldsymbol{\nabla}$ be a torsion free, flat connection 
on $TN$. 
If the pair $(\boldsymbol{\nabla},e)$ is regular,
then  a multiplication $\star$ on $TN$ satisfying \rm{({\bf ASS4})}
is unique and it is given by
\begin{equation}\label{mult}
x\star y=-\mathcal{Q}^{-1}(
\boldsymbol{\nabla}_x\boldsymbol{\nabla}_y \,e)
+\boldsymbol{\nabla}_x\,y~.
\end{equation}
\end{lemma}
\begin{proof} Immediate.
\end{proof}

\begin{proposition}\label{uniqueness1}
Let $\boldsymbol{\nabla}$ be a torsion free, flat connection on 
$TN$ and let $e\in \Gamma(N,\mathcal{T}_N)$
be a vector field on $N$.
Assume that 
the pair  $(\boldsymbol{\nabla},e)$ is regular.
Define the multiplication $\star\in \mathrm{Hom}_{\mathcal{O}_N}(\mathcal{T}_N\otimes \mathcal{T}_N,\mathcal{T}_N)$ by \eqref{mult}. 
\\
(1) $\star$ is  commutative.
\\
(2) $\star$ is associative if and only if \rm{({\bf ASS1})} holds.
\\
(3) $(\boldsymbol{\nabla},\star,e)$ is an almost Saito structure 
if and only if it satisfies \rm{({\bf ASS1})}, \rm{({\bf ASS2})} and \rm{({\bf ASS3})}.
\end{proposition}
\begin{proof}
(1) 
The commutativity of $\star$ follows from the flatness 
and the torsion freeness of 
$\boldsymbol{\nabla}$:
\begin{equation}\nonumber
\begin{split}
-\mathcal{Q}(x\star y-y\star x)
&=\boldsymbol{\nabla}_x\boldsymbol{\nabla}_y\,e
-\boldsymbol{\nabla}_y\boldsymbol{\nabla}_x\,e
-\mathcal{Q}(\boldsymbol{\nabla}_x\,y -\boldsymbol{\nabla}_y\,x)
\\
&=\boldsymbol{\nabla}_{[x,y]}\,e
-\mathcal{Q}([x,y])=0~.
\end{split}
\end{equation}
(2)  
By the commutativity of $\star$, we have
\begin{equation}\nonumber
\begin{split}
&\mathcal{Q}(y\star(z\star x)-(y\star z)\star x)\\
&=\mathcal{Q}(y\star(x\star z)-x\star (y\star z))
\\
&=-\boldsymbol{\nabla}_y\boldsymbol{\nabla}_{x\star z}\,e
+\mathcal{Q}(\boldsymbol{\nabla}_y(x\star z))-(x\leftrightarrow y)
\\
&=-\boldsymbol{\nabla}_y(-\boldsymbol{\nabla}_x\boldsymbol{\nabla}_z\,e+\mathcal{Q}(\boldsymbol{\nabla}_x\,z))
+\mathcal{Q}(\boldsymbol{\nabla}_y(x\star z))-(x\leftrightarrow y)
\\
&=-\boldsymbol{\nabla}_{[x,y]}\boldsymbol{\nabla}_z\,e
-\boldsymbol{\nabla}_y\boldsymbol{\nabla}_{\boldsymbol{\nabla}_x\,z}\,e
+\boldsymbol{\nabla}_x\boldsymbol{\nabla}_{\boldsymbol{\nabla}_y\,z}\,e
+\mathcal{Q}(\boldsymbol{\nabla}_y(x\star z)-
\boldsymbol{\nabla}_x(y\star z)
)~.
\end{split}
\end{equation}
In passing to the last line, we used the flatness of $\boldsymbol{\nabla}$. 
Moreover, we have
\begin{equation}\nonumber
\begin{split}
&\mathcal{Q}(-y\star \boldsymbol{\nabla}_x\,z
+x\star \boldsymbol{\nabla}_y\,z
-[x,y]\star z)
\\
&=\boldsymbol{\nabla}_y\boldsymbol{\nabla}_{\boldsymbol{\nabla}_x\,z}\,e
-\boldsymbol{\nabla}_x\boldsymbol{\nabla}_{\boldsymbol{\nabla}_y\,z}\,e
+\boldsymbol{\nabla}_{[x,y]}\boldsymbol{\nabla}_z\,e
+\mathcal{Q}
({\boldsymbol{\nabla}_x\boldsymbol{\nabla}_y\,z}
-{\boldsymbol{\nabla}_y\boldsymbol{\nabla}_x\,z}
-{\boldsymbol{\nabla}_{[x,y]}\,z}
)
\\
&=\boldsymbol{\nabla}_y\boldsymbol{\nabla}_{\boldsymbol{\nabla}_x\,z}\,e
-\boldsymbol{\nabla}_x\boldsymbol{\nabla}_{\boldsymbol{\nabla}_y\,z}\,e
+\boldsymbol{\nabla}_{[x,y]}\boldsymbol{\nabla}_z\,e~.
\end{split}
\end{equation}
Here we used the flatness of $\boldsymbol{\nabla}$. 
Adding the above two equations, we obtain
\begin{equation}\nonumber
\mathcal{Q}(y\star(z\star x)-(y\star z)\star x)
=
\mathcal{Q}(-\boldsymbol{\nabla}_x(y\star z)+y\star \boldsymbol{\nabla}_x\,z-(x\leftrightarrow y)+[x,y]\star z)~.
\end{equation}
Since $\mathcal{Q}$ is an isomorphism, it follows that 
the associativity of $\star$ is equivalent to ({\bf ASS1}).
(3) This follows form (1) and (2).
\end{proof}

\begin{definition}
We say that an almost Saito structure $(\boldsymbol{\nabla},\star,e)$
on $N$ is regular if $(\boldsymbol{\nabla},e)$ is regular.
\end{definition}

For a regular almost Saito structure $(\boldsymbol{\nabla},\star,e)$,
we sometimes omit $\star$, and call
$(\boldsymbol{\nabla},e)$ a regular almost Saito structure.

\begin{remark}
Let $(\boldsymbol{\nabla},\star,e)$ be an almost Saito structure 
on $N$ with parameter $r$ and a unit $E$. 
Then by ({\bf ASS3}) and \eqref{ASS2},
$$
\mathcal{Q}(E)=\boldsymbol{\nabla}_E \,e=
\boldsymbol{\nabla}_e\,E+[E,e]
=(r-1)e~.
$$
Thus,
if an almost Saito structure
$(\boldsymbol{\nabla},\star,e)$ 
has parameter $r=1$, then it is not regular.
\end{remark}

\subsection{Two-parameter family}\label{sec:two-par}
The next proposition shows that if there exists one almost Saito structure,
then there exists a two-parameter family of 
almost Saito structures. 

Let $(\boldsymbol{\nabla},\star,e)$ be an almost Saito structure on 
a manifold $N$ with parameter $r\in \mathbb{C}$
and a unit $E$.
Take $\lambda \in \mathbb{C}$ and define 
$\mathcal{I}_{\lambda}\in \mathrm{Hom}_{\mathcal{O}_N}
(\mathcal{T}_N,\mathcal{T}_N)$ by
$\mathcal{I}_{\lambda}(x)=(E-\lambda e)\star x$.
Assume that 
$$
N^{\lambda}=\{p\in N\mid 
\text{$I_{\lambda}:T_p N\to T_p N$ has rank $\dim N$}\}~
$$
is a nonempty subset of $N$. Here $I_{\lambda}$
is the endomorphism of $TN$ corresponding to $\mathcal{I}_{\lambda}$.
\begin{proposition}\label{shift}
Take $\nu\in \mathbb{C}$ and 
define a new multiplication $\star_{\lambda}$ and 
a new connection $\boldsymbol{\nabla}^{[\lambda,\nu]}$ 
on $TN^{\lambda}$ by
\begin{equation}\nonumber
\begin{split}
x\star_{\lambda} y&=\mathcal{I}_{\lambda}^{-1}(x\star y)~,
\\
\boldsymbol{\nabla}^{[\lambda,\nu]}_x\,y
&=\boldsymbol{\nabla_{x}}\,y+\nu x\star_{\lambda} y
+\lambda\boldsymbol{\nabla}_{x\star_{\lambda}y}\,e~
\quad (x,y\in \mathcal{T}_{N^{\lambda}}).
\end{split}
\end{equation}
Then $(\boldsymbol{\nabla}^{[\lambda,\nu]},\star_{\lambda},e)$ 
is an almost Saito structure  on $N^{\lambda}$
with parameter $r+\nu$  and the unit $E-\lambda e$.
\end{proposition}
\begin{proof}
In this proof, $x,y,z\in \mathcal{T}_{N^{\lambda}}$.
First we show a technical lemma.  Set
\begin{eqnarray}
\nonumber
\mathcal{C}(x,y,z;\lambda)&=&
\boldsymbol{\nabla}_x(y\star_{\lambda}z)
-y\star_{\lambda}\boldsymbol{\nabla}_x\,z
-\lambda y\star_{\lambda}\boldsymbol{\nabla}_{x\star_{\lambda} z}\,e~,
\\\nonumber
\mathcal{E}(x,y;\lambda)&=&
\boldsymbol{\nabla}_x(e\star_{\lambda} y)
-e\star_{\lambda} \boldsymbol{\nabla}_x\,y
-\boldsymbol{\nabla}_{x\star_{\lambda} y}\,e
+x\star_{\lambda}\boldsymbol{\nabla}_y\,e
+e\star_{\lambda} x\star_{\lambda} y
\\\nonumber
&+&\lambda 
x\star_{\lambda}\boldsymbol{\nabla}_{e\star_{\lambda }y}\,e
-\lambda e\star_{\lambda}\boldsymbol{\nabla}_{x\star_{\lambda}y}\,e~.
\end{eqnarray}
\begin{lemma}
\begin{eqnarray}\label{c1}
&&
\mathcal{C}(x,y,z;\lambda)-\mathcal{C}(y,x,z;\lambda)
=[x,y]\star_{\lambda}z~,
\\
&&\label{e1}
\mathcal{E}(x,y;\lambda)=0~.
\end{eqnarray}
\end{lemma}
\begin{proof} 
\eqref{c1}: 
Put $z'=\mathcal{I}_{\lambda}^{-1}(z)$.
\begin{equation}\nonumber
\begin{split}
\mathcal{C}(x,y,z;\lambda)&=
\mathcal{I}_{\lambda}^{-1}\big(
(E-e\lambda)\star \boldsymbol{\nabla}_x(y\star z')
-y\star \boldsymbol{\nabla}_x((E-\lambda e)\star z')
-\lambda y\star \boldsymbol{\nabla}_{x\star z'}\,e
\big)
\\&=\mathcal{I}_{\lambda}^{-1}\big(
\boldsymbol{\nabla}_x(y\star z')
-y\star \boldsymbol{\nabla}_x\, z'
\\&-\lambda (
e\star \boldsymbol{\nabla}_{x}(y\star z')
-y\star \boldsymbol{\nabla}_x(e\star z')
+y\star \boldsymbol{\nabla}_{x\star z'}\,e
)\big)~.
\end{split}
\end{equation}
So
\begin{equation}\nonumber
\begin{split}
\mathcal{C}(x,y,z;\lambda)&-\mathcal{C}(y,x,z;\lambda)
=\mathcal{I}_{\lambda}^{-1}\big(
\boldsymbol{\nabla}_x(y\star z')
-y\star \boldsymbol{\nabla}_x\,z'
-\boldsymbol{\nabla}_y(x\star z')
+x\star \boldsymbol{\nabla}_y\,z'
\\
&-\lambda(
e\star \boldsymbol{\nabla}_x(y\star z')
-y\star \nabla_x (e\star z')+y\star \boldsymbol{\nabla}_{x\star z'}\,e
-(x\leftrightarrow y)
)
\big)
\\
&=\mathcal{I}_{\lambda}^{-1}([x,y]\star z'-\lambda e\star[x,y]\star z')
=\mathcal{I}_{\lambda}^{-1}([x,y]\star z)=[x,y]\star_{\lambda}z~.
\end{split}
\end{equation}
In the last line, we used ({\bf ASS1}) and \eqref{ASS4}.
\\
\eqref{e1}: 
put $y'=\mathcal{I}_{\lambda}^{-1}(y)$. Then
\begin{equation}\nonumber
\begin{split}
\mathcal{E}(x,y;\lambda)
&=\mathcal{I}_{\lambda}^{-1}
\big(
(E-\lambda e)\star 
\boldsymbol{\nabla}_x(e\star y')
-e \star\boldsymbol{\nabla}_x\,(E-\lambda e)\star y'
-(E-\lambda e)\star\boldsymbol{\nabla}_{x\star y'}\,e
\\&+x\star \boldsymbol{\nabla}_{(E-\lambda e)\star y'}\,e
+e\star x\star y'
+\lambda 
x\star \boldsymbol{\nabla}_{e\star y'}\,e
-\lambda e\star \boldsymbol{\nabla}_{x\star y'}\,e
\big)
\\
&=\mathcal{I}_{\lambda}^{-1}\big(
\boldsymbol{\nabla}_x(e\star  y')
-e\star \boldsymbol{\nabla}_x\,y'
-\boldsymbol{\nabla}_{x\star  y'}\,e
+x\star \boldsymbol{\nabla}_{y'}\,e
+e\star x\star  y'
\big)
\\
&=0~,
\end{split}
\end{equation}
because of \eqref{ASS3}.
\end{proof}
The commutativity and the associativity of $\star_{\lambda}$,
and the torsion freeness of $\boldsymbol{\nabla}^{[\lambda,\nu]}$
follow from those of $\star,\boldsymbol{\nabla}$.
As for the flatness of $\boldsymbol{\nabla}^{[\lambda,\nu]}$,
\begin{equation}\nonumber
\begin{split}
&\boldsymbol{\nabla}^{[\lambda,\nu]}_x
\boldsymbol{\nabla}^{[\lambda,\nu]}_y\,z
-\boldsymbol{\nabla}^{[\lambda,\nu]}_y
\boldsymbol{\nabla}^{[\lambda,\nu]}_x\,z
\\&=\boldsymbol{\nabla}_x
\boldsymbol{\nabla}_y\,z
-\boldsymbol{\nabla}_y
\boldsymbol{\nabla}_x\,z
+
(\nu \mathrm{Id}+\lambda \mathcal{Q})
\big(\mathcal{C}(x,y,z;\lambda)-\mathcal{C}(y,x,z;\lambda)\big)
\\
&=\boldsymbol{\nabla}_{[x,y]}\,z
+(\nu \mathrm{Id}+\lambda \mathcal{Q})
([x,y]\star_{\lambda} z)
=\boldsymbol{\nabla}^{[\lambda,\mu]}_{[x,y]}\,z~.
\end{split}
\end{equation}
Here we used  ({\bf ASS4}), 
the flatness of $\boldsymbol{\nabla}$ and \eqref{c1}.
\\
Next let us check the conditions ({\bf ASS1})--({\bf ASS4}).
({\bf ASS1}) follows from \eqref{c1}:
\begin{equation}\nonumber
\begin{split}
\boldsymbol{\nabla}^{[\lambda,\nu]}_x (y\star_{\lambda} z)
-y\star_{\lambda} \boldsymbol{\nabla}^{[\lambda,\nu]}_x \,z
-(x\leftrightarrow y)
&=\mathcal{C}(x,y,z;\lambda)-\mathcal{C}(y,x,z;\lambda)
=[x,y]\star_{\lambda} z~.
\end{split}
\end{equation}
({\bf ASS2}): put $x'=\mathcal{I}_{\lambda}^{-1}(x)$
and $y'=\mathcal{I}_{\lambda}^{-1}(y)$.
Then by ({\bf ASS2}) and \eqref{ASS2},
\begin{equation}\nonumber
\begin{split}
[e,x]&=[e,(E-\lambda e)\star x']=
(E-\lambda e)\star [e,x']+\lambda e\star e\star x'~.
\end{split}
\end{equation}
Therefore
\begin{equation}\nonumber
\begin{split}
&[e,x\star_{\lambda} y]-[e,x]\star_{\lambda} y-x\star_{\lambda}[e,y]
+e\star_{\lambda} x\star_{\lambda} y
\\&=
(E-\lambda e)\star
\big([e,x'\star y']-[e,x']\star y'-x'\star[e,y']+e\star x'\star y'\big)
=0~.
\end{split}
\end{equation}
({\bf ASS3}):  Given that 
$E-\lambda e$ is a unit of $\star_{\lambda}$,
we have
\begin{equation}\nonumber
\begin{split}
\boldsymbol{\nabla}^{[\lambda,\nu]}_x(E-\lambda e)
=\boldsymbol{\nabla}_x (E-\lambda e)
+\nu x
+\lambda \boldsymbol{\nabla}_x \,e=
(r+\nu)x~.
\end{split}
\end{equation}
({\bf ASS4}): 
\begin{equation}\nonumber
\begin{split}
&\boldsymbol{\nabla}^{[\lambda,\nu]}_x
 \boldsymbol{\nabla}^{[\lambda,\nu]}_y \,e
-\boldsymbol{\nabla}^{[\lambda,\nu]}_{
\boldsymbol{\nabla}^{[\lambda,\nu]}_x y}\,e+
\boldsymbol{\nabla}^{[\lambda,\nu]}_{x\star_{\lambda} y} \,e
=
\boldsymbol{\nabla}_x\boldsymbol{\nabla}_y \,e
-\boldsymbol{\nabla}_{\boldsymbol{\nabla}_x y}\,e+
\boldsymbol{\nabla}_{x\star y}\, e 
\\&+(\nu \mathrm{Id} +\lambda \mathcal{Q})
\big(\mathcal{E}(x,y;\lambda)
\big)
+\mathcal{Q}(
x\star_{\lambda} y-x\star y-\lambda x\star y\star_{\lambda}e
)
\\
&=0~.
\end{split}
\end{equation}
This completes the proof of Proposition \ref{shift}.
\end{proof}

The next lemma shows how local flat coordinates 
of $\boldsymbol{\nabla}$
and the $\boldsymbol{\nabla}^{[0,-1]}$
are related.
\begin{lemma}\label{shift2}
Let $(\boldsymbol{\nabla},\star,e)$ be an almost Saito structure
with parameter $r\in \mathbb{C}$
on a manifold $N$.
Let $u^1,\ldots, u^n$ be local flat coordinates with respect to
$\boldsymbol{\nabla}$ and let $e^i=e(u^i)$. 
If $(\boldsymbol{\nabla},e)$ is regular,
$e^1,\ldots, e^n$  are flat local coordinates 
with respect to 
$\boldsymbol{\nabla}^{[0,-1]}$.
\end{lemma}
\begin{proof}
First let us show that $e^1,\ldots,e^n$ are local coordinates.
By the definition of $e^i$,
the vector field $e$ is written as
$$
e=\sum_{i=1}^n e^i \partial_{u^i}~.
$$
So the matrix representation of $\mathcal{Q}=\boldsymbol{\nabla}\,e$
with respect to $(\partial_{u^1},\ldots,\partial_{u^n})$
is given by
$$
\left(\frac{\partial e^j }{\partial u^k}\right)_{j,k}~.
$$
Then by the regularity of $(\boldsymbol{\nabla},e)$,
this matrix is invertible. This implies that  $e^1,\ldots,e^n$
are local coordinates.

The above matrix representation also implies that 
$
\mathcal{Q}(\partial_{ e^i})={\partial}_{ u^i}~.
$
Therefore by ({\bf ASS4}),
\begin{equation}\nonumber
\boldsymbol{\nabla}_x\,\partial_{ e^i}
=\mathcal{Q}^{-1}(\boldsymbol{\nabla}_x\mathcal{Q}(\partial_{ e^i}))
+x\star \partial_{ e^i}
=\mathcal{Q}^{-1}(\boldsymbol{\nabla}_x \,\partial_{u^i})+
x\star \partial_{e^i}=x\star \partial_{e^i}~\quad
(x\in \mathcal{T}_N)~.
\end{equation}
Thus $\boldsymbol{\nabla}^{[0,-1]}_x \partial_{e^i}=0$.
\end{proof}

\section{Almost duality for the Saito structure}
\label{sec:duality}
In \S \ref{sec:ss2ass}, we show that
one can construct a two-parameter family of almost Saito structures
from a given  Saito structure.
In \S \ref{sec:ass2ss}, we also  show that
one can construct a Saito structure if 
given an almost Saito structure.
In \S \ref{sec:AD-SS}, we explain that these constructions 
can be seen as inverse operations.
We call this phenomenon the almost duality for the Saito structure.

In fact, a Saito structure is always accompanied with 
a one-parameter family (Remark \ref{E-shift}) while
an almost Saito structure is accompanied with 
a two-parameter family (\S \ref{sec:two-par}).
How they correspond via these constructions
is mentioned in Remark \ref{correspondence}.

\subsection{From a  Saito structure to an almost Saito structure}
\label{sec:ss2ass}
Let $(\nabla,\ast,E)$ be a Saito structure
on a manifold $M$
and let $e$ be its unit.
Take $\lambda\in \mathbb{C}$ and 
define $\mathcal{U}_{\lambda}\in \mathrm{Hom}_{\mathcal{O}_M}(\mathcal{T}_M,\mathcal{T}_M)$ by
$
\mathcal{U}_{\lambda}(x)=(E-\lambda e)\ast x.
$
Let 
\begin{equation}\nonumber
M_{\lambda}
=\{ p\in M\mid \text{ ${U}_{\lambda}:T_pM\to T_pM$ 
has rank $\dim M$}\}~.
\end{equation}
Here $U_{\lambda}$ is the endomorphism of $TM$ corresponding
to $\mathcal{U}_{\lambda}$.
We put the assumption that $M_{\lambda}\subset M$ is not empty.

\begin{proposition}\label{ss2ass}
Choose  $r\in \mathbb{C}$ and 
define a multiplication $\star_{\lambda}$ and a connection $\boldsymbol{\nabla}^{(\lambda,r)}$ on $TM_{\lambda}$ by
\begin{equation}\label{eq:ss2ass}
\begin{split}
&x\star_{\lambda} y
=\mathcal{U}_{\lambda}^{-1}(x\ast y)~,
\\
&\boldsymbol{\nabla}^{(\lambda,r)}_x \,y=\nabla_x\,y
+r \,x\star_{\lambda} y-\nabla_{x\star_{\lambda} y} \,E~.
\end{split}
\end{equation}
Then \\
(1) $\star_{\lambda}$ is  commutative and associative with 
the unit $E-\lambda e$.
\\
(2) $\boldsymbol{\nabla}^{(\lambda,r)}$ is torsion free and flat.
\\
(3) 
$(\boldsymbol{\nabla}^{(\lambda,r)},\star_{\lambda}, e)$ is an almost Saito structure on $M_{\lambda}$
with parameter $r$.
\\
(4) Define $\mathcal{P}_{\lambda}\in \mathrm{Hom}_{\mathcal{O}_{M_{\lambda}}}(
\mathcal{T}_{M_{\lambda}},\mathcal{T}_{M_{\lambda}}
)$ by
$\mathcal{P}_{\lambda}(x)=
e\star_{\lambda} x.$
Then 
for $x,y\in \mathcal{T}_{M_{\lambda}}$,
\begin{equation}\nonumber
 \mathcal{P}_{\lambda}^{-1}(x\star_{\lambda}y)
=x\ast y~,
\quad \boldsymbol{\nabla}^{(\lambda,r)}_x\,y
 -\boldsymbol{\nabla}^{(\lambda,r)}_{x\ast y}\,e=
 \nabla_x\,y~.
\end{equation}
\end{proposition}
\begin{remark}
The appearance of the two parameters in 
Proposition \ref{ss2ass} is due to the existence
of the two-parameter family explained in \S \ref{sec:two-par}.
If one considers  the almost Saito structure obtained from
$(\boldsymbol{\nabla}^{(0,0)},\star_0,e)$ by
twisting by $\lambda$ and shifting by $\nu$
as in Proposition \ref{shift},
then it is  nothing but
$(\boldsymbol{\nabla}^{(\lambda,\nu)},\star_{\lambda},e)$
in Proposition \ref{ss2ass}.
(However, in this way, 
$(\boldsymbol{\nabla}^{(\lambda,\nu)},\star_{\lambda},e)$ is given on
$M_0\cap (M_0)^{\lambda}$ which may be smaller than 
$M_{\lambda}$.)
\end{remark}

\begin{remark}
The $2$-parameter family of connections $\boldsymbol{\nabla}^{(\lambda, r)}$
had been considered before in the case of the Frobenius manifolds
under the name of the second structure connection \cite[\S 9.2]{Hertling2002}.
The flatness and torsion freeness are proved in Theorem 9.4 of {\it loc.cit.}
and the condition ({\bf ASS3}) can be derived easily from formula (9.12) in {\it loc.cit.}.  
\end{remark}

To prove Proposition \ref{ss2ass},
we first prove a technical lemma.  Set
\begin{eqnarray}
\nonumber
\mathcal{C}_2(x,y,z;\lambda)&=&
{\nabla}_x(y\star_{\lambda}z)
-y\star_{\lambda}{\nabla}_x\,z
+y\star_{\lambda}{\nabla}_{x\star_{\lambda} z}\,E~,
\\\nonumber
\mathcal{E}_2(x,y;\lambda)&=&
{\nabla}_x(e\star_{\lambda} y)
-e\star_{\lambda} {\nabla}_x\,y
-x\star_{\lambda}{\nabla}_{e\star_{\lambda}y}\,E
+e\star_{\lambda} \boldsymbol{\nabla}_{x\star_{\lambda}y}\,E
+e\star_{\lambda} x\star_{\lambda}y~.
\end{eqnarray}
\begin{lemma}
\begin{eqnarray}\label{c2}
&&
\mathcal{C}_2(x,y,z;\lambda)-\mathcal{C}_2(y,x,z;\lambda)
=[x,y]\star_{\lambda}z~,
\\
&&\label{e2}
\mathcal{E}_2(x,y;\lambda)=0~.
\end{eqnarray}
\end{lemma}
\begin{proof}
\eqref{c2}: 
put $z'=\mathcal{U}_{\lambda}^{-1}(z)$. Then
\begin{equation}\begin{split}\nonumber
&\mathcal{C}_2(x,y,z;\lambda)-\mathcal{C}_2(y,x,z;\lambda)
\\&=
\mathcal{U}_{\lambda}^{-1}\big(
(E-\lambda e)\ast \nabla_x(y\ast z')
-y\ast \nabla_x((E-\lambda e)\ast z')+y\ast \nabla_{x\ast z'}\,E
\big)-(x\leftrightarrow y)
\\
&=
\mathcal{U}_{\lambda}^{-1}
\big(
E\ast \nabla_x(y\ast z')-y\ast \nabla_x(E\ast z')+
y\ast \nabla_{x\ast z'}\,E
-\lambda(\nabla_x(y\ast z')-y\ast \nabla_x \,z')
\big)-(x\leftrightarrow y)
\\
&=\mathcal{U}_{\lambda}^{-1} (E\ast[x,y]\ast z'-\lambda[x,y]\ast z')
=[x,y]\star_{\lambda} z~.
\end{split}
\end{equation}
In passing to the last line, we used ({\bf SS1}) and \eqref{SS4}.
\\
\eqref{e2}: put $y'=\mathcal{U}_{\lambda}^{-1}(y)$. Then
\begin{equation}\nonumber
\begin{split}
\mathcal{E}_2(x,y;\lambda)
&=\mathcal{U}_{\lambda}^{-1}
\big(
(E-\lambda e)\ast \nabla_x\, y'-x\ast \nabla_{y'}\,E
-\nabla_x(E-\lambda e)\ast y'+\nabla_{x\ast y'}\,E
+x\ast y'
\big)
\\
&=\mathcal{U}_{\lambda}^{-1}
(E\ast \nabla_x \,y'-x\ast \nabla_{y'}\,E-\nabla_x(E\ast y')
+\nabla_{x\ast y'}\,E+x\ast y')
=0~,
\end{split}
\end{equation}
due to \eqref{SS3}.
\end{proof}

\begin{proof}[Proof of Proposition \ref{ss2ass}]
(1)  The commutativity and associativity of $\ast$ imply the
same properties for $\star_{\lambda}$. 
By definition of $\star_{\lambda}$,
it is clear that $E-\lambda e$ is its unit.
\\
(2)  
The torsion freeness of $\nabla$ 
and the commutativity of $\star_{\lambda}$
imply the torsion freeness of  $\boldsymbol{\nabla}^{(\lambda,r)}$.
As for the flatness, we have
\begin{equation}\nonumber
\begin{split}
&\boldsymbol{\nabla}^{(\lambda,r)}_x
\boldsymbol{\nabla}^{(\lambda,r)}_y\,z
-\boldsymbol{\nabla}^{(\lambda,r)}_y
\boldsymbol{\nabla}^{(\lambda,r)}_x\,z
\\&=
{\nabla}_x{\nabla}_y\,z-{\nabla}_y{\nabla}_x\,z
+(r \mathrm{Id}-\mathcal{W})\big(
\mathcal{C}_2(x,y,z;\lambda)-\mathcal{C}_2(y,x,z;\lambda)
\big)
\\
&=
\nabla_{[x,y]}\,z+(r \mathrm{Id}-\mathcal{W})([x,y]\star_{\lambda}z)
=\boldsymbol{\nabla}^{(\lambda,r)}_{[x,y]}\,z
~.
\end{split}
\end{equation}
Here $\mathcal{W}(x)=\nabla_x E$.
We used ({\bf SS4}), the flatness of $\nabla$
and \eqref{c2}.
\\
(3) Let us check the conditions ({\bf ASS1})--({\bf ASS4}).
\\
Condition ({\bf ASS1}): 
\begin{equation}\nonumber
\begin{split}
&\boldsymbol{\nabla}^{(\lambda,r)}_x (y\star_{\lambda} z)
-y\star_{\lambda} \boldsymbol{\nabla}^{(\lambda,r)}_x \,z
-\boldsymbol{\nabla}^{(\lambda,r)}_y(x\star_{\lambda} z)
+x\star_{\lambda} \boldsymbol{\nabla}^{(\lambda,r)}_y \,z\\
&=\mathcal{C}_2(x,y,z;\lambda)-\mathcal{C}_2(y,x,z;\lambda)
\\
&=[x,y]\star_{\lambda}z~.
\end{split}
\end{equation}
Condition ({\bf ASS2}): 
put $x'=\mathcal{U}_{\lambda}^{-1}(x)$ and
$y'=\mathcal{U}_{\lambda}^{-1}(y)$. 
By \eqref{SS1},
\begin{equation}\nonumber
[e,x]=[e,(E-\lambda e)\ast x']=(E-\lambda e)\ast [e,x']
+[e,E]\ast x'~.
\end{equation}
Therefore
\begin{equation}\nonumber
\begin{split}
&[e,x\star_{\lambda}y]-[e,x]\star_{\lambda}y-x\star_{\lambda}[e,y]
+e\star_{\lambda}x\star_{\lambda}y
\\&=
(E-\lambda e)\ast \big(
[e,x'\ast y']-[e,x']\ast y'-x'\ast[e,y']\big)=0~.
\end{split}
\end{equation}
In the last line, we used \eqref{SS1} again.
\\
Condition ({\bf ASS3}): 
given that $E-\lambda e$ is a unit of $\star_{\lambda}$,
$$
\boldsymbol{\nabla}^{(\lambda,r)}_x\,(E-\lambda e)
=\nabla_x\,E
+rx-\nabla_{x}\,E=rx~.
$$
Condition ({\bf ASS4}): 
\begin{equation}\nonumber
\begin{split}
\boldsymbol{\nabla}^{(\lambda,r)}_x
\boldsymbol{\nabla}^{(\lambda,r)}_y\,e
-
\boldsymbol{\nabla}^{(\lambda,r)}_{
\boldsymbol{\nabla}^{(\lambda,r)}_x\,y} e
+
\boldsymbol{\nabla}^{(\lambda,r)}_{x\star_{\lambda}y}\,e
=
(r \mathrm{Id}-\mathcal{W})(\mathcal{E}_2(x,y;\lambda))
\stackrel{\eqref{e2}}{=}0~.
\end{split}
\end{equation}
\\
(4) Multiplication: notice that 
$\mathcal{P}_{\lambda}=\mathcal{U}_{\lambda}^{-1}$.
Therefore
$$
x\star_{\lambda} \mathcal{P}_{\lambda}^{-1}(y)=x\ast 
(\mathcal{U}_{\lambda}^{-1}\circ\mathcal{P}_{\lambda}^{-1})(y)
=x\ast y~.
$$
Connection:  since 
$x\star_{\lambda} y=x\ast y\star_{\lambda} e$, it is easy to see that
\begin{equation}\nonumber
\begin{split}
\boldsymbol{\nabla}^{(\lambda,r)}_x\,y
-\boldsymbol{\nabla}^{(\lambda,r)}_{x\ast y}\,e
&=\nabla_x\,y+r x\star_{\lambda} y-\nabla_{x\star_{\lambda} y}\,E
-\nabla_{x\ast y}\,e-r x\ast y\star_{\lambda} e+\nabla_{x\ast y\star_{\lambda} e}\,E
\\
&=\nabla_{x}\,y~.
\end{split}
\end{equation}
This completes the proof of Proposition \ref{ss2ass}.
\end{proof}

\begin{definition}\label{def-dualASS}
We call 
$(\boldsymbol{\nabla}^{(0,r)},\star_0,e)$ in Proposition \ref{ss2ass}
 the dual almost Saito structure
  of $(\nabla,\ast ,E)$ with parameter $r\in \mathbb{C}$.
\end{definition}
\begin{remark}
It would be simpler if we define
$(\boldsymbol{\nabla}^{(0,0)},\star_0,e)$
as the dual almost Saito structure of $(\nabla,\ast ,E)$.
However, we adopt Definition \ref{def-dualASS}
for the sake of application to the complex reflection groups.
\end{remark}
\subsection{From an almost  Saito structure to a Saito structure}
\label{sec:ass2ss}
Let $(\boldsymbol{\nabla},\star,e)$ be an almost Saito structure
with parameter $r\in \mathbb{C}$ on a manifold $N$
and let $E$ be its unit.
Define an endomorphism 
$\mathcal{P}\in \mathrm{Hom}_{\mathcal{O}_N} 
(\mathcal{T}_N,\mathcal{ T}_N)$
by 
$\mathcal{P}(x)=e\star x$.
Let
\begin{equation}\nonumber
N_0=\{p\in N\mid \text{$P:T_pN\to T_pN$ 
has rank $\dim N$}\}~.
\end{equation}
Here $P$ is the endomorphism of $TN$ corresponding to $\mathcal{P}$.
We put the assumption that $N_0$ is not empty.
\begin{proposition}\label{ass2ss}
Define a multiplication $\ast$ and a connection $\nabla$ on $TN_0$ by
\begin{equation}\label{eq:ass2ss}
\begin{split}
&x\ast y=x\star \mathcal{P}^{-1}(y)=\mathcal{P}^{-1}(x)\star y~,
\\
&\nabla_x\, y=\boldsymbol{\nabla}_x\, y-\boldsymbol{\nabla}_{x\ast y} \,e~.
\end{split}
\end{equation}
Then
\\
(1) $\ast $ is commutative and associative with the unit $e$.
\\
(2) $\nabla$ is torsion free and flat.
\\
(3)
$(\nabla,\ast,E)$ is a Saito structure on $N_0$.
\\
(4) Define $\mathcal{U}\in \mathrm{Hom}_{\mathcal{O}_{N_0}}(
\mathcal{T}_{N_0},\mathcal{T}_{N_0})$ by $\mathcal{U}(x)=E\ast x$.
Then 
for $x,y\in \mathcal{T}_{N_0}$,
\begin{equation}\nonumber
\begin{split}
&x\ast \mathcal{U}^{-1}(y)=\mathcal{U}^{-1}(x)\ast y=x\star y~,\\ 
&\nabla_x \,y+r x\star y-\nabla_{x\star y} E=\boldsymbol{\nabla}_x\,y~.
\end{split}
\end{equation}
\end{proposition}

Before proving Proposition \ref{ass2ss},
we show the next auxiliary lemma.
Let $x,y,z\in \mathcal{T}_{N_0}$.
We set
\begin{eqnarray}\nonumber
\mathcal{C}_3(x,y,z)&=&
\boldsymbol{\nabla}_x(y\ast z)
-y\ast \boldsymbol{\nabla}_x\,z
+y\ast\boldsymbol{\nabla}_{x\ast z}\,e~,
\\\nonumber
\mathcal{E}_3(x,y)&=&
\boldsymbol{\nabla}_x(E\ast y)
-E\ast \boldsymbol{\nabla}_x\,y
+E\ast \boldsymbol{\nabla}_{x\ast y}\,e
-x\ast \boldsymbol{\nabla}_{E\ast y}\,e
-x\ast y~.
\end{eqnarray}
\begin{lemma}
\begin{eqnarray}
\label{c3}
\mathcal{C}_3(x,y,z)-\mathcal{C}_3(y,x,z)&=&[x,y]\ast z~,
\\
\label{e3}
\mathcal{E}_3(x,y)&=&0~.
\end{eqnarray}
\end{lemma}
\begin{proof}
\eqref{e3}: put $z'=\mathcal{P}^{-1}(z)$. Then
\begin{equation}\nonumber
\begin{split}
&\mathcal{C}_3(x,y,z)-\mathcal{C}_3(y,x,z)
\\
&=\mathcal{P}^{-1}\big(
e\star \boldsymbol{\nabla}_x(y\star z')
-y\star \boldsymbol{\nabla}_x(e\star z')
+y\star\boldsymbol{\nabla}_{x\star z'}\,e
\big)-(x\leftrightarrow y)
\\
&\stackrel{\eqref{ASS4}}{=}
\mathcal{P}^{-1}(e\star [x,y]\star z')=[x,y]\ast z
\end{split}
\end{equation}
\eqref{e3}: put $y'=\mathcal{P}^{-1}(y)$. Then, given that $E$
is a unit of $\star$, 
\begin{equation}\nonumber
\begin{split}
\mathcal{E}_3(x,y)
&=
\mathcal{P}^{-1}\big(e\star
\boldsymbol{\nabla}_x\, y'
- \boldsymbol{\nabla}_x (e\star y')
+\boldsymbol{\nabla}_{x\star y'}\,e
-x\star \boldsymbol{\nabla}_{y'}\,e
-e\star x\star y'
\big)
\\
&\stackrel{\eqref{ASS3}}{=}0~.
\end{split}
\end{equation}
\end{proof}

\begin{proof}[Proof of Proposition \ref{ass2ss}]
(1) The commutativity and associativity of $\star$ imply the same 
property for $\ast$. By the definition, it is clear that 
$e$ is a unit of $\ast$.
\\
(2) The torsion freeness of $\boldsymbol{\nabla}$ 
and the commutativity of $\ast$ 
imply the torsion freeness of $\nabla$.
The flatness is shown as follows.
\begin{equation}\nonumber
\begin{split}
\nabla_x\nabla_y\,z-\nabla_y\nabla_x\,z
&=
\boldsymbol{\nabla}_x\boldsymbol{\nabla}_y\,z
-\boldsymbol{\nabla}_y\boldsymbol{\nabla}_x\,z
-\mathcal{Q}\big(\mathcal{C}_3(x,y,z)-\mathcal{C}_3(y,x,z)
\big)
\\
&=\boldsymbol{\nabla}_{[x,y]}\,z
-\mathcal{Q}([x,y]\ast z)
=\nabla_{[x,y]}\,z~.
\end{split}
\end{equation}
Here $\mathcal{Q}(x)=\boldsymbol{\nabla}_x\,e$.
We used ({\bf ASS4}), the flatness of $\boldsymbol{\nabla}$
and \eqref{c3}. 
\\
(3)  Let us check the conditions ({\bf SS1})--({\bf SS4}).
\\
Condition ({\bf SS1}): 
\begin{equation}\nonumber
\begin{split}
&\nabla_x (y\ast z)-y\ast \nabla_x \,z-\nabla_y(x\ast z)+x\ast \nabla_y \,z-[x,y]\ast z
\\
&=\mathcal{C}_3(x,y,z)-\mathcal{C}_3(y,x,z)
-[x,y]\ast z\stackrel{\eqref{c3}}{=}0~.
\end{split}
\end{equation}
Condition ({\bf SS2}):  
putting $x'=\mathcal{P}^{-1}(x)$ and $y'=\mathcal{P}^{-1}(y)$,
we have
\begin{equation}\nonumber
\begin{split}
&[E,x\ast y]-[E,x]\ast y-x\ast[E,y]-x\ast y
\\&=
[E,e\star x'\star y']-[E,e\star x']\star y'-x'\star [E,e\star y']-e\star x'\star y'
\\
&\stackrel{\eqref{ASS2}}{=}
e\star \big(
[E,x'\star y']-[E,x']\star y'-x'\star [E,y']
\big)
\stackrel{\eqref{ASS1}}{=}0 ~.
\end{split}
\end{equation}
Condition ({\bf SS3}): given that $e$ is a unit of $\ast$, 
$$
\nabla_x\,e=\boldsymbol{\nabla}_x\,e-\boldsymbol{\nabla}_{x\ast e}\,e
=0~.
$$
Condition ({\bf SS4}):
\begin{equation}\nonumber
\begin{split}
\nabla_x\nabla_y\,E-\nabla_{\nabla_x \,y}\stackrel{({\bf ASS4})}{=}
-\mathcal{Q}(\mathcal{E}_3(x,y))
\stackrel{\eqref{e3}}{=}0~.
\end{split}
\end{equation}
\\
(4) Multiplication: since 
$\mathcal{U}=\mathcal{P}^{-1}$, 
$$
x\ast \mathcal{U}^{-1}(y)=x\star
( \mathcal{P}^{-1}\circ\mathcal{U}^{-1})(y)=x\star y~.
$$
Connection: since $x\ast y=
x \star \mathcal{P}^{-1}(y)=x\star (E\ast y)=x\star y\ast E$,
it is easy to see that
\begin{equation}\nonumber
\begin{split}
\nabla_x\,y+r x\star y-\nabla_{x\star y}\,E 
&=\boldsymbol{\nabla}_x\,y
-\boldsymbol{\nabla}_{x\ast y}\,e
+rx\star y-
\boldsymbol{\nabla}_{x\star y}\,E+\boldsymbol{\nabla}_{x\star y\ast E}\,e
\\&=\boldsymbol{\nabla}_x\,y+rx\star y-
\boldsymbol{\nabla}_{x\star y}\,E
\stackrel{({\bf ASS3})}{=}
\boldsymbol{\nabla}_x\,y~.
\end{split}
\end{equation}
This completes the proof of Proposition \ref{ass2ss}.
\end{proof}

\begin{definition}\label{def-dualSS}
We call $(\nabla,\ast,E)$ in Proposition \ref{ass2ss} the dual Saito structure of $(\boldsymbol{\nabla},\star,e)$.
\end{definition}

\begin{remark}\label{correspondence}
Let $(\boldsymbol{\nabla},\star,e)$ be an almost Saito structure
and denote
$(\nabla,\ast,E)$ its dual Saito structure.
If one considers a member of the two-parameter family of  almost Saito structures (\S \ref{sec:two-par}),
$(\boldsymbol{\nabla}^{[\lambda,\nu]},\star_{\lambda},e)$,
then its dual Saito structure is given by
$(\nabla,\ast, E-\lambda e)$.

Similarly,
for the almost Saito structure 
$(\boldsymbol{\nabla}^{(\lambda,r)},\star_{\lambda}, e)$
constructed from a Saito structure 
 $(\nabla,\ast,E)$ as in Proposition \ref{ss2ass},
its dual Saito structure is $(\nabla,\ast,E-\lambda e)$.
\end{remark}

\subsection{Almost duality for the Saito structure}
\label{sec:AD-SS}
Propositions \ref{ss2ass}-(4) and \ref{ass2ss}-(4)
imply the duality between the Saito structure 
and the almost Saito structure.
The results in \S \ref{sec:ss2ass} and \S \ref{sec:ass2ss} 
can be summarized as follows.

\begin{theorem}\label{thm:duality} 
(i) Let $(\nabla,\ast,E)$ be a Saito structure with a unit $e$
and let $(\boldsymbol{\nabla}^{(0,r)},\star_{0}, e)$ be
its dual almost Saito structure with parameter $r\in \mathbb{C}$. 
Then  
the dual Saito structure of $(\boldsymbol{\nabla}^{(0,r)},\star_{0}, e)$
is $(\nabla,\ast,E)$.\\
(ii) Let $(\boldsymbol{\nabla},\star,e)$ be an almost Saito structure with 
parameter $r\in \mathbb{C}$ and  a unit $E$, 
and denote $(\nabla,\ast,E )$  its dual Saito structure.
Then the dual almost Saito structure of $(\nabla,\ast,E )$ with
parameter $r$ is $(\boldsymbol{\nabla},\star,e)$.
\end{theorem}

\begin{remark}
In \S \ref{sec:SS}, \S \ref{subsec:ASS},  we defined 
the equivalence relations for the Saito structure 
and the almost Saito structure. 
Notice that the constructions in \S \ref{sec:ss2ass}, 
\S \ref{sec:ass2ss}
preserve those equivalence relations.
\end{remark}

\section{On the almost duality of the Frobenius manifold
 and on the bi-flat $F$-manifold}\label{relationships}
\subsection{Relationship to the almost duality of 
Frobenius manifolds \cite{Dubrovin2004}}
\label{relationship-1}
A Frobenius structure \cite{Dubrovin1993} on a manifold $M$
of charge $D\in \mathbb{C}$ is 
a Saito structure $(\nabla,\ast ,E)$ on $M$
together with a nondegenerate symmetric bilinear form $\eta$ on $TM$
satisfying
\begin{eqnarray}
\label{f1}
&&x(\eta(y,z))=\eta(\nabla_x\,y,z)+\eta(y,\nabla_x\,z)\quad 
(x,y,z\in \mathcal{T}_M)~,
\\\label{f2}
&&
\eta(x\ast y,z)=\eta(x,y\ast z) \quad (x,y,z\in \mathcal{T}_M)~,
\\
\label{f3}
&&E\eta(x,y)-\eta([E,x],y)-\eta(x,[E,y])=(2-D)\eta(x,y)\quad (x,y\in \mathcal{T}_M)~.
\end{eqnarray}

An almost Frobenius manifold \cite[\S 3]{Dubrovin2004} of charge
$D\in \mathbb{C}$ is 
an almost Saito structure $(\boldsymbol{\nabla},\star ,e)$ 
with parameter $r=\frac{1-D}{2}$ on a manifold $N$
together with a nondegenerate symmetric bilinear form $g$ on $TN$
satisfying 
\begin{eqnarray}
\label{af1}
&&
x(g(y,z))=g(\boldsymbol{\nabla}_x\,y,z)+g(y,\boldsymbol{\nabla}_x\,z)\quad 
(x,y,z\in \mathcal{T}_N)~,
\\\label{af2}
&&g(x\star y,z)=g(x,y\star z)\quad (x,y,z\in \mathcal{T}_N)~,
\\
\label{af3}
&&e \,g(x,y)-g([e,x],y)-g(x,[e,y])+g(e\star x,y)=0\quad (x,y\in \mathcal{T}_N)~.
\end{eqnarray}
Here we stated the definition in a slightly different way
from Dubrovin's.
His axiom {\bf AFM3} is equivalent to ({\bf ASS4}) (see Lemma \ref{shift2})
and axiom {\bf AFM2} is equivalent to ({\bf ASS1})--({\bf ASS3}).
(Although he excluded the case $D=1$
in his definition, it is not necessary if stated in this way.)

Notice that in the definition of the Frobenius structure,
$\nabla$ is the Levi--Civita connection of $\eta$
and that in the definition of the almost Frobenius manifold,
$\boldsymbol{\nabla}$ is the Levi--Civita connection of $g$.

Dubrovin showed the following duality between 
the Frobenius structure and the almost Frobenius manifold
\cite{Dubrovin2004}:
if $(\eta,\ast,E)$ is a Frobenius structure on $M$ of charge $D$,
then 
\begin{equation}\nonumber
g(x,y)=\eta(x,\mathcal{U}_0^{-1}(y))~,\quad
x\star y=x\ast \mathcal{U}_0^{-1}(y)~
\end{equation}
and the unit  $e$ of $\ast$ make an almost Frobenius  structure 
of charge $D$ on $M_0$.
Here $\mathcal{U}_0=E\ast $ and $M_0\subset M$ are 
the same as  \S \ref{sec:ss2ass}. 
$g$ is called the intersection form of $(\eta,\ast,E)$.
Dubrovin gave a formula for the Levi--Civita connection of $g$
in \cite[Proposition 2.3]{Dubrovin2004}.
The connection $\boldsymbol{\nabla}^{(\lambda,r)}$ (with $\lambda=0$) given in \S \ref{ss2ass} is nothing but 
this Levi--Civita connection (see 
\eqref{newconnection1}).
He also showed that if $(g,\star,e)$ is an almost  Frobenius manifold structure on $N$ of charge $D$,
then 
\begin{equation}\nonumber
\eta(x,y)=g(x,\mathcal{P}^{-1}(y))~,\quad
x\ast y=x\star \mathcal{P}^{-1}(y)~
\end{equation}
and the unit $E$ of $\star$ make a Frobenius structure on $N_0\subset N$
of charge $D$.
Here $\mathcal{P}=e\star $ and $N_0\subset N$ are the same as 
\S \ref{sec:ass2ss}.
It is not difficult to see that 
the almost duality of the Frobenius manifold
implies the almost duality between
the underlying  Saito structure and the almost Saito structure
in the sense of 
Theorem \ref{thm:duality}.

Dubrovin also noted that 
given a Frobenius structure $(\eta,\ast,E)$ of charge $D$,
then the metric $\eta$ and the intersection form $g$ 
form a flat pencil of metrics $g^*-\lambda \eta^*$ on the cotangent bundle $T^*M_{\lambda}$ ($\lambda\in \mathbb{C}$) 
\cite[Remark 1 in \S 2]{Dubrovin2004}.
(Here ${}^*$ means taking the dual between the 
tangent bundle and the cotangent bundle.)
In \S \ref{ss2ass}, we give a construction of 
a two-parameter family of 
almost Saito structures from a given Saito structure. 
This is
nothing but a generalization of Dubrovin's flat pencil of metrics:
in our notation,
the Levi--Civita connection of the metric $(g^*-\lambda \eta^*)^*$
on $TM_{\lambda}$ is given by
$
\boldsymbol{\nabla}^{(\lambda,r)}
$ in Proposition \ref{ss2ass}, 
with  $r=\frac{1-D}{2}$.

\subsection{Relationship to the bi-flat $F$-manifold}
\label{relationship-2}
In \cite{Arsie-Lorenzoni2013}, 
Arsie and Lorenzoni introduced
the notion of bi-flat $F$-manifold, 
generalizing Dubrovin's flat pencils of  metrics.
\begin{definition}
A bi-flat $F$-manifold 
$(M,\nabla,\boldsymbol{\nabla},\ast,\star,e,E)$
consists of a manifold $M$
and 
\begin{itemize}
\item  torsion-free flat connections $\nabla,\boldsymbol{\nabla}$ on
$TM$,
\item associative commutative multiplications
$\ast$ and $\star$ on $TM$
with the units $e,E\in \Gamma(M,\mathcal{T}_M)$,
\end{itemize}
satisfying 
({\bf SS1}),({\bf SS2}),({\bf SS3}),({\bf ASS1})
and ({\bf ASS3}) with $r=0$, together with the following two conditions.
\begin{itemize}
\item $\mathcal{U}\in \mathrm{Hom}_{\mathcal{O}_M}
(\mathcal{T}_M,\mathcal{T}_M)$ given by
$\mathcal{U}(x)=E\ast x$ is an isomorphism and 
$
E\ast x\star y=x\ast y~.
$
\item ``$\nabla$ and $\boldsymbol{\nabla}$ are almost hydrodynamically
equivalent'', i.e.
\begin{equation}\nonumber
\boldsymbol{\nabla}_x (y\ast z)
-\boldsymbol{\nabla}_y(x\ast z)
=\nabla_x (y\ast z)-\nabla_y(x\ast z)\quad (x,y, z\in \mathcal{T}_M)~.
\end{equation}
\end{itemize}
\end{definition}

The bi-flat $F$-manifold and 
the almost duality for Saito structure are equivalent notions.
Precise statements are given in the next two lemmas.  
\begin{lemma} \label{biflat2ss}
Let $(M,\nabla,\boldsymbol{\nabla},\ast,\star,e,E)$
be a bi-flat $F$-manifold. 
Then
$(\nabla,\ast,E)$ is a Saito structure on $M$
and 
\begin{equation}\label{b1}
\boldsymbol{\nabla}_x\,y=
\nabla_x\,y-\nabla_{x\star y}\,E~.
\end{equation}
\end{lemma}
In other words, $(\nabla,\ast,E)$ is a Saito structure on $M$
and $(\boldsymbol{\nabla},\star,e)$ is the dual almost Saito structure
on $M$ with parameter $r=0$.
\begin{proof}
We write $E\ast x=\mathcal{U}(x)$ 
and $\nabla_x \,E=\mathcal{W}(x)$ $(x\in \mathcal{T}_M)$.
Notice that the equations 
\eqref{SS1},\eqref{SS2},\eqref{SS3} \eqref{SS4}
and \eqref{a1} hold since they are derived from ({\bf SS1}),({\bf SS2}),({\bf SS3}).

 Let $x,y,z\in \mathcal{T}_M$. 
First we show
\eqref{b1}.
Let us put $z'=\mathcal{U}^{-1}(z)$.
Substituting $z$ by $z'$ in ({\bf SS1}) 
and subtracting it from  
({\bf ASS1}),  we have
\begin{equation}\nonumber
\begin{split}
0&=(\boldsymbol{\nabla}_x-\nabla_x)(y\ast z')
-(\boldsymbol{\nabla}_y-\nabla_y)(x\ast z')
\\&+\mathcal{U}^{-1}\big(
x\ast \boldsymbol{\nabla}_y\,z-y\ast \boldsymbol{\nabla}_y\,z
-x\ast E\ast \nabla_y z'+y\ast E\ast \nabla_x z'
\big)
\\&=
\mathcal{U}^{-1}\big(
x\ast(\boldsymbol{\nabla}_y \,z-\nabla_y\,z+\nabla_{y\star z}\,E)-(x\leftrightarrow y)
\big)~.
\end{split}
\end{equation}
In passing to the last line, 
we applied the almost hydrodynamically equivalent condition
and \eqref{a1}. Since $\mathcal{U}$ is assumed to be 
invertible on $M$,
we obtain
\begin{equation}\label{b2}
x\ast(\boldsymbol{\nabla}_y \,z-\nabla_y\,z+\nabla_{y\star z}\,E)
=y\ast (\boldsymbol{\nabla}_x \,z-\nabla_x\,z+\nabla_{x\star z}\,E)
~.
\end{equation}
Now, if we set $y$ to $E$,  the LHS of \eqref{b2}  becomes
$$
x\ast (\boldsymbol{\nabla}_E\,z-\nabla_E\,z+\nabla_{z}\,E)
=x\ast([E,z]-[E,z])=0~.
$$
Here we used the torsion freeness of $\nabla,\boldsymbol{\nabla}$ 
and $\boldsymbol{\nabla}\,E=0$.
Therefore the RHS of \eqref{b2} is zero when $y=E$. This implies \eqref{b1} since the multiplication $E\ast=\mathcal{U}$ is invertible.
\\
({\bf SS4}): 
To show that $(\nabla,\ast ,E)$ is a Saito structure on $M$,
it is enough to show that ({\bf SS4}) holds.
Expanding the flatness condition for $\boldsymbol{\nabla}$ 
by using \eqref{b1}, we have
\begin{equation}\nonumber
\begin{split}
0=&(\boldsymbol{\nabla}_x\boldsymbol{\nabla}_y-
\boldsymbol{\nabla}_y\boldsymbol{\nabla}_x-\boldsymbol{\nabla}_{[x,y]})(E\ast z)
\\
=&(\nabla_x\nabla_y-\nabla_y\nabla_x-\nabla_{[x,y]})(E\ast z)
\\
&-\nabla_x\nabla_{y\ast z}\,E+\nabla_y\nabla_{x\ast z}\,E
\\
&-\mathcal{W}\circ\mathcal{U}^{-1}
\big(
x\ast \nabla_y(E\ast z)-x\ast \nabla_{y\ast z}\,E
-(x\leftrightarrow y)
\big)
-\mathcal{W}([x,y]\ast z)
\\
\stackrel{\eqref{a1}}{=}&-\nabla_x\nabla_{y\ast z}\,E+\nabla_y\nabla_{x\ast z}\,E
-\mathcal{W}(x\ast \nabla_y\,z-y\ast \nabla_x\,z-[x,y]\ast z)~.
\end{split}
\end{equation}
Here we used the flatness of $\nabla$.
Now if we set $y$ to $e$, this equation becomes
\begin{equation}\label{b3}
0=-\nabla_x\nabla_z \,E
+\nabla_e\nabla_{x\ast z}\,E-
\mathcal{W}(x\ast [e,z]-\nabla_x\,z+[e,x]\ast z)~.
\end{equation}
Here we used the condition $\nabla e=0$ and the torsion freeness of $\nabla$.
Using the flatness of $\nabla$,
we can write the second term in the RHS as follows.
\begin{equation}\nonumber
\nabla_e\nabla_{x\ast z}\,E=\nabla_{x\ast z}
\underbrace{\nabla_e\,E}_{=[e,E]=e}+\mathcal{W}([e,x\ast z])
=\mathcal{W}([e,x\ast z])~.
\end{equation}
Here we used the torsion freeness of $\nabla$,
\eqref{SS2} and $\nabla e=0$. Substituting  this into \eqref{b3},
we obtain
\begin{equation}\nonumber
\begin{split}
0&=-\nabla_x\nabla_z \,E
+
\mathcal{W}([e,x\ast z]-x\ast [e,z]+\nabla_x\,z-[e,x]\ast z)
\\
&\stackrel{\eqref{SS1}}{=}-\nabla_x\nabla_z \,E+\mathcal{W}(\nabla_x\,z)~.
\end{split}
\end{equation}
This equation is equivalent to (\bf {SS4}).  
\end{proof}

The converse of Lemma \ref{biflat2ss} also holds.
\begin{lemma}
Let $(\nabla,\ast,E)$ be a Saito structure on $M$ with a unit $e$.
Then
$$(M_{\lambda},
\nabla,\boldsymbol{\nabla}^{(\lambda,0)},\ast,\star_{\lambda},e,
E-\lambda e) \quad (\lambda\in \mathbb{C})$$
is a bi-flat $F$-manifold,
where 
$M_{\lambda}$, $\boldsymbol{\nabla}^{(\lambda,0)}$,
$\star_{\lambda}$ are the same as in \S \ref{ss2ass}.
\end{lemma}
\begin{proof}
It is enough to check that the connections $\nabla$ and 
$\boldsymbol{\nabla}^{(\lambda,0)}$
are almost hydrodynamically equivalent. This is immediate from \eqref{eq:ss2ass}:
\begin{equation*}
\begin{split}
\boldsymbol{\nabla}^{(\lambda,0)}_x(y\ast z)-
\boldsymbol{\nabla}^{(\lambda,0)}_y(x\ast z)
&=
\nabla_x(y\ast z) - \nabla_{x\star_{\lambda}(y*z)} -(x\leftrightarrow y)\\
&=
\nabla_x(y\ast z) - \nabla_y(x\ast z) - \nabla_{\mathcal{U}_{\lambda}^{-1} (x\ast y\ast z)}E + \nabla_{\mathcal{U}_{\lambda}^{-1} (y\ast x\ast z)}E\\
&=
\nabla_x(y\ast z) - \nabla_y(x\ast z).
\end{split}
\end{equation*}
\end{proof}
\section{Matrix representations}\label{matrix-rep}
In this section, we describe the conditions for 
a Saito structure and an almost Saito structure
by matrix representations.

Let $x^{\alpha}$ ($1\leq \alpha\leq n$) be local coordinates of $M$
and denote $\partial_{x^{\alpha}}=\partial_{\alpha} \in \mathcal{T}_M$  ($1\leq \alpha\leq n$) the corresponding basis.
The identity matrix and the zero matrix of order $n$ are 
denoted $I_n$ and $O_n$.
For matrices $A,B$ of order $n$, 
$[A,B]=AB-BA$.

\subsection{Saito structure}
Let 
$(\nabla,\ast ,E=\sum_{\mu=1}^n E^{\mu}\partial_{\mu})$ be 
a Saito structure on $M$
with a unit $e$.
Let us write
$$E=\sum_{\mu=1}^n E^{\mu}\partial_{\mu}~,\quad
e=\sum_{\mu=1}^n e^{\mu}\partial_{\mu}~,
\quad
\partial_{\alpha}\ast \partial_{\beta}
=\sum_{\gamma=1}^nC_{\alpha\beta}^{\gamma}\,\partial_{\gamma}~,
\quad
\nabla_{\alpha}(\partial_{\beta})
=\sum_{\gamma=1}^n\Gamma_{\alpha\beta}^{\gamma}\partial_{\gamma}~.
$$
Let $W,Q,C_{\alpha},\Gamma_{\alpha}$ be matrices 
whose entries are given by
\begin{equation}\nonumber
{W^{\mu}}_{\alpha}=\partial_{\alpha} E^{\mu}~,
\quad
Q^{\mu}_{\alpha}=\partial_{\alpha} e^{\mu}~,
\quad 
{(C_{\alpha})^{\gamma}}_{\beta}=C_{\alpha\beta}^{\gamma}~,
\quad 
{(\Gamma_{\alpha})^{\gamma}}_{\beta}=\Gamma_{\alpha\beta}^{\gamma}~.
\end{equation}
Then the commutativity and associativity of $\ast $   are
expressed as follows.
$$
C_{\alpha\beta}^{\gamma}=C_{\beta\alpha}^{\gamma}~,\quad
[C_{\alpha},C_{\beta}]=O_n~.
$$
That $e$ is a unit of $\ast$ is expressed as
$$
\sum_{\alpha=1}^n e^{\alpha}C_{\alpha}=I_n~.
$$
The torsion freeness and the flatness of $\nabla$
are expressed as follows.
$$
\Gamma_{\alpha\beta}^{\gamma}=\Gamma_{\beta\alpha}^{\gamma}~,
\quad
\partial_{\alpha}\Gamma_{\beta}-\partial_{\beta}\Gamma_{\alpha}
+[\Gamma_{\alpha},\Gamma_{\beta}]=O_n~.
$$
The conditions ({\bf SS1})--({\bf SS4}) are expressed as follows.
\begin{eqnarray}
\label{ss1-2}
& \partial_{\alpha} C_{\beta}+
[\Gamma_{\alpha},C_{\beta}]=\partial_{\beta} C_{\alpha}+
[\Gamma_{\beta},C_{\alpha}]~.
\\
&\label{ss2-2}
\partial_{\alpha} (E\cdot C)+[\Gamma_{\alpha},E\cdot C]
=[W+E\cdot \Gamma,C_{\alpha}]+C_{\alpha}~.
\\
\label{ss3-2}
&Q+e\cdot \Gamma=O_n~.
\\
&\label{ss4-2}
\partial_{\alpha}(W+E\cdot \Gamma)+[\Gamma_{\alpha},
W+E\cdot \Gamma]=O_n~.
\end{eqnarray}
Here 
\begin{equation}\nonumber
E\cdot C=\sum_{\mu=1}^nE^{\mu}C_{\mu}~,\quad
E\cdot \Gamma=\sum_{\mu=1}^nE^{\mu}\Gamma_{\mu}~,\quad
e\cdot \Gamma=\sum_{\mu=1}^n e^{\mu}\Gamma_{\mu}~.
\end{equation}
Together with \eqref{ss1-2}, \eqref{ss2-2} can be simplified to 
\begin{equation}\label{ss2-3}
E(C_{\alpha\beta}^{\gamma})
=-\sum_{\mu=1}^n {W^{\mu}}_{\alpha}C_{\mu\beta}^{\gamma}
+{[W,C_{\alpha}]^{\gamma}}_{\beta}+C_{\alpha\beta}^{\gamma}~.
\end{equation}
Together with the flatness of $\nabla$, \eqref{ss4-2}
can be simplified to
\begin{equation}\label{ss4-3}
E(\Gamma_{\alpha\beta}^{\gamma})
=-\sum_{\mu=1}^n{W^{\mu}}_{\beta}\Gamma_{\mu\beta}^{\gamma}
+{[W,\Gamma_{\alpha}]^{\gamma}}_{\beta}
-\partial_{\alpha} {W^{\gamma}}_{\beta}.
\end{equation}

If one constructs an almost Saito structure 
$(\boldsymbol{\nabla}^{(\lambda,r)},\star_{\lambda},e)$ by \eqref{eq:ss2ass},
then the new multiplication $\star_{\lambda}$ is expressed as
\begin{equation}\label{newmultiplication1}
\partial_{\alpha}\star_{\lambda} \partial_{\beta}=
\sum_{\gamma=1}^n B_{\alpha\beta}^{\gamma}(\lambda)\partial_{\gamma}~,
\text{ where }
B_{\alpha}(\lambda)=(B_{\alpha\beta}^{\gamma}(\lambda))
=C_{\alpha}(E\cdot C-\lambda I_n)^{-1}~.
\end{equation}
The new connection $\boldsymbol{\nabla}^{(\lambda,r)}$ is given by
\begin{equation}\label{newconnection1}
\begin{split}
\boldsymbol{\nabla}^{(\lambda,r)}_{\alpha}(\partial_{\beta})
&=\sum_{\gamma=1}^n
{\big(\Gamma_{\alpha}
+(r I_n-W-E\cdot\Gamma)B_{\alpha}(\lambda)\big)^{\gamma}}_{\beta} \partial_{\gamma}~.
\end{split}
\end{equation}

\subsection{Almost Saito structure}
Next consider
an almost Saito structure 
$(\boldsymbol{\nabla},\star,e)$  with a unit $E$.
Let us write 
\begin{equation}\nonumber
e=\sum_{\mu=1}^n e^{\mu}\partial_{\mu}~,
\quad
E=\sum_{\mu=1}^nE^{\mu}\partial_{\mu}
\quad
\partial_{\alpha}\star \partial_{\beta}
=\sum_{\gamma=1}^nB_{\alpha\beta}^{\gamma}\,\partial_{\gamma}~,
\quad
\boldsymbol{\nabla}_{\alpha}(\partial_{\beta})
=\sum_{\gamma=1}^n\Omega_{\alpha\beta}^{\gamma}\partial_{\gamma}~.
\end{equation}
Let $Q,W,B_{\alpha},\Omega_{\alpha}$ be matrices
whose entries are given by
\begin{equation}\nonumber
{Q^{\mu}}_{\alpha}=\partial_{\alpha} e^{\mu}~,
\quad 
{W^{\mu}}_{\alpha}=\partial_{\alpha} E^{\mu}~,
\quad
{(B_{\alpha})^{\gamma}}_{\beta}=B_{\alpha\beta}^{\gamma}~,
\quad 
{(\Omega_{\alpha})^{\gamma}}_{\beta}
=\Omega_{\alpha\beta}^{\gamma}~.
\end{equation}
Then the commutativity and associativity of $\star$  are 
expressed as follows.
$$
B_{\alpha\beta}^{\gamma}=B_{\beta\alpha}^{\gamma}~,\quad
[B_{\alpha},B_{\beta}]=O_n~.
$$
That $E$ is a unit of $\star$ is expressed as
$$
\sum_{\alpha=1}^n E^{\alpha}B_{\alpha}=I_n~.
$$
The torsion freeness and the flatness of $\nabla$
are expressed as follows.
$$
\Omega_{\alpha\beta}^{\gamma}=\Omega_{\beta\alpha}^{\gamma}~,
\quad
\partial_{\alpha}\Omega_{\beta}-\partial_{\beta}\Omega_{\alpha}
+[\Omega_{\alpha},\Omega_{\beta}]=O_n~.
$$
The conditions ({\bf ASS1})--({\bf ASS4}) are expressed as follows.
\begin{eqnarray}
\label{ass1-2}
& \partial_{\alpha} B_{\beta}+[\Omega_{\alpha},B_{\beta}]
=\partial_{\beta} B_{\alpha}+[\Omega_{\beta},B_{\alpha}]~.
\\
\label{ass2-2}
&
\partial_{\alpha} (e\cdot B)+[\Omega_{\alpha},e\cdot B]
=[Q+e\cdot \Omega,B_{\alpha}]-B_{\alpha}(e\cdot B)~.
\\
\label{ass3-2}
&W+E\cdot \Omega=r I_n ~.
\\
\label{ass4-2}
&\partial_{\alpha}(Q+e\cdot \Omega)+[\Omega_{\alpha},
Q+e\cdot \Omega]+(Q+e\cdot \Omega) B_{\alpha}=O_n~.
\end{eqnarray}
Here 
\begin{equation}\nonumber
e\cdot B=\sum_{\mu=1}^ne^{\mu}B_{\mu}~,\quad
e\cdot \Omega=\sum_{\mu=1}^ne^{\mu}\Omega_{\mu}~,\quad
E\cdot \Omega=\sum_{\mu=1}^n E^{\mu}\Omega_{\mu}~.
\end{equation}
Together with \eqref{ass1-2}, \eqref{ass2-2} can be simplified as 
follows.
\begin{equation}\label{ass2-3}
e(B_{\alpha\beta}^{\gamma})=
-\sum_{\mu=1}^n{Q^{\mu}}_{\alpha}B_{\beta\mu}^{\gamma}
+{[Q,B_{\alpha}]^{\gamma}}_{\beta}
-{(B_{\alpha} (e\cdot B))^{\gamma}}_{\beta}~.
\end{equation}
Together with the flatness of $\boldsymbol{\nabla}$,
\eqref{ass4-2} can also be simplified as follows.
\begin{equation}\label{ass4-3}
e(\Omega_{\alpha\beta}^{\gamma})
=-\sum_{\mu=1}^n {Q^{\mu}}_{\alpha}\Omega_{\mu\beta}^{\gamma}
+{[Q,\Omega_{\alpha}]^{\gamma}}_{\beta}
-{\partial_{\alpha}Q^{\gamma}}_{\beta}
-{((Q+e\cdot \Omega)B_{\alpha})^{\gamma}}_{\beta}~.
\end{equation}
If one constructs a Saito structure 
$(\nabla,\ast,E)$ by \eqref{eq:ass2ss},
then the new multiplication $\ast$ is expressed as
$$
\partial_{\alpha}\ast \partial_{\beta}=
\sum_{\gamma=1}^n C_{\alpha\beta}^{\gamma}\partial_{\gamma}~,
\text{ where }
C_{\alpha}=(C_{\alpha\beta}^{\gamma})=B_{\alpha}(e\cdot B)^{-1}~.
$$
The new connection $\nabla$ is given by
\begin{equation}\nonumber
\begin{split}
\nabla_{\alpha}(\partial_{\beta})
&=\sum_{\gamma=1}^n
{\big(\Omega_{\alpha}-(Q+e\cdot\Omega) C_{\alpha}
\big)^{\gamma}}_{\beta} \partial_{\gamma}~.
\end{split}
\end{equation}

\section{The orbit spaces of complex reflection groups}
\label{orbit-space}
\subsection{Orbit spaces of finite complex reflection groups}
\label{CRGs}
Let  $n$ be a positive integer and 
let $V=\mathbb{C}^n$ be equipped with the standard Hermitian metric.
The standard coordinates of $V$ are 
denoted $u^1,\ldots,u^n$ 
and the ring of polynomials on $V$ is denoted $\mathbb{C}[V]=\mathbb{C}[u^1,\ldots,u^n]=\mathbb{C}[u]$.

Let $G\subset GL(V)$ be a (not necessarily irreducible) 
finite complex  reflection group. 
Denote  $\mathbb{C}[V]^G$  the ring of $G$-invariant polynomials
and let $M=\mathrm{Spec}\, \mathbb{C}[V]^G$ be the 
space of $G$-orbits.
It is well-known that $\mathbb{C}[V]^G$ is 
a polynomial ring generated by $n=\mathrm{dim}\,V$ 
homogeneous polynomials. Such generators are called a set of basic invariants.
Fix a set of basic invariants $x^1,\ldots, x^n\in \mathbb{C}[V]^G$.
So $\mathbb{C}[V]^G=\mathbb{C}[x^1,\ldots,x^n]=\mathbb{C}[x]$.
The degrees $d_1,\ldots, d_n$ of $x^1,\ldots, x^n$ are uniquely determined by $G$
and called the degrees of $G$.
Except for \S \ref{sec:monomial-groups},
we always assume that the degrees are in
descending order:
\begin{equation}\nonumber
d_1\geq \ldots\geq d_n\geq 1~,
\end{equation}
for any set of basic invariants 
$x^1,\ldots,x^n$.

We regard $\mathbb{C}[x]$ as a graded ring with 
the grading $\deg\, x^{\alpha}=d_{\alpha}$ $(1\leq \alpha\leq n)$ .
Note that for a homogeneous polynomial 
$f\in \mathbb{C}[u]$ or for a weighted homogeneous polynomial
$f\in \mathbb{C}[x]$, it holds that
\begin{equation}\label{degree}
\sum_{\alpha=1}^n
d_{\alpha}x^{\alpha}\frac{\partial f}{\partial x^{\alpha}}=
\sum_{k=1}^n u^k\frac{\partial f}{\partial u_k}=(\mathrm{deg}\,f) f~.
\end{equation}
We set
\begin{equation}\label{E-deg}
E_{\mathrm{deg}}=\sum_{\alpha=1}^n
d_{\alpha}x^{\alpha}\frac{\partial }{\partial x^{\alpha}}
\end{equation}

Let
$\mathcal{A}$  be the arrangement of reflection hyperplanes of $G$.
For each hyperplane $H\in \mathcal{A}$, 
fix a  linear form $L_H\in V^*$ such that 
$H=\mathrm{ker}\,L_H$.
Let $e_H$ be the order of 
the cyclic subgroup of $G$ fixing $H$ pointwise.
Let us  define polynomials $\Pi, \delta\in \mathbb{C}[V]$ by
\begin{equation}\label{def:delta}
\Pi=\prod_{H\in \mathcal{A}}L_H^{e_H-1} ,\qquad
\delta=\prod_{H\in\mathcal{A}} L_H^{e_H}~.
\end{equation}
Then it is known that
$\Pi$ is skew invariant and $\delta$ is invariant
(i.e. $g(\Pi)=\det(g)\Pi$ and $g(\delta)=\delta$ for any $g\in G$).
Since 
\begin{equation}
\frac{\partial(x^1,\ldots,x^n)}{\partial(u^1,\ldots,u^n)}=\text{const.}\cdot
\Pi ~,
\end{equation}
the discriminant locus of the orbit map $\omega: V\to M$,
$\omega(u^1,\ldots,u^n)=(x^1(u),\ldots, x^n(u))$,
is given by $\{\Delta=0\}\subset M$.
Here  we write $\Delta$ for $\delta$ 
when we regard it as a polynomial in $x$, 
i.e. $\Delta(x(u))=\delta(u)$.
We set
$$
M_0=M\setminus \{\Delta=0\}~.
$$

\subsection{Natural connection on $TM_0$}
The trivial holomorphic connection on $TV$ induces a 
connection $\boldsymbol{\nabla}^V$ on $TM_0$.
Explicitly, it is given as follows.
\begin{equation}\nonumber
\begin{split}
\boldsymbol{\nabla}^V_{\partial_{\alpha}}(\partial_{\beta})
&=\sum_{\gamma=1}^n\Omega_{\alpha\beta}^{\gamma}\,\partial_{\gamma}~,
\\
\Omega_{\alpha\beta}^{\gamma}&=
-\sum_{i,j=1}^n\frac{\partial u^i}{\partial x^{\alpha}}\frac{\partial u^j}{\partial x^{\beta}}
\frac{\partial^2  x^{\gamma}}{\partial u^i \partial u^j}~~\in
\mathbb{C}(x)~,\quad 
\Omega_{\alpha}=(\Omega_{\alpha\beta}^{\gamma})_{\gamma,\beta}~.
\end{split}
\end{equation}
The induced connection $\boldsymbol{\nabla}^V$ is 
torsion free and flat.

Let us list a few properties of the connection matrix $\Omega_{\alpha}$.
Below, $M(n,\mathbb{C}[x])$ denotes the ring of $n\times n$ matrices 
with coefficients in $\mathbb{C}[x]$.
\begin{proposition} \label{AP}
\begin{eqnarray}
&&\label{AP1}
\Delta\Omega_{\alpha\beta}^{\gamma}\in \mathbb{C}[x]~.
\\
&& \label{AP2}
\Delta\cdot \det \Omega_{\alpha}\in \mathbb{C}[x].
\text{ Moreover }
 \det \Omega_{\alpha}=\frac{\text{const.}}{\Delta}
\text{ if }\mathrm{deg}\,\delta=nd_{\alpha}~.
\\
&&\label{AP3}
\Omega_{\mu}^{-1}\in M(n,\mathbb{C}[x])
\text{ if }\mathrm{deg}\,\delta=nd_{\mu}\text{ and  if }
\det \Omega_{\mu}\neq 0~.
\\
&& \label{AP4}
(\Omega_{\mu})^{-1}\Omega_{\alpha}
    \in M(n,(\mathbb{C}[x])
\text{ if }\mathrm{deg}\,\delta=n d_{\mu}\text{ and  if }
\det \Omega_{\mu}\neq 0~.
\end{eqnarray}
\end{proposition}
\begin{lemma}\label{IP}
Let $$
{J^{\gamma}}_{i}=\frac{\partial x^{\gamma}}{\partial u^i}~.
$$
If $\det \Omega_{\mu}\neq 0$, then
$$
J^{-1}\Omega_{\mu}^{-1}\in M(n,\mathbb{C}[u])~.
$$
\end{lemma}
Proposition \ref{AP} and Lemma \ref{IP} 
can be proven by studying pole orders
along each reflection hyperplane $H\in \mathcal{A}$. 
See \S \ref{appendix:proofs}.

\begin{lemma}
\begin{equation}\label{EOmega}
\sum_{\alpha=1}^n d_{\alpha}x^{\alpha}\Omega_{\alpha\beta}^{\gamma}
=\begin{cases}(1-d_{\gamma})&(\beta=\gamma)\\
0&(\beta\neq\gamma)
\end{cases}~~~.
\end{equation}
\end{lemma}
\begin{proof}
By \eqref{degree},
\begin{equation}\nonumber
\begin{split}
-\sum_{\alpha=1}^n d_{\alpha}x^{\alpha}\Omega_{\alpha\beta}^{\gamma}
&=
\sum_{i,j=1}^n
\left(
\sum_{\alpha=1}^nd_{\alpha}x^{\alpha}
\frac{\partial u^i}{\partial x^{\alpha}}\right)
\frac{\partial u^j}{\partial x^{\beta}}\frac{\partial^2 x^{\gamma}}{\partial u^i\partial u^j}
\\
&=\sum_{j=1}^n
\frac{\partial u^j}{\partial x^{\beta}}\cdot
\sum_{i=1}^n u^i \frac{\partial}{\partial u^i}
\left(\frac{\partial x^{\gamma}}{\partial u^j}\right)
\\
&=(d_{\gamma}-1)\sum_{j=1}^n\frac{\partial u^j}{\partial x^{\beta}}
\frac{\partial x^{\gamma}}{\partial u^j}
\\
&=\eqref{EOmega}~.
\end{split}
\end{equation}
\end{proof}
Notice that \eqref{EOmega} implies that 
\begin{equation}\label{nablaE}
\boldsymbol{\nabla}^V_x\,E_{\mathrm{deg}}=
x\quad (x\in \mathcal{T}_{M_0})~.
\end{equation}

\subsection{The Natural (Almost) Saito Structure}
Let $\mathcal{X}_{-d_1}$ be the $\mathbb{C}$-vector space of polynomial vector fields on $M$
of degree $-d_1$. Here $d_1$ is the maximal degree of $G$.
Let $x_1,\ldots, x_n$ be a set of basic invariants of $G$.
Then the space $\mathcal{X}_{-d_1}$ is spanned by
$\partial_{x^{\alpha}}$ with $d_{\alpha}=d_1$.

\begin{definition}\label{def:natural}
(1)
An almost Saito structure $(\boldsymbol{\nabla},\star,e)$
on $M_0$ is a natural almost Saito structure for $G$ if 
it satisfies the following conditions.
\begin{itemize}
\item $\boldsymbol{\nabla}=\boldsymbol{\nabla}^V$.
\item A unit of $\star$ is $\frac{1}{d_1}E_{\mathrm{deg}}$
and hence the parameter is $r=\frac{1}{d_1}$ (by \eqref{nablaE}).
\item $e\neq 0$ and $e \in \mathcal{X}_{-d_1}$, i.e. $e$ is a nonzero vector field on $M$ of degree $-d_1$. 
\end{itemize}
(2) A Saito structure $(\nabla,\ast,E)$ on an open subset of $M$
is a natural Saito structure for $G$ if
its dual almost Saito structure with parameter 
$r=\frac{1}{d_1}$
is a natural almost Saito structure for $G$. 
(Especially, the Euler vector field $E$ must be $\frac{1}{d_1}E_{\mathrm{deg}}$.)
\\
(3)
A Saito structure $(\nabla,\ast,E)$ on $M$
is called a polynomial Saito structure 
if 
\begin{itemize}
\item there exist $\nabla$-flat coordinates $t^1,\ldots,t^n$
which form a set of basic invariants of $G$.
\item all entries of the matrix representations of $\partial_{t^{\alpha}}\ast$
($1\leq \alpha\leq n$) with respect to the basis
$\partial_{t^1},\ldots,\partial_{t^n}$ are polynomials in $t$.
\end{itemize}
\end{definition}

\begin{remark}
In the case when $G$ is a Coxeter group, 
the Saito structure considered in \cite{Saito1993},
\cite{SaitoSekiguchiYano} and \cite{Dubrovin1993}
can be characterized as follows.
\begin{itemize}
\item
The intersection form $g$ (i.e. the metric of the dual almost 
Frobenius structure)
is the complexification of the 
standard Euclidean metric.
\item The Euler vector field is $\frac{1}{d_1}E_{\mathrm{deg}}$.
\item  
The unit vector field $e$ is the 
vector field $\partial_{x^1}$ corresponding to a basic invariant $x^1$ of the maximal degree.
\end{itemize}
The Levi--Civita connection of the above intersection form  $g$ is the natural connection
$\nabla^V$. Therefore
the three conditions in Definition \ref{def:natural}-(1)
appear as a result of a straightforward generalization of
the case of the Coxeter groups.

Let us mention  that the Frobenius manifold structures for the Shephard groups
\cite[\S 5.3]{Dubrovin2004}
is a result of 
a generalization in a different direction,
in which one 
takes  $\frac{\partial x^n}{\partial u^i\partial u^j}$ $(1\leq i,j\leq n)$ as $g(\partial_{u^i},\partial_{u^j})$.
\end{remark}

\begin{lemma}\label{uniqueness2}
Let $e\in \mathcal{X}_{-d_1}$ be a nonzero  vector field 
such that the pair $(\boldsymbol{\nabla}^V,e)$ 
is regular on $M_0$.
Then the pair $(\boldsymbol{\nabla}^V,e)$ 
makes a natural 
almost Saito structure for $G$ 
if and only if it satisfies the conditions $(${\bf ASS1}$)$ and $(${\bf ASS2}$)$.
\end{lemma}

\begin{proof}
By \eqref{nablaE}, it is clear that $\boldsymbol{\nabla}^V$ and $\frac{1}{d_1}E_{\mathrm{deg}}$ satisfy the condition ({\bf ASS3}).
If  the pair $(\boldsymbol{\nabla}^V,e)$  is regular,
the multiplication $\star$ 
obtained from ({\bf ASS4}) (or \eqref{mult}) 
has the unit $\frac{1}{d_1}E_{\mathrm{deg}}$
since
\begin{equation}\begin{split}\nonumber
x\star E_{\mathrm{deg}}&=
-\mathcal{Q}^{-1}(\boldsymbol{\nabla}^V_x\boldsymbol{\nabla}^V_{E_\mathrm{deg}}\,e)
+\boldsymbol{\nabla}^V_x E_{\mathrm{deg}}
\\&=
-\mathcal{Q}^{-1}(\boldsymbol{\nabla}^V_x (1-d_1)e
)+x=d_1x~.
\end{split}
\end{equation}
Therefore the statement follows from Proposition \ref{uniqueness1}.
\end{proof}

Lemma \ref{uniqueness2} can be used to show 
the nonexistence of the natural Saito structure for 
certain irreducible groups.  See \S \ref{example-nonexistence}.

\section{When the discriminant $\Delta$ is a monic of degree $n$ in $x^1$}
\label{monic-degreen}
\subsection{Assumptions}\label{sec:assumptions}
Let $G\subset GL(V)$ be a finite complex reflection group. 
In this section, we assume that $G$ satisfies the conditions 
(i) and (ii) below.
\begin{enumerate}
\item[ (i)] $d_{\alpha}>1$ $(1\leq \alpha\leq n=\dim V)$.
\item[(ii)] there exists a set of basic invariants
$x^1,\ldots,x^n$
such that
the discriminant $\Delta\in \mathbb{C}[x]$ is a monic polynomial 
of degree $n$ as a polynomial in $x^1$.
\end{enumerate}
The assumption (i) implies that $V$ does not contain
a trivial representation of $G$.
\begin{remark}
If $G$ is irreducible, then (i) (ii) are equivalent to 
the condition that $G$ is a duality group. See \S \ref{example-well-generated}.
\end{remark}

We will see that
the pair $(\boldsymbol{\nabla}^V,\partial_{x^1})$ is regular on $M_0$
(Corollary \ref{reg}).
Moreover we show that 
it makes 
a regular natural almost Saito structure for $G$
and that
its dual Saito structure is  
a polynomial Saito structure defined on the whole orbit space $M$
(Corollary \ref{main-corollary}).
We do this by first constructing 
a polynomial Saito structure on the whole orbit space $M$,
and show that its dual almost Saito structure is 
$(\boldsymbol{\nabla}^V,\partial_{x^1})$ (Theorem \ref{main-theorem}).

For the sake of simplicity, we write
$\mathbb{C}[x']=\mathbb{C}[x^2,\ldots,x^n]$.

\subsection{Constructing a polynomial Saito structure on $M$}
\label{sec:constructing}
By \eqref{AP1}, 
$\Delta \Omega_{\alpha\beta}^{\gamma}$ is 
a polynomial in $x$.
By assumption (i), $\mathrm{deg}\,\Delta=nd_1$. So
$$
\mathrm{deg}\,\Delta \Omega_{\alpha\beta}^{\gamma}=
d_{\gamma}-d_{\alpha}-d_{\beta}+nd_1<(n+1)d_1~.
$$
Therefore it is at most of degree $n$ 
as a polynomial in $x^1$. 
So we can write it in the following form:
\begin{equation}\label{def-Gamma}
\begin{split}
\Delta \Omega_{\alpha\beta}^{\gamma}&
=(x^1)^n\Gamma_{\alpha\beta}^{\gamma}+(x^1)^{n-1} D_{\alpha\beta}^{\gamma}
+\cdots,
\quad 
(\Gamma_{\alpha\beta}^{\gamma}~,~D_{\alpha\beta}^{\gamma}\in \mathbb{C}[x'])~.
\end{split}\end{equation}
Here $\cdots$ means terms of degree less than $n-1$ in $x^1$. 
We also set
\begin{equation}\label{def-Gamma2}
\Gamma_{\alpha}={(\Gamma_{\alpha\beta}^{\gamma})}_{\gamma,\beta}~,\quad
D_{\alpha}={(D_{\alpha\beta}^{\gamma})}_{\gamma,\beta}~.
\end{equation}

\begin{lemma}\label{BL2} 
Under the assumptions (i) and (ii), the following holds.
\begin{eqnarray}
&&\label{BL2-1}
\Gamma_{\alpha\beta}^{\gamma}=0 \text{ if } d_{\gamma}\leq d_{\beta}~,\quad \Gamma_{1}=O_n ~.
\\
&& \label{BL2-2}
D_1=
\frac{1}{d_1}\mathrm{diag}[1-d_1,\ldots,1-d_n]
-\sum_{\alpha=2}^n\frac{d_{\alpha}}{d_1}\,x^{\alpha}
\Gamma_{\alpha}~.
\\
&&\label{BL2-3} 
\Delta\det\Omega_1=\det D_1
=\prod_{\gamma=1}^n\frac{1-d_{\gamma}}{d_1}\neq 0~. 
\\
&&\label{BL2-4}
\Omega_1^{-1}
\in M(n,\mathbb{C}[x]) \text{ and  }
\Omega_1^{-1}-D_1^{-1}x^1\in M(n,\mathbb{C}[x'])~.
\\
&&\label{BL2-5}
\Omega_1^{-1}\Omega_{\alpha}\in 
M(n,\mathbb{C}[x])\text{ and }
\Omega_1^{-1}(\Omega_{\alpha}-\Gamma_{\alpha})\in 
M(n,\mathbb{C}[x'])~.
\end{eqnarray} 
\end{lemma}

\begin{proof}
(For \eqref{BL2-1} and \eqref{BL2-2}, 
the assumption (i) is not necessary.)
\\
\eqref{BL2-1}: let us compute the degree of $\Gamma_{\alpha\beta}^{\gamma}\in 
\mathbb{C}[x']$.
$$\mathrm{deg}\,\Gamma_{\alpha\beta}^{\gamma}=
\mathrm{deg}(\Delta \Omega_{\alpha\beta}^{\gamma})-nd_1
=d_{\gamma}-d_{\alpha}-d_{\beta}.
$$
Therefore
$\Gamma_{\alpha\beta}^{\gamma}=0$ if 
$d_{\gamma}-d_{\alpha}-d_{\beta}<0$.
A sufficient condition for this is $d_{\gamma}-d_{\beta}\leq 0$.
When $\alpha=1$,
this holds for any $(\beta,\gamma)$
since $d_1\geq d_{\gamma} $ and $d_{\beta}\geq 1$.
\\
\eqref{BL2-2} By the expansion \eqref{def-Gamma} and \eqref{BL2-1},
\begin{equation}\nonumber
\Delta\sum_{\alpha=1}^{n}\frac{d_{\alpha}}{d_1}\,x^{\alpha}
\Omega_{\alpha}
=\left(D_{1}
+\sum_{\alpha=2}^n\frac{d_{\alpha}}{d_1}\,x^{\alpha}
\Gamma_{\alpha}
\right) (x^1)^n+\cdots
\end{equation}
Here $\cdots$ means terms of smaller degree in $x^1$. On the other hand, by \eqref{EOmega},
$$
\Delta\sum_{\alpha=1}^{n}\frac{d_{\alpha}}{d_1}\,x^{\alpha}\Omega_{\alpha}
=\frac{\Delta}{d_1} \,\mathrm{diag}[1-d_1,\ldots,1-d_n]
=\frac{(x^1)^n}{d_1}\mathrm{diag}[1-d_1,\ldots,1-d_n]
+\cdots
~.
$$
The last equality follows from the assumption (ii).
Comparing the coefficients of $(x^1)^n$,  we obtain \eqref{BL2-2}.
\\
\eqref{BL2-3}: by the expansion \eqref{def-Gamma},
\begin{equation}\nonumber
\Delta\det\Omega_1=\frac{1}{\Delta^{n-1}}(
\det D_1 \cdot (x^1)^{n(n-1)}+\cdots)~.
\end{equation}
On the other hand, since $\mathrm{deg}\,\Delta=nd_1$,
$\Delta \det \Omega_1$ is a constant  by \eqref{AP2}.
Therefore in the RHS, the denominator must divide the 
numerator. Given their degrees, 
$
\Delta\det\Omega_1=\det D_1.
$
By \eqref{BL2-1} \eqref{BL2-2}, 
$D_1$ is triangular with diagonal entries
$\frac{1-d_{\gamma}}{d_1}$ ($1\leq \gamma \leq n$).
Therefore $$
\Delta \det \Omega_1=\det D_1=\prod_{\gamma=1}^n\frac{1-d_{\gamma}}{d_1}\neq 0~.
$$
The last inequality holds since we assumed $d_{\alpha}>1$.
\\
\eqref{BL2-4}:
since $\mathrm{deg}\,\Delta=nd_1$ and 
$\det \Omega_1\neq 0$ by \eqref{BL2-3}, 
$\Omega_1^{-1}\in M(n,\mathbb{C}[x])$ follows from \eqref{AP3}. 
Given that $$\mathrm{deg}\,{(\Omega_1^{-1})^{\gamma}}_{\beta}
=d_1+d_{\gamma}-d_{\beta}<2d_1~,$$
${(\Omega_1^{-1})^{\gamma}}_{\beta}$ is at most degree one 
as a polynomial in $x^1$. So
let us write $\Omega_1^{-1}=Ax^1+B$ ($A,B\in M(n,\mathbb{C}[x'])$). Then by the expansion \eqref{def-Gamma},
$$
\Delta\cdot I_n= (\Delta\Omega_1)\Omega_1^{-1}=
(D_1 (x^1)^{n-1}+\cdots) (Ax^1+B)
=(D_1 A (x^1)^n+\cdots ).
$$
Comparing the coefficients of $(x^1)^n$, $AD_1=I_n$.
\\
\eqref{BL2-5}:
by \eqref{BL2-4} and the expansion \eqref{def-Gamma},
\begin{equation}\nonumber
\begin{split}
\Omega_1^{-1}(\Omega_{\alpha}-\Gamma_{\alpha})&=
\frac{1}{\Delta}
 (x^1 D_1^{-1}+B)
((x^1)^{n-1} (D_{\alpha}-a\Gamma_{\alpha}) +\cdots)
\\
&=\frac{1}{\Delta}((x^1)^n D_1^{-1}(D_{\alpha}-a\Gamma_{\alpha})+\cdots)~.
\end{split}\end{equation}
Here $a\in \mathbb{C}[x']$ is the coefficient of $(x^1)^{n-1}$ in $\Delta$, i.e.,
$$\Delta=(x^1)^n+a(x^1)^{n-1}+\cdots.$$
On the other hand, 
$\Omega_1^{-1}, 
\Omega_1^{-1}\Omega_{\alpha}\in M(n,\mathbb{C}[x])$
by \eqref{BL2-4} and \eqref{AP4},
and $\Gamma_{\alpha}\in M(n,\mathbb{C}[x'])$.
So the LHS is a matrix with polynomial entries.
Therefore in the RHS, the denominator $\Delta$ must divide
the  numerator and its quotient is $D_1^{-1}(D_{\alpha}-a\Gamma_{\alpha})$.
Since $D_1,D_{\alpha}, a, \Gamma_{\alpha}$ do not depend on $x^1$,
$$
\Omega_1^{-1}(\Omega_{\alpha}-\Gamma_{\alpha})
=D_1^{-1}(D_{\alpha}-a\Gamma_{\alpha}) \in M(n,\mathbb{C}[x'])~.$$
The proof of Lemma \ref{BL2}  is finished.
\end{proof}

\begin{corollary}\label{reg}
The pair $(\boldsymbol{\nabla}^V,\partial_{x^1})$ is regular.
\end{corollary}
\begin{proof}
The endomorphism $\mathcal{Q}=\boldsymbol{\nabla}^V(\partial_{x^1})\in \mathrm{Hom}(\mathcal{T}_{M_0},\mathcal{T}_{M_0})$ is represented by the matrix $\Omega_1$.
By \eqref{BL2-3}, it is invertible on $M_0$.
Thus $\mathcal{Q}$ is an isomorphism.
\end{proof}

\begin{lemma}\label{BL4}
Assume (i) and (ii).
Let 
\begin{equation}\nonumber 
\begin{split}
&C_{\alpha}=\Omega_1^{-1}(\Omega_{\alpha}-\Gamma_{\alpha})\in
M(n,\mathbb{C}[x'])~,
\quad C_{\alpha\beta}^{\gamma}=
{(C_{\alpha})^{\gamma}}_{\beta}~,
\quad
U=\sum_{\alpha=1}^n \frac{d_{\alpha}}{d_1} x^{\alpha}\,C_{\alpha}
~.
\end{split}
\end{equation}
Then the following holds.
\begin{eqnarray}
&&\label{BL4-0}
C_1=I_n~.
\\
&&\label{BL4-1}
\Gamma_{\alpha\beta}^{\gamma}=\Gamma_{\beta\alpha}^{\gamma}~,\quad
C_{\alpha\beta}^{\gamma}=C_{\beta\alpha}^{\gamma}~.
\\
&&\label{BL4-2}
\partial_{\alpha}\Gamma_{\beta}-
\partial_{\beta}\Gamma_{\alpha}+[\Gamma_{\alpha},\Gamma_{\beta}]=O_n~.
\\
&&\label{BL4-3}
 \partial_{\alpha}C_{\beta}+[\Gamma_{\alpha},C_{\beta}]=
\partial_{\beta}C_{\alpha}+[\Gamma_{\beta},C_{\alpha}]~.
\\
&& \label{BL4-4}
C_{\alpha}C_{\beta}=C_{\beta}C_{\alpha}~.
\\
&&\label{BL4-5}
U=\Omega_1^{-1}D_1~,\quad \det U=\Delta~.
\end{eqnarray}
\end{lemma}
\begin{proof}
\eqref{BL4-0} holds since $\Gamma_1=O_n$.\\
\eqref{BL4-1} follows from the torsion freeness of $\boldsymbol{\nabla}^V$,
i.e.
$\Omega_{\alpha\beta}^{\gamma}=\Omega_{\beta\alpha}^{\gamma}$.
\\
\eqref{BL4-2} \eqref{BL4-3} \eqref{BL4-4} follow from the flatness of $\boldsymbol{\nabla}^V$
as follows.
Substituting $\Omega_{\alpha}=\Omega_1C_{\alpha}+\Gamma_{\alpha}$ into
\begin{equation}\nonumber
\partial_{\alpha}\Omega_{\beta}-\partial_{\beta}\Omega_{\alpha}+[\Omega_{\alpha},\Omega_{\beta}]=O_n \quad (1\leq \alpha,\beta\leq n)~,
\end{equation}
we obtain
\begin{equation}\nonumber
\begin{split}
O_n&=
\underbrace{\partial_{\alpha}\Gamma_{\beta}-\partial_{\beta}\Gamma_{\alpha}
+[\Gamma_{\alpha},\Gamma_{\beta}]}_{S_1}
\\
&+\underbrace{\Omega_1\big(\partial_{\alpha}C_{\beta}+[\Gamma_{\alpha},C_{\beta}]-\partial_{\beta}C_{\alpha}-[\Gamma_{\beta},C_{\alpha} ]\big)}_{S_2}
\\
&+\underbrace{(\partial_{\alpha}\Omega_1+[\Gamma_{\alpha},\Omega_1])C_{\beta}
-(\partial_{\beta}\Omega_1 +[\Gamma_{\beta},\Omega_1])C_{\alpha}
+[\Omega_1C_{\alpha},\Omega_1 C_{\beta}]}_{S_3}~.
\end{split}
\end{equation}
Let us simplify $S_3$.
Using $\partial_{\alpha}\Omega_1=\partial_1\Omega_{\alpha}-[\Omega_{\alpha},\Omega_1]$,
\begin{equation}\begin{split}\nonumber
S_3&=
(\partial_1 \Omega_{\alpha}-[\Omega_1 C_{\alpha},\Omega_1])C_{\beta}
-(\partial_1 \Omega_{\beta}-[\Omega_1 C_{\beta},\Omega_1])C_{\alpha}+[\Omega_1 C_{\alpha},\Omega_1C_{\beta}]
\\
&=\partial_1\Omega_{\alpha}C_{\beta}-\partial_1\Omega_{\beta}C_{\alpha}+\Omega_1^2[C_{\alpha},C_{\beta}]~.
\end{split}
\end{equation}
Noticing that
$$
O_n=\partial_1 C_{\alpha}=D_1^{-1}(\Omega_{\alpha}-\Gamma_{\alpha})
+\Omega_1^{-1}\partial_1\Omega_{\alpha}
~~\Rightarrow~~
\partial_1\Omega_{\alpha}=-\Omega_1 D_1^{-1}\Omega_1 C_{\alpha}~,
$$
we have
$$
S_3=\Omega_1(I_n-D_1^{-1})\Omega_1 [C_{\alpha},C_{\beta}]~.
$$
Now regard $S_1,S_2,S_3$ as 
functions in $x^1$.
Then $S_1$ is a  constant matrix.
$S_2$ and $S_3$ 
depend on $x^1$ only through the factor $\Omega_1$
and $\Omega_1(I_n-D_1^{-1})\Omega_1$.
Since entries of $\Omega_1$ 
are rational functions 
whose numerator and denominator have degrees $n-1$ and $n$,
entries of $S_2+S_3$ 
are rational functions 
whose numerator and denominator have degrees $2n-1$ and $2n$.
Therefore 
$O_n=S_1+S_2+S_3$ implies that
$S_1=O_n$ and $S_2+S_3=O_n$.
Applying a 
similar argument to $S_2+S_3=O_n$, we obtain 
$S_2=O_n$ and $S_3=O_n$.
\\
\eqref{BL4-5}: by \eqref{EOmega} and \eqref{BL2-2},
\begin{equation}\nonumber
U=
\sum_{\alpha=1}^n \frac{d_{\alpha}}{d_1}x^{\alpha}C_{\alpha}
=
\Omega_1^{-1}\sum_{\alpha=1}^n  
\frac{d_{\alpha}}{d_1}x^{\alpha}
(\Omega_{\alpha}-\Gamma_{\alpha})=
\Omega_1^{-1}D_1~.
\end{equation}
Therefore by \eqref{BL2-3},
$$
\det U=\frac{\det D_1}{\det \Omega_1}=\Delta~.
$$
\end{proof}

\begin{theorem} \label{main-theorem}
Assume that a finite complex reflection group $G$ satisfies the conditions 
(i) and (ii) in \S \ref{sec:assumptions}.
Let $\nabla$  and $\ast$
be a connection and a multiplication on 
the tangent bundle $TM$ of the orbit space $M$  of $G$ given by
\begin{equation}\nonumber
\begin{split}
\nabla_{\alpha}(\partial_{\beta})=\sum_{\gamma=1}^n\Gamma_{\alpha\beta}^{\gamma}\cdot \partial_{\gamma}~,
\quad 
\partial_{\alpha}\ast\partial_{\beta}=\sum_{\gamma=1}^n
C_{\alpha\beta}^{\gamma}\cdot\partial_{\gamma}~.
\end{split}
\end{equation}
Then\\
(1) $\nabla$ is  torsion-free and flat.
\\
(2) $\ast$ is commutative and associative
and its unit is $e=\partial_{x^1}$.
\\
(3) $(\nabla,\ast,\frac{1}{d_1}E_\mathrm{deg})$ is a polynomial 
Saito structure defined on the whole orbit space $M$.
\\
(4) Its dual almost Saito structure with the parameter $r=\frac{1}{d_1}$
is the regular natural almost Saito structure $(\boldsymbol{\nabla}^V,\partial_{x^1})$ for $G$.
\end{theorem}
\begin{proof}
We check the conditions by using the matrix representations 
in \S \ref{matrix-rep}.

(1) follows from  \eqref{BL4-1} \eqref{BL4-2}.

(2) follows from  \eqref{BL4-0} \eqref{BL4-1}  and \eqref{BL4-4}.

(3) Let us check the conditions \eqref{ss1-2},
\eqref{ss3-2},\eqref{ss2-3},\eqref{ss4-3}.
The condition \eqref{ss1-2} is equivalent to \eqref{BL4-3}.
The condition \eqref{ss3-2} holds 
because $Q=\Gamma_1=O_n$.  
\\
The condition \eqref{ss2-3} : computing the LHS 
by the degree formula \eqref{degree}, 
\begin{equation}\nonumber
\text{LHS}=\frac{1}{d_1}E_{\mathrm{deg}}(C_{\alpha\beta}^{\gamma})
=\frac{1}{d_1}(\mathrm{deg} \,C_{\alpha\beta}^{\gamma})C_{\alpha\beta}^{\gamma}
=\frac{d_1+d_{\gamma}-d_{\alpha}-d_{\beta}}{d_1}
C_{\alpha\beta}^{\gamma}~.
\end{equation}
On the other hand, since ${W^{\gamma}}_{\beta}=
\frac{d_{\beta}}{d_1}{\delta_{\gamma}}_{\beta}$,
\begin{equation}\nonumber
\text{RHS}=-\frac{d_{\alpha}}{d_1}C_{\alpha\beta}^{\gamma}
+\frac{d_{\gamma}-d_{\beta}}{d_1}C_{\alpha\beta}^{\gamma}
+C_{\alpha\beta}^{\gamma}=\text{LHS}~.
\end{equation}
The calculation of the condition \eqref{ss4-3} is similar:
\begin{equation}\nonumber
\text{LHS}=\frac{1}{d_1}E_{\mathrm{deg}}(\Gamma_{\alpha\beta}^{\gamma})=\frac{1}{d_1}\mathrm{deg}\, \Gamma_{\alpha\beta}^{\gamma}=\frac{d_{\gamma}-d_{\alpha}-d_{\beta}}{d_1}
\Gamma_{\alpha\beta}^{\gamma}~.
\end{equation}
On the other hand, 
\begin{equation}\nonumber
\text{RHS}=-\frac{d_{\alpha}}{d_1}\Gamma_{\alpha\beta}^{\gamma}
+\frac{d_{\gamma}-d_{\beta}}{d_1}\Gamma_{\alpha\beta}^{\gamma}
=\text{LHS}~.
\end{equation}
Next we show that there exist $\nabla$-flat coordinates 
$t^1,\ldots,t^n$ which are basic invariants.
Recall that the matrices $\Gamma_{\mu} $ ($1\leq \mu\leq n$)
are strictly upper triangular (see \eqref{BL2-1}).
Therefore there exists a unique upper triangular matrix 
$X\in M(n,\mathbb{C}[x])$
satisfying
\begin{equation}\nonumber
{X^{\gamma}}_{\gamma}=1~,\quad
\mathrm{deg}\,{X^{\gamma}}_{\beta}=d_{\gamma}-d_{\beta}~,\quad
\frac{\partial}{\partial x^{\mu}}  X+\Gamma_{\mu}X=O_n \quad (1\leq \mu\leq n)~.
\end{equation}
(One can solve the system of differential equations starting from 
$(\gamma,\gamma+1)$ entries and then 
moving to $(\gamma,\gamma+2)$ entries, and so on. 
The integrability of the equations follows from 
the flatness of $\nabla$.)  
It is clear that $X$ is invertible and that 
$X^{-1}\in M(n,\mathbb{C}[x])$.
Then unique homogeneous solutions 
$t^1,\ldots,t^n$ of 
the equations
$$
dt^{\alpha}=\sum_{\beta=1}^n{(X^{-1})^{\alpha}}_{\beta}\,dx^{\beta}
\quad (1\leq\alpha\leq n)
$$
are $\nabla$-flat coordinates.
They are of the form 
$t^{\alpha}=x^{\alpha}+\text{(a polynomial in $x^{\beta}$ ($\beta>\alpha$}))$
and have degrees $\mathrm{deg}\,t^{\alpha}=d_{\alpha}$.
Thus  they are basic invariants.
The matrix representations of $\partial_{t^{\alpha}}\ast$ 
with respect to the basis $\partial_{t^1},\ldots,\partial_{t^n}$ are 
given by $X^{-1}(\sum_{\mu=1}^{n} {X^{\mu}}_{\alpha}C_{\mu})X$.
It is clear that their entries are polynomials in $x$.
\\
(4)  
The  matrix representation of $\frac{1}{d_1}E_{\mathrm{deg}}\ast$ 
is given by $U$ in Lemma \ref{BL4}. 
Therefore the dual almost Saito structure 
$(\boldsymbol{\nabla}^{(0,\frac{1}{d_1})},\star,\partial_1)$ with the parameter 
$\frac{1}{d_1}$
is defined on the subset 
where $U$ has rank $n$, that is,
on $M_0=M\setminus \{\Delta=0\}$
by \eqref{BL4-5}.
By \eqref{BL4-5},
the matrix representation $B_{\alpha}$
of the multiplication $\partial_{\alpha}\star$ is given by
$$
B_{\alpha}=U^{-1}C_{\alpha}=U^{-1}\Omega_1^{-1}(\Omega_{\alpha}-\Gamma_{\alpha})=D_1^{-1}(\Omega_{\alpha}-\Gamma_{\alpha})~.
$$
Therefore
by \eqref{newconnection1}, 
the matrix representation of the dual connection
$\boldsymbol{\nabla}^{(0,\frac{1}{d_1})}_{\alpha}$ ($1\leq \alpha\leq n$) is given by
\begin{equation}\nonumber
\begin{split}
\Gamma_{\alpha}+ \frac{1}{d_1}(I_n-
\mathrm{diag}[d_1,\ldots,d_n]-
 E_{\mathrm{deg}}\cdot\Gamma)B_{\alpha}
\stackrel{\eqref{BL2-2}}{=}
\Gamma_{\alpha}+D_1 B_{\alpha}
=\Omega_{\alpha}~.
\end{split}
\end{equation}
Thus $\boldsymbol{\nabla}^{(0,\frac{1}{d_1})}=\boldsymbol{\nabla}^V$.
Regularity of the pair
$(\boldsymbol{\nabla}^V,\partial_{x^1})$ 
is already shown in Lemma \ref{reg}.
\end{proof}

Theorem \ref{main-theorem} can be restated in the following way.
\begin{corollary}\label{main-corollary}
If a finite complex reflection group $G$ satisfies
the assumptions (i) and (ii), 
then
$(\boldsymbol{\nabla}^V,\partial_{x^1})$
makes a regular natural almost Saito structure for $G$.
Its dual Saito structure can be extended to 
a polynomial Saito structure  on the whole orbit
space $M$.

Moreover if $d_1 >d_2$, then $(\boldsymbol{\nabla}^V,\partial_{x^1})$ is a unique natural almost Saito structure for $G$
up to equivalence.
\end{corollary}
\begin{proof}
The uniqueness follows since
$\dim_{\mathbb{C}}\mathcal{X}_{-d_1}=1$
if $d_1>d_2$.
\end{proof}

\subsection{Basic derivations and relationship to 
Kato--Mano--Sekiguchi's work \cite{KatoManoSekiguchi2015}}\label{sec:IP}
First let us recall the definitions of the basic derivations, the codegrees and the discriminant matrix
(see \cite[Chapter 6]{OrlikTerao}).
For a finite complex reflection group $G$,
the set of $G$-invariant polynomial vector fields on $V$ is 
a free $\mathbb{C}[V]^G$-module of rank $n=\dim V$ 
(see Lemma 6.48 of {\it loc.cit}).
A homogeneous basis of this module is called a set of basic derivations for $G$.
The degrees of basic derivations are called 
the codegrees of $G$ and denoted 
$$0=d_1^*\leq d_2^*\leq \ldots \leq d_n^*~.$$
\begin{definition}
The discriminant matrix of $G$ with respect to 
a set of basic invariants $x^1,\ldots, x^n$ and 
a set of basic derivations $\tilde{X}_1,\ldots, \tilde{X}_n$
is the matrix $M=({M^{\alpha}}_{\beta})$
 (${M^{\alpha}}_{\beta}\in \mathbb{C}[x]$) defined by
$$
{M^{\alpha}}_{\beta}=dx^{\alpha}(X_{\beta})=X_{\beta}(x^{\alpha})~.
$$
Here $X_{\beta}=\omega_* \tilde{X}_{\beta}$ is a polynomial vector field
on $M$ corresponding to $\tilde{X}_{\beta}$
and $\omega:V\to M$ is the orbit map.
\end{definition}

Now assume that $G$ is  a finite complex reflection group satisfying 
the assumptions (i) and (ii) in \S \ref{sec:assumptions}.
Take a set of basic invariants $x^1,\ldots, x^n$ satisfying the assumption (i).
Let $(\nabla,\ast,\frac{1}{d_1}E_{\mathrm{deg}})$ be the natural 
Saito structure on $M$ with the unit $\partial_{x^1}$,
given in Theorem \ref{main-theorem}.
Define vector fields $X_1,\ldots, X_n$ on $M$ by
\begin{equation}\label{basic-derivations}
X_{\alpha}=\frac{1}{d_1}E_{\mathrm{deg}}
\ast \partial_{x^{\alpha}}~.
\end{equation}
Notice that $\mathrm{deg}\,X_{\alpha}=d_1-d_{\alpha}$
and notice also that $X_1=\frac{1}{d_1}E_{\mathrm{deg}}$.
In this setting, the following proposition holds.
\begin{proposition}\label{prop:IP}
(1) For each $1\leq \alpha\leq n$, 
there exists  unique homogeneous $G$-invariant
polynomial vector field $\tilde{X}_{\alpha}$ of degree 
$d_1-d_{\alpha}$
on $V$  such that 
$\omega_* \tilde{X}_{\alpha}=X_{\alpha}$.
Here 
$\omega:V\to M$ is the orbit map.
\\
(2) Moreover if $d_{\alpha}+d_{\alpha}^*=d_1$ $(1\leq \alpha\leq n)$ holds,
$\tilde{X}_1,\ldots,\tilde{X}_n$ are basic derivations for $G$
and the matrix $U$ defined in Lemma \ref{BL4} is the discriminant matrix
with respect to $x^1,\ldots, x^n$ and $\tilde{X}_1,\ldots, \tilde{X}_n$.
\end{proposition}
\begin{proof} 
(1) Since the matrix representation of $\frac{1}{d_1}E_{\mathrm{deg}}\ast$ is given by $U$,
\begin{equation}\label{discriminant-matrix}
X_{\alpha}=
\sum_{\beta=1}^n{U^{\beta}}_{\alpha}\,\partial_{x^{\beta}}
\stackrel{\eqref{BL4-5}}{=}
\sum_{\beta=1}^n{(\Omega_1^{-1}D_1)^{\beta}}_{\alpha}\partial_{x^{\beta}}~.
\end{equation}
Therefore its (local) lift to $V$ is 
$$
\sum_{\beta,i=1}^n{(\Omega_1^{-1}D_1)^{\beta}}_{\alpha}
\frac{\partial u^i}{\partial x^{\beta}}
\partial_{u^i}
=
\sum_{i=1}^n{(J^{-1}\Omega_1^{-1}D_1)^i}_{\alpha} \partial_{u^i}~.
$$
By Lemma \ref{IP}, entries of 
$J^{-1}\Omega_1^{-1}$ are polynomials in $u$.
Entries of $D_1$ are polynomials in $x$ by definition, and hence 
they are polynomials in $u$.
Therefore this is a polynomial vector field globally defined on $V$.
This is the $\tilde{X}_{\alpha}$ in the proposition.
Notice that $\tilde{X}_{\alpha}$ are 
independent on the complement of 
reflection hyperplanes of $V$
since the matrix $U$ is invertible on $M_0$.
\\
(2) If $d_{\alpha}+d_{\alpha}^*=d_1$ holds, 
$\mathrm{deg}\,\tilde{X}_{\alpha}=d_{\alpha}^*$.
Therefore $\tilde{X}_1,\ldots, \tilde{X}_n$ are
a set of basic derivations for $G$. 
By \eqref{discriminant-matrix}, 
it is immediate to see that 
entries of 
the discriminant matrix with respect to $x^1,\ldots, x^n$
and $X_1,\ldots, X_n$ are given by
 $$
 dx^{\alpha}(X_{\beta})={U^{\alpha}}_{\beta}~.
 $$
Therefore the discriminant matrix is $U$.  
\end{proof}

Now we will explain the relationship to 
Kato--Mano--Sekiguchi's result \cite[\S 6]{KatoManoSekiguchi2015}.

In the rest of this subsection, 
we assume that  $G$ is a duality group 
(see \S \ref{example-well-generated}  and 
\cite[\S 12.6]{LehrerTaylor}
for the definition).
Notice that Proposition \ref{prop:IP} holds for $G$
since any duality group satisfies the assumptions (i) and (ii)
and satisfies also the requirement in Proposition \ref{prop:IP} (2).
Let $x^1,\ldots, x^n$ be a set of basic invariants satisfying the 
assumption (i)
 and let $(\nabla,\ast,\frac{1}{d_1}E_{\mathrm{deg}})$
be the natural Saito structure on $M$. 
(Recall that a set of $\nabla$-flat basic invariants $t^1,\ldots, t^n$ exists
and that 
$t^{\alpha}$ is of the form
$t^{\alpha}=x^{\alpha}+\text{(a polynomial in $x^{\beta}$, $\beta>\alpha)$}$. 
See the proof of Theorem \ref{main-theorem}-(3).)
We can assume that 
the basic invariants  $x^1,\ldots, x^n$ are 
$\nabla$-flat coordinates.

Denote by
$(\boldsymbol{\nabla}^V)^*$ the connection on 
the cotangent bundle $T^*M_0$
which is dual to
the natural connection $\boldsymbol{\nabla}^V$.
Given that $\boldsymbol{\nabla}^V$ is induced from 
the trivial connection on $TV$,
it is clear that
$(\boldsymbol{\nabla}^V)^*$
has local flat sections $du^1,\ldots, du^n$
where $u^1,\ldots, u^n$ are the standard coordinates of $V
=\mathbb{C}^n$.

Let $X_{\alpha}=\frac{1}{d_1}E_{\mathrm{deg}}\ast \partial_{\alpha}$ for $1\leq \alpha\leq n$
as in \eqref{basic-derivations}.
They form a basis of $\mathcal{T}_{M_0}$
since $\frac{1}{d_1}E_{\mathrm{deg}}\ast $ is invertible on $M_0\subset M$.
So denote the dual basis 
of $\mathcal{T}^*_{M_0}$ by $X^1,\ldots, X^n$. 
Let us write down the flatness equation for $du^i$
($1\leq i\leq n$) using this basis.
\begin{lemma} 
For $1\leq i\leq n$,
\begin{equation}\label{kms}
\begin{split}
&(\boldsymbol{\nabla}^V)^*(du^i)=0
\\
 \Leftrightarrow\quad &
\partial_{x^{\alpha}} y^i=y^i \,
U^{-1}C_{\alpha}
\left(\left(\frac{1}{d_1}+1\right)I_n-\frac{1}{d_1}\mathrm{diag}(d_1,\ldots, d_n)\right)
\quad (1\leq \alpha\leq n)~.
\end{split}
\end{equation}
Here
$y^i$
is the row vector $(X_1(u^i),\cdots,X_n(u^i))$,
$C_{\alpha}$ $(1\leq \alpha\leq n)$
and $U=\frac{1}{d_1}\sum_{\alpha=1}^n d_{\alpha}x_{\alpha}C_{\alpha}$ 
are the matrix representations of 
the multiplications $\partial_{\alpha}\ast$ 
and $\frac{1}{d_1}E_{\mathrm{deg}}\ast$
with respect to the basis $\partial_{x^1},\ldots, \partial_{x^n}$.
\end{lemma}
\begin{proof}
It is enough to show that 
the connection matrix  of 
$\boldsymbol{\nabla}^V_{\partial_{\alpha}}$ 
with respect to the basis $X_1,\ldots, X_n$
is the matrix appearing in the RHS of \eqref{kms}.

Since $\boldsymbol{\nabla}^V$ 
is the connection of the dual 
almost Saito structure with parameter $r=\frac{1}{d_1}$,
the connection matrix $\Omega_{\alpha}$ of 
$\boldsymbol{\nabla}^V_{\partial_{\alpha}}$ 
with respect to $\partial_{x^1},\ldots,\partial_{x^n}$ is given by \eqref{newconnection1}:
\begin{equation}\nonumber
\Omega_{\alpha}=\left(
\frac{1}{d_1}I_n-
\frac{1}{d_1} \mathrm{diag}(d_1,\ldots, d_n)
\right)C_{\alpha}U^{-1}~.
\end{equation}
Then 
the connection matrix $\Omega_{\alpha}^X$ with respect to the basis
$X_1,\ldots, X_n$
is given by
\begin{equation}\nonumber
\begin{split}
\Omega_{\alpha}^X&=
U^{-1}\Omega_{\alpha}U+U^{-1}\partial_{\alpha}U
\\&
\stackrel{\eqref{ss2-2}}{=}
U^{-1}C_{\alpha}
\left(\left(\frac{1}{d_1}+1\right)I_n-\frac{1}{d_1} \mathrm{diag}(d_1,\ldots, d_n)\right)~.
\end{split}
\end{equation}
\end{proof}

It is not difficult to see that 
\eqref{kms} is nothing but
 the Okubo type equation  (102)
in 
\cite[Theorem 6.1]{KatoManoSekiguchi2015}
and that
$\nabla$-flat basic invariants
$x^{\alpha}$ ($1\leq \alpha\leq n$) here 
correspond to ``uniquely defined special $G$-invariant
homogeneous polynomials $F'_{n-\alpha}(u)$'' there.
Thus the polynomial Saito structure
for the duality group $G$ in {\it loc.cit.}
agrees with
the natural Saito structure obtained in this section .
\subsection{Relationship to Arsie--Lorenzoni's standard bi-flat structure}
In this subsection, 
we explain that Arsie--Lorenzoni's standard bi-flat structure 
obtained in \cite[\S 5]{Arsie-Lorenzoni2016} is the same as  
the natural Saito structure.

As in the previous section, 
we assume that $G$ is  a finite complex reflection group satisfying 
the assumptions (i) and (ii) in \S \ref{sec:assumptions}.
Take a set of basic invariants $x^1,\ldots, x^n$ satisfying the assumption (i).
Let $(\nabla,\ast,\frac{1}{d_1}E_{\mathrm{deg}})$ be the natural 
Saito structure on $M$ with the unit $\partial_{x^1}$,
given in Theorem \ref{main-theorem}.
Modifying if necessary, we assume that the basic invariants 
$x^1,\ldots, x^n$ are $\nabla$-flat coordinates.

By \eqref{newconnection1}, 
the matrix representation $B_{\alpha}$ 
of $\partial_{x^{\alpha}}\star$
and the connection matrix $\Omega_{\alpha}$ of 
the dual conenction $\boldsymbol{\nabla}^V$
with respect to the basis $\partial_{x^1},\ldots,\partial_{x^n}$
satisfies the relation
\begin{equation}\nonumber
\begin{split}
\Omega_{\alpha\beta}^{\gamma}
&=
{\left(\frac{1}{d_1}\big(I_n-
\mathrm{diag}(d_1,\ldots, d_n)
\big)B_{\alpha} \right)^{\gamma}}_{\beta} 
~.
\end{split}
\end{equation}
Therefore
\begin{equation}\nonumber
\begin{split}
B_{\alpha\beta}^{\gamma}&=
{\left(d_1\big(I_n-\mathrm{diag}(d_1,\ldots, d_n)\big)^{-1}\Omega_{\alpha} \right)^{\gamma}}_{\beta} 
\\
&=\frac{d_1}{d_{\gamma}-1}
\sum_{i,j=1}^n \frac{\partial u^i}{\partial x^{\alpha}}
\frac{\partial u^j}{\partial x^{\beta}}
\frac{\partial^2 x^{\gamma}}{\partial u^i \partial u^j}~,
\end{split}
\end{equation}
and 
\begin{equation}\nonumber
\begin{split}
\partial_{u^i}\star \partial_{u^j}
&=\sum_{\alpha,\beta,\gamma=1}^n
\sum_{k=1}^n
\frac{\partial u^i}{\partial x^{\alpha}}
\frac{\partial u^j}{\partial x^{\beta}}
B_{\alpha\beta}^{\gamma}
\frac{\partial u^k}{\partial x^{\gamma}}
\partial_{u^k}
\\
&=
\frac{d_1}{d_{\gamma}-1}
\sum_{k=1}^n
\frac{\partial^2 x^{\gamma}}{\partial u^i \partial u^j}
\frac{\partial u^k}{\partial x^{\gamma}}
\partial_{u^k}~.
\end{split}
\end{equation}
This agrees with the formula of the structure constants
in \cite[Theorem 5.3]{Arsie-Lorenzoni2016}. 
\section{The Case of branched covering}
\label{covering-case}
\subsection{Setting}
In this section,
$G$ is a finite complex reflection group acting on 
an $n$-dimensional complex vector spaces $V=\mathbb{C}^n$
and $K$ is a normal reflection subgroup of $G$.
Denote the degrees of $K$ and $G$ by $d_1^K\geq d_2^K\geq \ldots \geq d_n^K$
and $d_1^G\geq d_2^G\geq \ldots \geq d_n^G$ respectively.
Since $K$ is a normal subgroup,
the action of $G$ on $\mathbb{C}[V]$ preserves the
subring $\mathbb{C}[V]^K$ of
$K$-invariant polynomials
and $\mathbb{C}[V]^G=\left( \mathbb{C}[V]^K\right)^{G/K}$.
Hence $G/K$ acts on the orbit space $M_K=\mathrm{Spec}\,\mathbb{C}[V]^K$
as automorphisms of the branched covering map $\pi:M_K\rightarrow M_G=\mathrm{Spec}\, \mathbb{C}[V]^G$.
We write $\Delta_K$ (resp. $\Delta_G$) for the discriminant of $K$ (resp. $G$). 
Let $\Delta_{G/K}:=\Delta_G/\Delta_K\in \mathbb{C}[V]^K$.
\begin{lemma}
The ramification locus of $\pi$ is $\{\Delta_{G/K}=0\}$. 
\end{lemma}\label{lem:covering}
\begin{proof}
Let $p\in V$
and let $K_p=\{g\in K\mid g(p)=p\}$, 
$G_p=\{g\in G\mid g(p)=p\}$ be the stabilizer subgroups.
Then $\omega_K(p)\in M_K$ is contained in the ramification locus
of $\pi$ if and only if $K_p\subsetneq G_p$. 
Here $\omega_K:V\to M_K$ is the orbit map. 
By the theorem of Steinberg (see \cite[\S 9.7]{LehrerTaylor}),
$K_p$ (resp. $G_p$) is generated by reflections 
of $K$ (resp. $G$) which fixes $p$.
Therefore $K_p\subsetneq G_p\Leftrightarrow 
(\delta_G/\delta_K)(p)=0$ (see \eqref{def:delta}).
\end{proof}
Let us set $M_K'=M_K\setminus \{\Delta_{G/K}=0\}$
and $M_G'=\mathrm{Im}\, \pi|_{M'_K}$.
Therefore $\pi|_{M_K'}:M_{K}'\to M_{G}'$  is 
an unramified Galois covering with the Galois group $G/K$.
Moreover, let us set
$M_{K,0}'=M_K\setminus \{\Delta_G=0\}$ and
$M_{G,0}=M_G\setminus\{\Delta_G=0\}$. 
Then  it follows that
$\pi|_{M_{K,0}'}:
M_{K,0}'\to M_{G,0}$
is an unramified Galois covering 
with the Galois group $G/K$.
In the next Definition \ref{def:invariance},
either $(U_K,U_G)=(M_K',M_G')$ or $(M_{K,0}',M_{G,0})$.

\begin{definition} \label{def:invariance}
(1)
A connection $\nabla$ on $TU_K$
is $G/K$-invariant if 
$$
\nabla_{{g}_{\ast}x}({g}_{\ast}y)
={g}_{\ast} (\nabla_{x}\,y)
\quad (g\in G/K~,~~x,y\in \mathcal{T}_{U_K})~.
$$
(2)
A multiplication $\ast$ on $TU_K$
is $G/K$-invariant if 
$$
({g}_{\ast}x)\ast ({g}_{\ast}y)=
{g}_{\ast}(x\ast y)\quad 
(g\in G/K~,~~x,y\in \mathcal{T}_{U_K})~.
$$
(3) 
We say that
a Saito structure $(\nabla,\ast,E)$ on $U_K$ is 
$G/K$-invariant if 
$\nabla,\ast,E$ are $G/K$-invariant.
We also say that an almost Saito structure $(\boldsymbol{\nabla},\star,e)$ on $U_K$ is $G/K$-invariant if
$\boldsymbol{\nabla},\star,e$ are $G/K$-invariant. 
\end{definition}
If $\nabla$ is a $G/K$-invariant connection on $TU_K$,
then the covering map $\pi$ induces
a connection $\nabla^{\pi}$ on $TU_G$ as follows.
Let $p\in U_G$ and take an open neighborhood
$U$ of $p$
such that
$\pi^{-1}(U)$ consists of $|G/K|$ disjoint open sets,
say $U_1,\ldots, U_{|G/K|}$. 
Then $\pi:U_i\to U$ 
is an isomorphism ($1\leq i\leq |G/K|$).
For $x,y\in \Gamma(U,\mathcal{T}_{U_G})$,
take their lifts $x_1,y_1\in \Gamma(U_1,\mathcal{T}_{U_K})$
to $U_1$ and define ${\nabla^{\pi}}$
by
\begin{equation}\nonumber
{\nabla}^{\pi}_{x}\,y=
\pi_{\ast} ({\nabla}_{x_1}\,y_1)~.
\end{equation}
This definition does not depend on the 
choice of the lifts. For, if $x_2,\, y_2\in \Gamma(U_2,T_{U_K})$ 
are the lifts of $x,y$ to $U_2$,
then there exists ${g}\in G/K$ such that 
${g}_{\ast} x_1=x_2$, ${g}_{\ast}y_1=y_2$.
Then $
{\nabla}_{x_2}\,y_2
=g_{\ast} ( {\nabla}_{x_1}\,y_1)$ follows from
the $G/K$-invariance.
Similarly, if $\ast$ is a $G/K$-invariant multiplication on $TU_K$,
$\pi$ induces a multiplication $\ast_{\pi}$ on $TU_G$ by
$$
x \ast_{\pi} y=\pi_*(x_1\ast y_1)~.
$$
Therefore,
given that the conditions for the Saito structures  and almost Saito structures are local conditions, it is easy to see the following statements.
\begin{enumerate}
\item
 If $(\nabla,\ast,E)$  is a $G/K$-invariant Saito structure on $U_K$,
then $\pi$ induces a Saito structure
$(\nabla^{\pi},\ast_{\pi},\pi_{*}E)$ on $U_G$.
\item 
If $(\boldsymbol{\nabla},\star,e)$ is a $G/K$-invariant almost Saito structure
on $U_K$,
then $\pi$ induces an almost Saito structure $(\boldsymbol{\nabla}^{\pi},\star_{\pi},\pi_{*}e)$  on $U_G$.
\item If the above $(\nabla,\ast,E)$ and $(\boldsymbol{\nabla},\star,e)$
are dual to each other,
then $(\nabla^{\pi},\ast_{\pi},\pi_{*}E)$ and
$(\boldsymbol{\nabla}^{\pi},\star_{\pi},\pi_{*}e)$ are dual to each
other.
\end{enumerate}

\subsection{Natural (almost) Saito structure via covering}
We first give criterions for a natural almost Saito structure
for $K$ to be $G/K$ invariant.

Denote $\boldsymbol{\nabla}^{V,K}$ (resp. $\boldsymbol{\nabla}^{V,G}$)
the natural connection 
on $TM_{K,0}$ (resp. $TM_{G,0}$)
induced from
the trivial connection  $d$ on $TV$  by the orbit maps $\omega_K,\omega_G$.
It is clear that $\boldsymbol{\nabla}^{V,K}$ is $G/K$-invariant
and that the $\pi$-induced connection 
$(\boldsymbol{\nabla}^{V,K})^{\pi}$ is nothing but 
$\boldsymbol{\nabla}^{V,G}$.
Notice also that $E_{\mathrm{deg}}^K
=\sum_{\alpha=1}^n d_{\alpha}^K y^{\alpha}\frac{\partial}{\partial y^{\alpha}}$ is $G/K$-invariant and 
$$
\pi_*E_{\mathrm{deg}}^K=\sum_{\alpha=1}^n d_{\alpha}^G x^{\alpha}\frac{\partial}{\partial x^{\alpha}}=E_{\mathrm{deg}}^G
~.
$$

\begin{lemma}
Assume that a vector field $e$ on $M_K$ is $G/K$-invariant.
Assume moreover that
the pair $(\boldsymbol{\nabla}^{V,K},e)$ is regular.
If $\star$ is the multiplication determined by $(${\bf ASS4}$)$ (or \eqref{mult}),
then $\star$ is $G/K$-invariant.
\end{lemma}

\begin{proof}
The $G/K$-invariance of $\star$ easily follows from 
those of $\boldsymbol{\nabla}^{V,K}$ and $e$.
\end{proof}

Let $k_K$ be the number of 
the degrees of $K$ which are equal to $d_1^K$:
$$d_1^K=\cdots=d_{k_K}^K>d_{k_K+1}^K\geq \ldots \geq d_n^K~.
$$
We put the following assumptions on the pair $(G,K)$.
\begin{enumerate}
\item[(iii)] either $k_K=1$ (i.e. $d_1^K>d_2^K$) or 
$G/K$ is abelian.
\item[(iv)]  $d_1^K=d_1^G$
\end{enumerate}
\begin{lemma}\label{lemma-semi-invariants}
If the assumption (iii) holds, 
there exists a set of basic invariants $y^1,\ldots, y^n$of $K$
such that $y^1,\ldots,y^{k_K}$ are  semi-invariants of $G$.
Moreover if the assumption (iv) holds, such $y^1,\ldots, y^{k_K}$ are 
$G$-invariant.
\end{lemma}
The proof of Lemma \ref{lemma-semi-invariants} 
is given in \S \ref{appendix-C}.
Notice that under the assumptions (iii) (iv),
any vector field $e$ on $M_K$ of degree $-d_1^K$
is $G/K$-invariant.

We arrive at the following 
\begin{proposition}\label{prop:covering-case}
Assume that $K$ satisfies the assumptions (i) and (ii)  in \S \ref{sec:assumptions}
and let $y^1,\ldots,y^n$ be a set of basic invariants  satisfying (ii). 
Denote $(\boldsymbol{\nabla}^{V,K},\partial_{y^1})$
the associated 
regular natural Saito structure.
Assume also that $(G,K)$ satisfies the assumptions (iii)(iv).
\\
(1) The projection $\pi:M_K\to M_G$ induces 
a regular natural almost Saito structure 
$(\boldsymbol{\nabla}^{V,G},\pi_*(\partial_{y^1}))$
for $G$.
Its parameter is $r=\frac{1}{d_1^G}$
and a unit is $E=\frac{1}{d_1^G}E_{\mathrm{deg}}^G$.
\\
(2) Denote by $(\nabla,\ast, \frac{1}{d_1}E_{\mathrm{deg}}^K)$
the natural Saito structure 
on $M_K$  which is dual to $(\boldsymbol{\nabla}^{V,K},\partial_{y^1})$.
Then $\pi$ induces a natural Saito structure 
$(\nabla^{\pi},\ast_{\pi},\frac{1}{d_1}E_{\mathrm{deg}}^G)$ 
on $M_G'$.
Moreover it
is dual to $(\boldsymbol{\nabla}^{V,G},\pi_*(\partial_{y^1}))$.
\end{proposition}

\section{Irreducible finite complex reflection groups}
In this section, notations are the same as in \S \ref{orbit-space} and \S \ref{covering-case}.
\subsection{Classification}
Let $m, n \in \mathbb{Z}_{>0}$ and let $p>0$ be a divisor of $m$.
Recall that the monomial group $G(m,p,n)$ 
is defined as
the semidirect product of 
$$
A(m,p,n)=\{(\theta_1,\ldots,\theta_n)\in 
\underbrace{\mu_m\times \ldots\times\mu_m}_{\text{$n$-times}}
\mid
(\theta_1\theta_2\ldots \theta_n)^{m/p}=1\},
$$
with the symmetric group $\mathfrak{S}_n$ 
acting by permutations of the factors.
Here $\mu_m$ is the cyclic group of $m$-th roots of unity.
The irreducible finite complex reflection groups 
are classified by Shephard--Todd \cite{ShephardTodd}:
\begin{itemize}
\item $G(1,1,n)=\mathfrak{S}_n$ for $n\geq 2$ 
regarded as acting on the $(n-1)$-dimensional invariant subspace
 i.e. the Weyl group of type $A_{n-1}$, 
\item $G(m,p,n)$ for $m>1$, $n>1$ and $(m,p,n)\neq (2,2,2)$,
\item $G(m,p,1)=G(m/p,1,1)=\mu_{m/p}$,  
\item 34 exceptional cases named 
$G_4,\ldots, G_{37}$.
\end{itemize}
$G(2,2,2)$ does not appear in the list
since $G(2,2,2)=A_1\times A_1$.

\subsection{Duality groups}\label{example-well-generated}
For an irreducible finite complex reflection group $G$,
the following conditions are equivalent 
\cite{Bessis2006} \cite{OrlikSolomon}.
\begin{itemize}
\item $G$ satisfies the assumptions (i) and (ii) in \S \ref{sec:assumptions}.
\item 
$ d_{\alpha}+d_{\alpha}^*=d_1$ ($1\leq \alpha\leq n$).
\item $G$ is generated by $n$ reflections.
\end{itemize}
Such $G$ is said to be well-generated, or $G$ is called a
duality group.
The duality groups are 
the monomial groups
$G(m,1,1)\cong \mu_m$ $(m\geq 1)$,
$G(1,1,n)=\mathfrak{S}_n\cong A_{n-1}$ ($n\geq 2$),
$G(m,1,n)$ ($m,n\geq 2$) , 
$G(m,m,n)$ ($m\geq 3$ and $n\geq 2$ or $m=2$ and $n\geq 3$),
and exceptional groups 
$G_4$ to $G_{37}$ except for
$G_7,G_{11},G_{12},G_{13},G_{15},G_{19},G_{22},G_{31}$.

If $G$ is one of these groups,  $d_1>d_2$. 
So $\dim \mathcal{X}_{-d_1}=1$ and 
a choice of the nonzero vector field $e\in \mathcal{X}_{-d_1}$
is unique up to scalar multiplication.
Moreover the pair
$(\boldsymbol{\nabla}^V,e)$ is regular. 
Therefore
a natural almost Saito structure (and hence a natural Saito structure)
exists uniquely up to equivalence by Corollary \ref{main-corollary}.
This was already proved by Kato, Mano and Sekiguchi \cite{KatoManoSekiguchi2015}.
See \S \ref{sec:IP} for how their result and our formulation
are related.

In \cite[\S 5]{Arsie-Lorenzoni2016},
Arsie and Lorenzoni computed
the bi-flat $F$-manifold structures
equivalent to our natural (almost) Saito structures
when $G$ is a duality groups of rank $\leq 3$.
They also found a universal formula 
\cite[Theorem 6.1]{Arsie-Lorenzoni2016}
for the multiplication  $\star$,
which is a generalization of the formula for 
the Coxeter case (see \cite[Eq. (5.21)]{Dubrovin2004}).

For the duality groups of rank two, 
we list the natural Saito structure  in 
Tables \ref{Gmpn}, \ref{typeT}, \ref{typeO}, \ref{typeI}~.
These results agree with the computations in \cite{Arsie-Lorenzoni2016}.
From these tables, we can see that some admit
a Frobenius structure, but some do not.

\subsection{$G(m,p,n)$ which are not duality groups}
The degrees of $G(m,p,n)$ are
$$m,\, 2m,\,\ldots,\, (n-1)m,\, \frac{nm}{p}.$$
In fact, $\sigma_i:=e_{i}\left((u^1)^m,\ldots, (u^n)^m\right)$ ($1\leq i\leq n-1$) 
together with $\sigma_{n}^{m,p}:=e_n(u_1,\ldots, u_n)^{m/p}$, where $e_k$ denotes the $k$-th elementary 
symmetric polynomial, 
form a set of basic invariants.   
If $p>1$ and $n>1$, the maximal degree is $(n-1)m$
which is of multiplicity one if $(p,n)\neq (2,2)$.
In the following, we consider the cases of irreducible $G(m,p,n)$'s which are not duality groups.

\subsubsection{The case when the maximal degree has multiplicity one}\label{sec:monomial-groups}
First, we assume that either $1<p<m$ and $n\geq 3$
or $2<p<m$ and $n=2$.
In these cases, 
a natural (almost) Saito structure  is unique if exists
since the maximal degree has multiplicity one.
The group $K=G(m,m,n)$ is a normal reflection 
subgroup of $G=G(m,p,n)$ and $G/K=\mu_{m/p}$.
If $p> 1$, they satisfy the assumptions (i)--(iv).  
We apply the construction of Proposition \ref{prop:covering-case}
to $(G, K)=(G(m,p,n), G(m,m,n))$. 
Let us set 
$$x^i:=
\begin{cases}
\sigma_{n-i} \quad (1\leq i \leq n-1)\\
\sigma_n^{m, p} \quad (i=n)
\end{cases},
\qquad 
y^i:=
\begin{cases}
\sigma_{n-i} \quad (1\leq i \leq n-1) \\
\sigma_n^{m,m}\qquad (i=n)
\end{cases}.
$$
We put $d_i^G:=\mathrm{deg}\,x^i$ and $d_i^K:=\mathrm{deg}\,y^i$.
The map 
$$\pi: M_K=\mathrm{Spec}\, \mathbb{C}[y_1,\ldots, y_{n}]
\to M_G=\mathrm{Spec}\, \mathbb{C}[x_1,\ldots, x_{n}]$$
defined by $x^i=y^i$ ($1\leq i \leq n-1$) and $x^n=(y^n)^{m/p}$ 
is a $\mu_{m/p}$-covering branched at $\{x^n=0\}$.
Let $(\nabla,\ast, \frac{1}{d_1}E_{\mathrm{deg}}^K)$
be the natural polynomial Saito structure 
on $M_K$  and 
$(\nabla^{\pi},\ast_{\pi},\frac{1}{d_1}E_{\mathrm{deg}}^G)$ 
be the natural Saito structure on $M_G'$ induced by $\pi$.
Then we have the following.\\

\begin{proposition}
(i) The multiplication $\ast_{\pi}$ is defined on the whole orbit space $M_G$.\\
(ii) The connection $\nabla^{\pi}$ has a logarithmic pole along $\{x^n=0\}$.
\end{proposition}

\begin{proof}
(i) Define $C_{i j}^{k}$ (resp. $(C_{\pi})_{\alpha \beta}^{\gamma}$) by
$$\partial_{y^{i}} \ast \partial_{y^{j}} =\sum_{k} C_{ij}^k \partial_{y^{k}}
\quad
\left( 
\mathrm{resp.}~~
\partial_{x^{\alpha}} \ast \partial_{x^{\beta}} =\sum_{\gamma} (C_{\pi})_{\alpha \beta}^{\gamma} \partial_{x^{\gamma}}
\right)~.
$$
We show that $(C_{\pi})_{\alpha \beta}^{\gamma}$ is a polynomial 
in $x$. It is easy to see that
$$
(C_{\pi})_{\alpha \beta}^{\gamma} 
=
\sum_{i,j,k} 
C_{ij}^k
\dfrac{\partial y^i} {\partial x^{\alpha}} \dfrac{\partial y^j}{\partial x^{\beta}} \dfrac{\partial x^{\gamma}}{\partial y^k}
=\left\{
\begin{array}{ll}
C_{\alpha \beta}^{\gamma} & \mathrm{if}~~\alpha \neq n,~\beta\neq n,~\gamma\neq n
\\
C_{\alpha \beta}^{n} \frac{m}{p} (y^n)^{m/p-1} & \mathrm{if}~~\alpha \neq n,~\beta\neq n,~\gamma= n
\\
C_{n \beta}^{\gamma} \frac{1}{m/p} (y^n)^{1-m/p} & \mathrm{if}~~\alpha = n,~\beta\neq n,~\gamma\neq n
\\
C_{n \beta}^{n}  & \mathrm{if}~~\alpha = n,~\beta\neq n,~\gamma= n
\\
C_{n n}^{\gamma} \frac{1}{(m/p)^2} (y^n)^{2(1-m/p)} & \mathrm{if}~~\alpha = n,~\beta= n,~\gamma\neq n
\\
C_{n n}^{n} \frac{1}{m/p} (y^n)^{1-m/p} & \mathrm{if}~~\alpha = n,~\beta = n,~\gamma= n
\end{array}\right.
~.
$$
The omitted cases are equivalent to one of the above cases under the symmetry of $\alpha$ and $\beta$.
On the other hand, $\mu_{m/p}$-invariance of the multiplication $\ast$ implies that 
$(y^n)^k$ appears in the polynomial $C_{\alpha \beta}^{\gamma}$ only if 
$$k\equiv 
\left\{
\begin{array}{ll}
0& \mathrm{if}~~\alpha \neq n,~\beta\neq n,~\gamma\neq n
\\
1& \mathrm{if}~~\alpha \neq n,~\beta\neq n,~\gamma= n
\\
\frac{m}{p}-1& \mathrm{if}~~\alpha = n,~\beta\neq n,~\gamma\neq n
\\
0& \mathrm{if}~~\alpha = n,~\beta\neq n,~\gamma= n
\\
\frac{m}{p}-2& \mathrm{if}~~\alpha = n,~\beta= n,~\gamma\neq n
\\
\frac{m}{p}-1& \mathrm{if}~~\alpha = n,~\beta = n,~\gamma= n
\end{array}\pmod{\frac{m}{p}}~. \right.
$$
Comparing the above two equations, it is clear that $C_{\alpha \beta}^{\gamma}$ is a
polynomial in $x^n$ except for the case $\alpha = n,~\beta= n,~\gamma\neq n$. 
To see that $C_{nn}^{\gamma}$ ($\gamma\neq n$) is a polynomial, it is enough to 
show that $\mathrm{deg}~C_{nn}^{\gamma}>(\frac{m}{p}-2)d_n^K$.
This holds since
$$\mathrm{deg}~C_{nn}^{\gamma}-(\frac{m}{p}-2)d_n^K
=\left(d_1^K + d_{\gamma}^K -2d_n^K\right)-(\frac{m}{p}-2)d_n^K
=d_1^G + d_{\gamma}^G-d_n^G>0~.$$
Note that the last inequality holds for any $\gamma \neq n$ if and only if $p>1$ .
\\
(ii) Let $\Gamma_{i j}^{k}$ (resp. $(\Gamma^{\pi})_{\alpha \beta}^{\gamma}$)
be the Christoffel symbol of $\nabla$ (resp. $\nabla^{\pi}$) defined by
$$\nabla_{\partial_{y^{i}}} \partial_{y^{j}} =\sum_{k} \Gamma_{ij}^k \partial_{y^{k}}
\quad
\left( 
\mathrm{resp.}~~
\nabla^{\pi}_{\partial_{x^{\alpha}}} \partial_{x^{\beta}} =\sum_{\gamma} (\Gamma^{\pi})_{\alpha \beta}^{\gamma} \partial_{x^{\gamma}}
\right)~.
$$
We have
$$
(\Gamma^{\pi})_{\alpha \beta}^{\gamma} 
=
\sum_{i,j,k} 
\Gamma_{ij}^k
\dfrac{\partial y^i} {\partial x^{\alpha}} \dfrac{\partial y^j}{\partial x^{\beta}} \dfrac{\partial x^{\gamma}}{\partial y^k}
+\sum_{l} \dfrac{\partial^2y^l}{\partial x^{\alpha} \partial x^{\beta}} \dfrac{\partial x^{\gamma}}{\partial y^l}
=\Gamma_{\alpha \beta}^{\gamma}
\dfrac{\partial y^{\alpha}}{\partial x^{\alpha}} \dfrac{\partial y^{\beta}}{\partial x^{\beta}} \dfrac{\partial x^{\gamma}}{\partial y^{\gamma}}
+\dfrac{\partial^2y^{\gamma}}{\partial x^{\alpha} \partial x^{\beta}} \dfrac{\partial x^{\gamma}}{\partial y^{\gamma}}~.
$$
By the same argument as above, 
the $\mu_{m/p}$-invariance of $\nabla$ implies that the first term 
of the above equation is a polynomial except for the case $\alpha = n,~\beta= n,~\gamma\neq n$. 
But in fact one can show that $\Gamma_{nn}^{\gamma}=0$ for any $\gamma$, 
see Lemma \ref{lem:vanishing} below. 
So the first term is regular for any $\alpha$, $\beta$, $\gamma$.  
For the second term,  it is clear that the term is zero unless 
$\alpha=\beta=\gamma=n$ and
$$
\dfrac{\partial^2y^{n}}{\partial x^{n} \partial x^{n}} \dfrac{\partial x^{n}}{\partial y^{n}}=\frac{1}{x^n}~.
$$
Hence $(\Gamma^{\pi})_{\alpha \beta}^{\gamma}$ is regular unless
$\alpha=\beta=\gamma=n$ and $(\Gamma^{\pi})_{n n}^{n}$ has 
a logarithmic pole along $\{x^n=0\}$.
\end{proof}

\begin{lemma}\label{lem:vanishing}
Under the same notation as in the above proof, 
we have $\Gamma_{nn}^{\gamma}=0$ for any $\gamma$. 
\end{lemma}
\begin{proof}
Let $K=G(m,m,n)$ and $G=G(m,1,n)$.
Note that their discriminants are related as
$\Delta_G=x^n\Delta_K$
under $x^i=y^i$ ($1\leq i\leq n-1$) and $x^n=(y^n)^m$.
Let $(\Omega^K)_{\alpha \beta}^{\gamma}$
(resp. $(\Omega^G)_{\alpha \beta}^{\gamma}$)
be the Christoffel symbol of the natural connection 
$\boldsymbol{\nabla}^{V,K}$ (resp. $\boldsymbol{\nabla}^{V,G}$)
on $M_{K,0}$ (resp. $M_{G,0}$).
For $\gamma\neq n$, we have
$$(\Omega^G)_{n n}^{\gamma}
=\frac{1}{m^2} (y^n)^{2(1-m)}
(\Omega^K)_{n n}^{\gamma}.$$
So we have
$$
\Delta_G(\Omega^G)_{n n}^{\gamma}
=\frac{1}{m^2} (y^n)^{2-m}
\Delta_K
(\Omega^K)_{n n}^{\gamma}.$$
By Proposition \ref{AP}, we have $\Delta_G(\Omega^G)_{n n}^{\gamma}\in \mathbb{C}[x]$.
It follows that $\Delta_K(\Omega^K)_{n n}^{\gamma}$
is divisible by $(y^n)^{m-2}$.
It then follows that $\Gamma_{nn}^{\gamma}$ is also 
divisible by $(y^n)^{m-2}$, see \eqref{def-Gamma}.
In particular, if 
$\mathrm{deg}\, \Gamma_{nn}^{\gamma}=d_{\gamma}^K-2d_n^K<(m-2)d_n^K$
then $\Gamma_{nn}^{\gamma}=0$.
The above inequality certainly holds for $\gamma\neq n$
since it is equivalent to $d_{\gamma}^K<md_n^K=mn$.
Thus we have $\Gamma_{nn}^{\gamma}=0$ for $\gamma\neq n$.
By the degree reason, we also have $\Gamma_{nn}^{n}=0$.
\end{proof}

\begin{remark}
One can also consider the case $p=1$, i.e. the case
when $G=G(m,1,n)$ and  $K=G(m,m,n)$.
In this case too, the $\mu_{m}$-covering 
$\pi: M_K=\mathrm{Spec}\, \mathbb{C}[y_1,\ldots, y_{n}]
\to M_G=\mathrm{Spec}\, \mathbb{C}[x_1,\ldots, x_{n}]$
branched along $\{x^n=0\}$
induces a Saito structure on $M_G'=M_G\setminus \{x^n=0\}$. However this is not a 
natural Saito structure for $G$ since
the maximal degrees of $G$ and $K$ are different.
Moreover both the multiplication and the connection have a
logarithmic pole along  $\{x^n=0\}$.
\end{remark}
\subsubsection{The case when the maximal degree has multiplicity two: 
$G(2k,2,2)$ with $k>1$}\label{seq:G2k22}
The degrees of $G(2k, 2, 2)$ satisfy $d_1=d_2$,
so the space of  the vector fields $\mathcal{X}_{-d_1}$ 
of degree $-d_1$  is two-dimensional.
We can see that $(\boldsymbol{\nabla}^V,e)$
is regular for any nonzero $e\in \mathcal{X}_{-d_1}$.
We can also see that 
there exist only three lines $l_1,l_2,l_3\subset \mathcal{X}_{-d_1}$
such that $e\in l_i$ ($e\neq 0$) and $\boldsymbol{\nabla}^V$ make 
a natural almost Saito structure. 
They are the ones induced by 
branched covering maps $\pi:M_K\to M_{G(2k,2,2)}$
from the orbit spaces of certain duality groups $K$.
See Table \ref{table:2k-2-2}
for the covering maps.
The connections of the natural Saito structures are logarithmic along the 
ramification loci of the branched covering maps.
However,  one can check that all the multiplications are 
defined over the whole space $M_{G(2k,2,2)}$.

\subsection{Exceptional groups which are not duality groups}
\subsubsection{Nonexistence of natural Saito structures for
$G_{12},G_{13},G_{22}$ and $G_{31}$}
\label{example-nonexistence}
Let $G$ be one of 
$G_{12},G_{13},G_{22}$ (of rank $n=2$)
and $G_{31}$ (of rank $n=4$). 
The degrees satisfy $d_1>d_2$.
So a choice of 
a nonzero vector field $e=\partial_{x^1}$ 
of degree $-d_1$ is unique up to
scalar multiplication.
The discriminant $\Delta$ of $G$ is monic of degree $n+1$ in $x^1$ (and hence $G$ is not a duality group).

The pair
$(\boldsymbol{\nabla}^V,e) $ is regular
since 
the matrix representation $\Omega_1$ 
of $\boldsymbol{\nabla}^V e$ is invertible on $M_0$.
Then the multiplication $\star$ on $TM_0$ defined by ({\bf ASS4}) or \eqref{mult}
is represented by matrices (see \eqref{ass4-3})
\begin{equation}\nonumber
B_{\alpha}=-\Omega_1^{-1}\partial_{x^1} \Omega_{\alpha}
\quad (1\leq \alpha\leq n)~.
\end{equation}
We can check that these $B_{\alpha}$'s  do not satisfy
({\bf ASS2}) (or \eqref{ass2-3}).
Thus by Lemma \ref{uniqueness2},
$G$ does not admit any natural almost Saito structure,
hence it does not admit any natural Saito structure.

\subsubsection{The case of $G_{15}$}
\label{example-covering1}
The group $G_{15}$ satisfies the condition $d_1>d_2$.
Therefore a natural almost Saito structure  is unique if exists.
It contains $G_{14}$ as a normal subgroup
and the pair $(G,K)=(G_{15},G_{14})$ satisfy  the assumptions (i)--(iv).
See Table \ref{table:G15}.
Therefore by Proposition \ref{prop:covering-case},
the branched covering map $\pi:M_{G_{14}}\to M_{G_{15}}$ induces a unique natural 
almost Saito structure and a unique natural Saito structure for $G_{15}$
from that for $G_{14}$.
So the situation is similar to the cases in \S \ref{sec:monomial-groups}.
One can check that the connection of 
the natural Saito structure for $G_{15}$ is logarithmic along the 
branch locus of $\pi$ but
the multiplication is defined all over the orbit space $M_{G_{15}}$.

\subsubsection{The cases of $G_7,G_{11}$ and $G_{19}$}\label{example-covering2}
These rank $2$ groups are not duality groups.
Let $G$ be one of the three groups.
The degrees of $G$ satisfy $d_1=d_2$. 
We can see that $(\boldsymbol{\nabla}^V,e)$
is regular for any nonzero $e\in \mathcal{X}_{-d_1}$.
So the situation is completely analogous to 
that of $G(2k,2,2)$ with $k>1$ (cf. \S \ref{seq:G2k22}).
One can show that 
there exist only three lines $l_1,l_2,l_3$
in the $2$-dimensional space  $\mathcal{X}_{-d_1}$
such that $e\in l_i$ ($e\neq 0$) and $\boldsymbol{\nabla}^V$ make 
a natural almost Saito structure. 
They are the ones induced by 
branched covering maps $\pi:M_K\to M_G$
from the orbit spaces of certain duality groups $K$.
See Tables \ref{table:G7}, \ref{table:G11},
\ref{table:G19}. 
The connections of the natural Saito structures 
are logarithmic along the 
branch loci but the multiplications are
defined over the whole space $M_G$.

\appendix
\section{Proofs of Proposition \ref{AP} and Lemma \ref{IP}}
\label{appendix:proofs}
In this section, notations are the same as in \S \ref{CRGs}.

Let us fix a reflection hyperplane of $G$, say, $H_1\in \mathcal{A}$
and set $e_1=e_{H_1}$, the order of the cyclic subgroup  $G_{H_1}$
preserving $H_1$ pointwise.
Let us fix a basis $\bm{e}_1,\ldots,\bm{e}_n$ of $V$
as follows: $\bm{e}_1$ is an eigenvector of $g\in G_{H_1}$ ($g\neq \mathrm{Id}$)
whose eigenvalue is not one; $\bm{e}_2,\ldots,\bm{e}_n$ form a basis
of $H_1$ which is the eigenspace  of $g$ with eigenvalue one.
Denote by  $(v_1,\ldots,v_n)$
the associated complex coordinates of $V$. 
It is clear that $L_{H_1}=\text{const.}\, v_1$.

\subsection{Lemmas}
Let us set
$$
\Pi'=\prod_{H\in \mathcal{A}\setminus\{H_1\}} L_H^{e_H-1}
=\frac{\Pi}{L_{H_1}^{e_1-1}}~.
$$
\begin{lemma}\label{AL1}
$\Pi'$ is $G_{H_1}$-invariant. 
\end{lemma}
\begin{proof}
The action of  $g\in G_{H_1}$ can be written, with
some $e_1$-th root of unity $\mu$,  as 
$g \bm{e}_1=\mu \bm{e}_1$ and 
$g\bm{e}_i= \bm{e}_i$ $ (i\neq 1)$.
Therefore $g(v_1)=\mu^{-1} v_1$, $g(v_i)=v_i$ ($i\neq 1$).
Since $\Pi$ is skew-invariant, $g(\Pi)=\mu \Pi$. 
Therefore $g(\Pi')=g(\Pi)/g(L_{H_1})^{e_1-1}=g(\Pi')$.
\end{proof}

For the sake of convenience, 
we introduce the notation
\begin{equation}\nonumber
f\succ g\stackrel{\mathrm{def.}}{\Longleftrightarrow}
\frac{f}{g}\in \mathbb{C}(v_2,\ldots,v_n)[v_1^{e_1}]
\quad \big(f,g\in \mathbb{C}(v_1,\ldots,v_n)\big)~.
\end{equation}
Note the following facts.
\begin{itemize}
\item If $f_1\succ g_1$ and $ f_2\succ g_2$
then $ f_1 f_2\succ g_1g_2$.
\item If $f\in \mathbb{C}[v]$ is a $G_{H_1}$-invariant polynomial,
then $f\succ 1$.
\item If  $f\succ 1$ then
\begin{equation}\nonumber
\frac{\partial f}{\partial v_1}\succ v_1^{e_1-1}~,\quad
\frac{\partial f}{\partial v_i}\succ 1\quad (i\neq 1)~.
\end{equation}
\item If $f\succ v^p$ ($p\neq 0$), then
\begin{equation}\nonumber
\frac{\partial f}{\partial v_1}\succ v_1^{p-1}~,\quad
\frac{\partial f}{\partial v_i}\succ v^p \quad (i\neq 1)~.
\end{equation}
\end{itemize} 

For $1\leq \alpha\leq n$ and $1\leq i\leq n$, we put
\begin{equation}\nonumber
w_{\alpha}^i=\frac{\partial v^i}{\partial x^{\alpha}}~,
\quad
z_{\alpha,j}^i=\frac{\partial w_{\alpha}^i}{\partial v^j}~,
\quad
z_{\alpha}=(z_{\alpha,j}^i)_{i,j}~.
\end{equation}
\begin{lemma}\label{AL3}
\begin{eqnarray} 
\label{AL3-1}
\frac{\partial x^{\alpha}}{\partial v_i}&\succ&
\begin{cases}
v_1^{e_1-1}&(i=1)\\1&(i\neq 1)
\end{cases}~.
\\\label{AL3-2}
w_{\alpha}^i &\succ&
 \begin{cases}v_1^{1-e_1}&(i=1)
\\1&(i\neq 1)
\end{cases}~.
\\\label{AL3-3}
z_{\alpha,j}^i &\succ& 
\begin{cases}
v_1^{-e_1}&(i=j=1)\\
v_1^{1-e_1}&(i=1,j\neq 1)\\
v_1^{e_1-1}&(i\neq 1,j=1)\\
1&(i\neq 1,j\neq 1)
\end{cases}~.
\\\label{AL3-4}
\det z_{\alpha} &\succ& v_1^{-e_1} ~.
\\\label{AL3-5}
{\left(z_{\alpha}^{-1}\right)^j}_i &\succ&
\frac{1}{\det z_{\alpha}}\times
\begin{cases}
1&(i=j=1)\\
v_1^{e_1-1}&(i=1,j\neq 1)\\
v_1^{1-e_1}&(i\neq 1,j=1)\\
v_1^{-e_1}&(i,j\neq 1)
\end{cases}
\end{eqnarray}
In \eqref{AL3-5}, we assume $\det z_{\alpha}\neq 0$.
\end{lemma}
\begin{proof}
\eqref{AL3-1} holds since $x^{\alpha}$ is a $G$-invariant 
polynomial.
\\
\eqref{AL3-2}: 
since the determinant of the Jacobian matrix is
proportional to $\Pi$, 
\begin{equation}\nonumber
\begin{split}
\frac{\partial v^i}{\partial x^{\alpha}}&=
\det\left(\frac{\partial x^{\alpha}}{\partial v^i}\right)^{-1}
\times \text{the $(\alpha,i)$ cofactor of }
 \left(\frac{\partial x^{\alpha}}{\partial v^i}\right)
\\&\succ 
\frac{1}{\Pi}\times \begin{cases} 1&(i=1)\\v_1^{e_1-1}&(i\neq 1)
\end{cases}
\\&\succ
\begin{cases}
{v_1^{1-e_1}}&(i=1)\\
1&(i\neq 1)
\end{cases}~~.
\end{split}
\end{equation}
\eqref{AL3-3} follows from \eqref{AL3-2}.
\eqref{AL3-4} and \eqref{AL3-5} follow from \eqref{AL3-3}.
\end{proof}

\begin{corollary} \label{AL4}
For $1\leq \alpha,\beta,\mu, \gamma\leq n$ and $1\leq i,j\leq n$,
\begin{eqnarray}
\label{AL4-1}
\frac{\partial x^{\gamma}}{\partial v^i}z_{\alpha,j}^i w_{\beta}^j 
&\succ &
\begin{cases}
{v_1^{-e_1}}&(i=j=1)\\
1&((i,j)\neq (1,1))
\end{cases}~.
\\\label{AL4-2}
\frac{\partial x^{\gamma}}{\partial v^j}{(z_{\mu}^{-1})^j}_i w_{\beta}^i
&\succ &
\frac{1}{\det z_{\mu}}\times
\begin{cases}
1&(i=1 \text{ or }j=1)\\
{v_1^{-e_1}}&(i\neq 1\text{ and }j\neq 1)
\end{cases}
\\\label{AL4-3}
\frac{\partial x^{\gamma}}{\partial v^j}
{(z_{\mu}^{-1})^j}_iz_{\alpha,k}^i w_{\beta}^k 
&\succ& \frac{1}{\det z_{\mu}}\times 
\begin{cases}
{v_1^{-e_1}}&(i=k=1\text{ or } i\neq 1 \text{ and } j\neq 1)
\\
1&(\text{else})
\end{cases}~
\\\label{AL4-4}
{(z_{\mu}^{-1})^j}_i {w^i}_{\gamma}
&\succ&
\begin{cases}
v_1&(j=1)\\1&(j\neq 1)
\end{cases}
\end{eqnarray}
In \eqref{AL4-2} \eqref{AL4-3} \eqref{AL4-4},
 we assume $\det z_{\mu}\not \equiv 0$.
\end{corollary}

\subsection{Proof of Proposition \ref{AP}}
Notice that 
$\Omega_{\alpha\beta}^{\gamma}$
can have poles only along reflection hyperplanes and that
\begin{equation}\nonumber
\Omega_{\alpha\beta}^{\gamma}=\sum_{i,j=1}^n
\frac{\partial x^{\gamma}}{\partial v^i}z_{\alpha,j}^i w_{\beta}^j~,
\quad 
\det \Omega_{\alpha}=\det z_{\alpha}~,\quad
\delta\succ v_1^{e_1}~.
\end{equation}
\eqref{AP1}: 
by \eqref{AL4-1}, 
$\delta\cdot \Omega_{\alpha\beta}^{\gamma}$ does not have a pole along
the hyperplane ${H_1}$.  Given that $H_1\in \mathcal{A}$ can be taken arbitrarily,   
$\delta\cdot \Omega_{\alpha\beta}^{\gamma}$ does not have a pole along any reflection hyperplanes.
This implies that $\delta\cdot\Omega_{\alpha\beta}^{\gamma}$ 
is a polynomial in $v$. Since both $\delta$ and $\Omega_{\alpha\beta}^{\gamma}$ are $G$-invariant,
so is
$\delta \cdot \Omega_{\alpha\beta}^{\gamma}$.
\\
\eqref{AP2}:  by \eqref{AL3-4},  
$\delta \det \Omega_{\alpha}=\delta\det z_{\alpha}$
does not have a pole along $H_1$. 
Given that $H_1$ is taken arbitrary,
$\delta \det \Omega_{\alpha}$ is a polynomial in $v$.
Since both $\delta$ and $\det\Omega_{\alpha}$ are $G$-invariant,
so is $\delta \det\Omega_{\alpha}$. 
Since
$\mathrm{deg}\,(\delta\cdot \det \Omega_{\alpha})=
\mathrm{deg}\,\delta-nd_{\alpha}
$,
$\delta\cdot \det \Omega_{\alpha}$ is a constant
if $\mathrm{deg}\,\delta=nd_{\alpha}$.
 \\
\eqref{AP3} \eqref{AP4}:  notice that
$$
\sum_{i,j=1}^n\frac{\partial x^{\gamma}}{\partial v^j}{(z_{\mu}^{-1})^j}_iw_{\beta}^i
={(\Omega_{\mu}^{-1})^{\gamma}}_{\beta}~,
\quad
\sum_{i,j,k=1}^n
\frac{\partial x^{\gamma}}{\partial v^j}
{(z_{\mu}^{-1})^j}_iz_{\alpha,k}^i w_{\beta}^k 
={(\Omega_{\mu}^{-1}\Omega_{\alpha})^{\gamma}}_{\beta}~~.
$$
If $\mathrm{deg}\,\delta=nd_{\mu}$ and $\det \Omega_{\mu}\neq 0$,
then by \eqref{AP2},
$$
\det z_{\mu}=\det \Omega_{\mu}=
\frac{\text{nonzero const.}}{\delta}~.
$$
Then  \eqref{AL4-2} \eqref{AL4-3} imply that
${(\Omega_{\mu}^{-1})^{\gamma}}_{\beta}$ and 
${(\Omega_{\mu}^{-1}\Omega_{\alpha})^{\gamma}}_{\beta}$
do not have poles along $H_1$.
So they do not have poles  along any reflection hyperplanes.
Therefore they are $G$-invariant polynomials.
This completes the proof of Proposition \ref{AP}.

\subsection{Proof of Lemma \ref{IP}}
Notice that 
$
\Omega_{\mu}J
=-Jz_{\mu}.
$
Therefore by Corollary \ref{AL4-4},
each entry of $J^{-1}\Omega_{\mu}^{-1}=-z_{\mu}^{-1}J^{-1}$
does not have pole along any hyperplane.
Thus it is a polynomial on $V$.
 
\section{Proof of Lemma \ref{lemma-semi-invariants}}
\label{appendix-C}
We use the same notation as \S \ref{covering-case}.
For the case when $G/K$ is abelian, 
let us show the next sub-lemma.
\begin{lemma}\label{sub-lemma1}
If $G/K$ is abelian, then 
there exists a set of basic invariants $y^1,\ldots, y^n\in \mathbb{C}[V]^K$ of $K$
such that each $y^{\alpha}$ is a semi-invariant of $G/K$.
\end{lemma}

\begin{proof}
First, we claim that 
there exists a set of basic invariants 
$z^1,\ldots, z^n$ of $K$ of $\mathrm{deg}\,z^{\alpha}=d_{\alpha}^K$
such that 
\begin{equation}
{g}(z^{\alpha})=\chi_{\alpha}(g) z^{\alpha}+\text{ a polynomial in $z^{\alpha+1},\ldots, z^n$}
\end{equation}
holds for some characters $\chi_{\alpha}$ of $G/K$.

Let $d$ be one of the degrees of $K$
and denote by $I_d\subset \{1,\ldots,n\}$  the set of integers $\alpha$
such that $\mathrm{deg}\,z^{\alpha}=d$.
Set $R_{d}:=\bigoplus_{\alpha\in I_{d}}\mathbb{C}z^{\alpha}$.
Since the action of $G$ preserves the grading on
$\mathbb{C}[V]^K=\mathbb{C}[z^1,\ldots,z^n]$, the action of 
${g}\in G/K$ 
on $z^{\alpha}$ ($\alpha\in I_d$)
can be  written as 
\begin{equation}\label{eq:action}
{g}(z^{\alpha})=\sum_{\beta\in I_d} A^{{g}}_{\alpha\beta}z^{\beta}+
\text{a polynomial in $z^{\gamma}$'s with $\mathrm{deg}\,z^{\gamma}<d$},
\quad (A^{{g}}_{\alpha\beta}\in \mathbb{C}).
\end{equation}
By the algebraic independence of $z^{\alpha}$, 
it follows that $R_d$ is a representation of $G/K$ via $A^g=(A^g_{\alpha\beta})$.
Since the group $G/K$ is finite and abelian,
we may assume that $z^{\alpha}$ ($\alpha\in I_{d}$)
are eigenvectors of $A^g$ for any $g\in G_K$. 
The claim follows from this.

Next, we put
\begin{equation}\label{eq:semi-inv}
y^{\alpha}=\frac{1}{|G/K|}\sum_{{g}\in G/K} \chi_{\alpha}({g})^{-1} {g}(z^{\alpha})~.
\end{equation}
Then it is immediate to show that 
${g}(y^{\alpha})=\chi_{\alpha}(g) y^{\alpha}$.
Moreover, we have
$$
y^{\alpha}=z^{\alpha}+\text{ a polynomial in $z^{\alpha+1},\ldots, z^n$}~.
$$
So $y^1,\ldots, y^n$ are algebraically independent and form
a set of basic invariants of $K$.
\end{proof}
\begin{lemma}\label{sub-lemma2}
If $d_1^K>d_2^K$, then there exists a set of basic invariants
$y^1,\ldots,y^n$ of $K$ of $\mathrm{deg}\,y^{\alpha}=d_{\alpha}^K$
such that 
$y^1$ is a semi-invariant of $G$.
\end{lemma}
\begin{proof}
If 
$z^1,\ldots, z^n$ are 
a set of basic invariants 
of $K$ of $\mathrm{deg}\,z^{\alpha}=d_{\alpha}^K$,
then 
\begin{equation}
{g}(z^1)=A^{{g}}z^{1}+\text{ a polynomial in $z^{2},\ldots, z^n$},
\quad (A^{{g}}\in \mathbb{C})
\end{equation}
holds for any ${g}\in G/K$
because the action of $G/K$ preserves the grading.
Then by the same argument as in Lemma \ref{sub-lemma1},
we can show that 
$\chi_1:G/K\to \mathbb{C}^{\ast}$ given by
$\chi_1({g})=A^{{g}}$ 
is a character of $G/K$.
Therefore if we take $y^1$
as in \eqref{eq:semi-inv},
then ${g}(y^1)=\chi_1({g})y^1$
and $y^1,z^2,\ldots,z^n$ form a set of basic invariants of $K$. 
\end{proof}

Now we prove Lemma \ref{lemma-semi-invariants}.
By the above lemmas, the first part was already proved.
Therefore assume $d_1^K=d_1^G$.
Let $k_G$ be the number of degrees of $G$
which are equal to $d_1^G$. 
Note that $k_G\geq k_K$ since $\mathbb{C}[V]^G\subset \mathbb{C}[V]^K
=\mathbb{C}[y^1,\ldots, y^n]$, where $y^1,\ldots,y^{k_K}$ are $G$-semi-invariants 
with characters $\chi_1,\ldots,\chi_{k_K}$.
Take a set of basic invariants $x^1,\ldots, x^n$
of $G$ with $\mathrm{deg}\,x^{\alpha}=d_{\alpha}^G$.
Express $x^1,\ldots, x^{k_G}$
as polynomials in $y^1,\ldots, y^n$:
\begin{equation}\nonumber
x^{\alpha}=\sum_{\beta=1}^{k_K} X_{\alpha\beta}y^{\beta} 
+\text{a polynomial in $y^{k_K+1},\ldots, y^n$},
\quad (X_{\alpha\beta}\in\mathbb{C})~.
\end{equation}
If the rank of 
the $k_G\times k_K$ matrix $X=(X_{\alpha\beta})$
is smaller than $k_K$,
then this contradicts the algebraic independence of $x$ and $y$.
Therefore $\mathrm{rank}\, X=k_K$ and
we may assume that $x^1,\ldots, x^{k_G}$ are chosen 
so that
 \begin{equation}\nonumber
 x^{\alpha}=
 \begin{cases}
y^{\alpha}+
\text{a polynomial in $y^{\alpha+1},\ldots, y^n$}
&(\text{if } 1\leq \alpha\leq k_K)~
\\
\text{a polynomial in $y^{k_K+1},\ldots, y^n$}
&(\text{if } \alpha> k_K)
\end{cases}~~~.
\end{equation}
Now look at the action of ${g}\in G/K$ on $x^{\alpha}$ ($1\leq \alpha\leq k_K$): 
\begin{equation}
x^{\alpha}={g}(x^{\alpha})=\chi_{\alpha}({g}) y^{\alpha}+
\text{a polynomial in $y^{\alpha+1},\ldots, y^n$}~.
\end{equation}
Then $\chi_{\alpha}({g})=1$ follows from
the algebraic independence of $y$.
Thus $y^1,\ldots, y^{k_K}$ are $G/K$-invariant.
\section{Rank two examples}\label{appendix:tables}
This section contains tables for
the irreducible finite complex reflection groups of rank two.
In Tables \ref{table:degrees-1}, \ref{table:degrees-2}, 
\ref{table:degrees-3}, \ref{table:degrees-4},
the degrees $d_1,d_2$,  the discriminant $\Delta$, a set of basic invariants $x,y$ are 
shown for each $G$.

For the duality groups,
we list a unique (up to equivalence) natural Saito structure
in Tables \ref{Gmpn}, \ref{typeT}, \ref{typeO}, \ref{typeI}.
Flat coordinates are denoted $t^1,t^2$ and 
$\partial_{t^1}$ is assumed to be a unit of  the multiplication
$\ast$. For the remaining groups
$G=G(m,p,2), G(2k,2,2)$ and $G_7,G_{11},G_{15},G_{19}$,
we list covering maps $\pi:M_K\to M_G$ which induce
natural (almost) Saito structures.
In the tables, $(x,y)$ denotes the set of basic invariants for $G$
listed in Tables \ref{table:degrees-1}, \ref{table:degrees-2}, 
\ref{table:degrees-3}, \ref{table:degrees-4}
and $(x',y')$ denotes the set of basic invariants for the normal
subgroup $K$. 

Let $i=\sqrt{-1}$,
$\zeta_m=e^{\frac{2\pi i}{m}}$, 
$\tau=\zeta_5+\zeta_5^{-1}+1$ and
\begin{equation}\nonumber
\begin{split}
\rho&=\begin{pmatrix}0&1\\1&0\end{pmatrix}
~,\quad
\tau_{m}=\begin{pmatrix}\zeta_{m}&0\\0&1\end{pmatrix}~,
\quad
\sigma_{m}=\tau_{m}^{-1}\rho\,\tau_{m}~,
\quad
r=\begin{pmatrix}1&0\\0&-1\end{pmatrix}~,
\\
r_1&=\frac{\zeta_3}{2}\begin{pmatrix}-1-i&1-i\\-1-i&-1+i \end{pmatrix}
,~~
r_2=\frac{\zeta_3}{2}\begin{pmatrix}-1+i&-1+i\\1+i&-1-i\end{pmatrix}
,\quad
s=\begin{pmatrix}\zeta_8^{-1}&0\\0&\zeta_8\end{pmatrix}
\\
r_3&=\frac{1}{\sqrt{2}}\begin{pmatrix}1&-1\\-1&-1\end{pmatrix}
, \quad 
r_4=\begin{pmatrix}1&0\\0&i \end{pmatrix}~,\quad
r_5=\frac{\zeta_5^2}{\sqrt{2}}
\begin{pmatrix}-\tau+i&-\tau+1\\\tau-1&-\tau-i\end{pmatrix}~.
\end{split}
\end{equation}

\begin{table}[h]
\begin{equation}\nonumber
\begin{array}{|cc| c|c| c|c| c |  c|}\hline
G&&\text{degrees}&\Delta&x&y&\text{duality group?}
\\\hline
 G(m,1,2)&(m\geq 2)&2m,m &u^mv^m(u^m-v^m)^2
 &u^mv^m& u^m+ v^m
&\text{yes}
\\\hline
G(m,m,2)&(m\geq 3)&m,2&(u^m-v^m)^2
&u^m+v^m&uv
&\text{yes}
\\\hline
G(kp,p,2)&(p>2,k>1)&kp,2k&u^kv^k(u^{kp}-v^{kp})^2
&u^{kp}+v^{kp}&u^kv^k
&\text{no}
\\\hline
G(2k,2,2)&(k>1)&2k,2k&u^kv^k(u^{2k}-v^{2k})^2
&u^{2k}+v^{2k}&u^kv^k
&\text{no}
\\\hline
\end{array}
\end{equation}
\caption{The monomial groups of rank two.}\label{table:degrees-1}
\end{table}
\begin{table}[h]
\begin{equation}\nonumber
\begin{array}{|c|c|c|c|c| c|}\hline
 G&\text{degrees}&x&y&\Delta
 &\text{duality group?} \\\hline
G_4
&6~,~4
&t_{\mathcal{T}} &f_{\mathcal{T}}
&h_{\mathcal{T}}^3 &\text{yes}
\\\hline
G_5
&12~,~6
&f_{\mathcal{T}}^3&t_{\mathcal{T}}
&  f_{\mathcal{T}}^3 ~h_{\mathcal{T}}^3 &\text{yes}
\\\hline
G_6
&12~,~4
&t_{\mathcal{T}}^2& f_{\mathcal{T}}
&h_{\mathcal{T}}^3~ t_{\mathcal{T}}^2 &\text{yes}
 \\\hline
G_7&12~,~~12
& f_{\mathcal{T}}^3&t_{\mathcal{T}}^2
&f_{\mathcal{T}}^3~h_{\mathcal{T}}^3~t_{\mathcal{T}}^2&\text{no}
 \\\hline
\end{array}
\end{equation}
\begin{equation}\nonumber
\begin{split}
f_{\mathcal{T}}(u,v)&={u^4+2i\sqrt{3}\,u^2 v^2+v^4}~,
\\
h_{\mathcal{T}}(u,v)&={u^4-2i\sqrt{3}\, u^2 v^2+v^4}~,
\\
t_{\mathcal{T}}(u,v)&=u^5 v-u v^5 ~.
\end{split}
\end{equation}
\caption{The exceptional groups of type $\mathcal{T}$.
Note 
$f_{\mathcal{T}}^3-h_{\mathcal{T}}^3=12i\sqrt{3} \,t_{\mathcal{T}}^2.$}\label{table:degrees-2}
\end{table}
\begin{table}[h]
\begin{equation}\nonumber
\begin{array}{|c| c| c| c| c| c| }\hline
G&\text{degrees}&x&y
&\Delta&\text{duality group?}
\\\hline
G_8&12~,~8
&t_{\mathcal{O}}
&h_{\mathcal{O}}
&f_{\mathcal{O}}^4
&\text{yes} 
\\\hline
G_9&24~,~8
&t_{\mathcal{O}}^2
&h_{\mathcal{O}}
&f_{\mathcal{O}}^4 ~t_{\mathcal{O}}^2
&\text{yes}
\\\hline
G_{10}&24~,~12
&h_{\mathcal{O}}^3
&t_{\mathcal{O}}
&f_{\mathcal{O}}^4~h_{\mathcal{O}}^3
&\text{yes}
\\\hline
G_{11}&24~,~24
&h_{\mathcal{O}}^3
&t_{\mathcal{O}}^2
&f_{\mathcal{O}}^4~h_{\mathcal{O}}^3~t_{\mathcal{O}}^2
&\text{no}
\\\hline
G_{12}&8~,~6
&h_{\mathcal{O}}
&f_{\mathcal{O}}
&t_{\mathcal{O}}^2
&\text{no}
\\\hline
G_{13}&12~,~8
&f_{\mathcal{O}}^2
&h_{\mathcal{O}}
&f_{\mathcal{O}}^2~t_{\mathcal{O}}^2
&\text{no}
\\\hline
G_{14}&24~,~6
&t_{\mathcal{O}}^2
&f_{\mathcal{O}}
&h_{\mathcal{O}}^3~t_{\mathcal{O}}^2
&\text{yes}
\\\hline
G_{15}&24~,~12
&t_{\mathcal{O}}^2
&f_{\mathcal{O}}^2
&f_{\mathcal{O}}^2~ h_{\mathcal{O}}^3 ~t_{\mathcal{O}}^2
&\text{no}
\\\hline
\end{array}
\end{equation}
\begin{equation}\nonumber
\begin{split}
f_{\mathcal{O}}&=u^5 v-u v^5~,\\
h_{\mathcal{O}}&=u^8+14u^4v^4+v^8~,\\
t_{\mathcal{O}}&=u^{12}-33u^8v^4-33u^4v^8+v^{12}~.
\end{split}
\end{equation}
\caption{The exceptional groups of type $\mathcal{O}$.
Note 
$h_{\mathcal{O}}^3-t_{\mathcal{O}}^2=108 f_{\mathcal{O}}^4~.$}
\label{table:degrees-3}
\end{table}

\begin{table}[h] 
\begin{equation}\nonumber
\begin{array}{|c|c|c|c|c|c|}\hline
G&\text{degrees}&x&y
& \Delta&\text{duality group?}
 \\\hline
G_{16}
&30~,~20
&t_{\mathcal{I}}
&h_{\mathcal{I}}
&f_{\mathcal{I}}^5
&\text{yes}
\\\hline
G_{17}
&60~,~20
&t_{\mathcal{I}}^2
&h_{\mathcal{I}}
&f_\mathcal{I}^5~t_\mathcal{I}^2
&\text{yes}
\\\hline
G_{18}
&60~,~30
&h_{\mathcal{I}}^3
&t_{\mathcal{I}}
&f_{\mathcal{I}}^5~h_{\mathcal{I}}^3
&\text{yes}
\\\hline
G_{19}
&60~,~60
&h_{\mathcal{I}}^3
&t_{\mathcal{I}}^2
&f_{\mathcal{I}}^5~h_{\mathcal{I}}^3~t_{\mathcal{I}}^2
&\text{no}
\\\hline
G_{20}
&30~,~12
&t_{\mathcal{I}}
&f_{\mathcal{I}}
&h_{\mathcal{I}}^3
&\text{yes}
\\\hline
G_{21}
&60~,~12
&t_{\mathcal{I}}^2 
&f_{\mathcal{I}}
&h_{\mathcal{I}}^3~t_{\mathcal{I}}^2
&\text{yes}
\\\hline
G_{22}
&20~,~12
&h_{\mathcal{I}}
&f_{\mathcal{I}}
&t_{\mathcal{I}}^2
&\text{no}
\\\hline
\end{array}
\end{equation}

\begin{equation}\nonumber
\begin{split}
f_{\mathcal{I}}&=
u^{12}+\frac{22 u^{10} v^2}{\sqrt{5}}-33 u^8 v^4-\frac{44 u^6 v^6}{\sqrt{5}}-33
   u^4 v^8+\frac{22 u^2 v^{10}}{\sqrt{5}}+v^{12}~,
\\
h_{\mathcal{I}}&=\frac{\sqrt{5}}{5808}\mathrm{hess}\,(f_{\mathcal{I}})=u^{20}+\cdots~,
\\
t_{\mathcal{I}}&=-\frac{1}{480\sqrt{5}}\frac{\partial(f_{\mathcal{I}},h_{\mathcal{I}})}{\partial(u,v)}
=u^{29} v+\cdots~,
\end{split}
\end{equation}
\caption{The exceptional groups of type $\mathcal{I}$.
Note that $ f_{\mathcal{I}}^5=h_{\mathcal{I}}^3+60\sqrt{5}\,t_{\mathcal{I}}^2$.}
\label{table:degrees-4}
\end{table}

\begin{table}[h]
\begin{equation}\nonumber
\begin{array}{|c|c|c|c|c|}\hline
G&t^1&t^2& \partial_{t^2}\ast\partial_{t^2}&\text{Frobenius manifold?}
\\\hline
G(m,1,2)
&u^mv^m-\frac{(u^m+v^m)^2}{4m}&u^m+v^m
&\frac{m-1}{4m^2}(t^2)^2\partial_{t^1}-\frac{m-2}{2m}t^2\partial_{t^2}
&\text{no if $m\neq 2$}
\\\hline
G(m,m,2)
&u^m+v^m&uv
&m^2(t^2)^{m-2}\partial_{t^1}
&\text{yes}
\\\hline
\end{array}
\end{equation}
\caption{
Natural Saito structures for  $G(m,1,2)$ and $G(m,m,2)$.
}
\label{Gmpn}
\end{table}

\begin{table}[h]
\begin{equation}\nonumber
\begin{array}{|c|c|c|c|c|}\hline
G&t^1&t^2&\partial_{t^2}\ast \partial_{t^2}&\text{Frobenius manifold?}
\\\hline
G_4& t_{\mathcal{T}}&f_{\mathcal{T}}
&-\frac{i\sqrt{3}}{16}t^2\,\partial_{t^1}
&\text{yes} 
\\\hline
G_5
&f_{\mathcal{T}}^3-6i\sqrt{3}t_{\mathcal{T}}^2
&t_{\mathcal{T}}
& -432 (t^2)^2\,\partial_{t^1}
&\text{yes}
\\\hline
G_6
&t_{\mathcal{T}}^2+\frac{5i}{96\sqrt{3}}{f_{\mathcal{T}}}^3
&f_{\mathcal{T}}
&-\frac{5 }{2^{10}}(t^2)^4\partial_{t^1}
-\frac{i}{16\sqrt{3}}(t^2)^2\partial_{t^2}
 &\text{no}
\\\hline
\end{array}
\end{equation}
\caption{
Natural Saito structure for the duality groups of type $\mathcal{T}$.
}\label{typeT}
\end{table}

\begin{table}
\begin{equation}\nonumber
\begin{array}{|c|c|c|c|c|}\hline
G&t^1&t^2&\partial_{t^2}\ast \partial_{t^2}&\text{Frobenius manifold?}\\\hline
G_8
& t_{\mathcal{O}}&h_{\mathcal{O}}
&\frac{9}{4}t^2\,\partial_{t^1}
&\text{yes}
\\\hline
G_9
& t_{\mathcal{O}}^2-\frac{11}{16}h_{\mathcal{O}}^3
&h_{\mathcal{O}}
&\frac{495}{2^8}(t^2)^4\partial_{t^1}+\frac{9}{8}(t^2)^2\partial_{t^2}
&\text{no}
\\\hline
G_{10}
&h_{\mathcal{O}}^3-\frac{7}{12}t_{\mathcal{O}}^2
&t_{\mathcal{O}}
&\frac{35}{36}(t^2)^2\partial_{t^1}+\frac{1}{3}t^2\partial_{t^2}
&\text{no}
\\\hline
G_{14}
&t_{\mathcal{O}}^2+66f_{\mathcal{O}}^4
&f_{\mathcal{O}}
&44352(t^2)^6\partial_{t^1}-96 (t^2)^3\partial_{t^2}
&\text{no}
\\\hline
\end{array}
\end{equation}
\caption{
Natural Saito structures for duality groups of type $\mathcal{O}$. 
}\label{typeO}
\end{table}

\begin{table}[h]
\begin{equation}\nonumber
\begin{array}{|c|c|c|c|c|}\hline
G&t^1&t^2&\partial_{t^2}\ast \partial_{t^2}&\text{Frobenius manifold?}\\\hline
G_{16}
&t_{\mathcal{I}}&h_{\mathcal{I}}
&-\frac{3}{80\sqrt{5}}t^2\,\partial_{t^1}
&\text{yes}
\\\hline
G_{17}
&t_{\mathcal{I}}^2+\frac{29}{2400\sqrt{5}}h_{\mathcal{I}}^3
&h_{\mathcal{I}}
&\frac{319}{32\cdot 10^5} (t^2)^4\partial_{t^1}-
\frac{9}{400\sqrt{5}}(t^2)^2\partial_{t^2}
&\text{no}
\\\hline
G_{18}
&h_{\mathcal{I}}^3+38\sqrt{5}t_{\mathcal{I}}^2
&t_{\mathcal{I}}
&16720(t^2)^2\partial_{t^1}-32\sqrt{5}t^2\partial_{t^2}
&\text{no}
\\\hline
G_{20}
&t_{\mathcal{I}}&f_{\mathcal{I}}
&\frac{\sqrt{5}}{48}(t^2)^3\partial_{t^1}
&\text{yes}
\\\hline
G_{21}
&t_{\mathcal{I}}^2-\frac{29}{2880\sqrt{5}}f_{\mathcal{I}}^5
&f_{\mathcal{I}}
&\frac{551}{2^{12}\cdot 3^4\cdot 5}(t^2)^8\partial_{t^1}
+\frac{\sqrt{5}}{288}(t^2)^4\partial_{t^2}
&\text{no}
\\\hline
\end{array}
\end{equation}
\caption{Natural Saito structures for the duality groups of type $\mathcal{I}$.}
\label{typeI}
\end{table}

\begin{table}[h]
\begin{equation}\nonumber
\begin{array}{|c|c|c|c|}\hline
e&K&K\rightarrow G(kp,p,2)&M_{K}\rightarrow M_{G(kp,p,2)}
\\\hline
\partial_x
&G(kp,kp,2)=\langle \sigma_{kp},\rho \rangle
&g\mapsto g
&(x',y')\mapsto(x',(y')^k)
\\\hline
\end{array}
\end{equation}
\caption{A covering map for
$G(kp,p,2)=\langle \sigma_{kp},\tau_{kp}^{p},\rho\rangle$ 
($p>2,k>1$) }\label{table:m-p-2}
\end{table}

\begin{table}[h]
\begin{equation}\nonumber
\begin{array}{|c|c|c|c|}\hline
e&K& K\rightarrow G(2k,2,2)
&M_{K}\rightarrow M_{G(2k,2,2)}
\\\hline
\partial_x
&G(2k,2k,2)=\langle \sigma_{2k},\rho \rangle
&g\mapsto g
&(x',y')\mapsto (x',(y')^k)
\\\hline
-2\partial_x+\partial_y
&G(k,1,2)=\langle \tau_{2k}^2,\rho \rangle
&g\mapsto g
&(x',y')\mapsto ((y')^2-2x',x')
\\\hline
-2\partial_x-\partial_y
&G(k,1,2)=\langle \tau_{2k}^2,\rho\rangle
&g\mapsto \tau^{-1}g\tau
&(x',y')\mapsto((y')^2-2x',-x')
\\\hline
\end{array}
\end{equation}
\caption{
Covering maps for $G(2k,2,2)=\langle \sigma_{2k},\tau_{2k}^{2},\rho\rangle$ ($k>1$).}\label{table:2k-2-2}
\end{table}

 \begin{table}[h]
\begin{equation}\nonumber
\begin{array}{|c||c|c|c|}\hline
e&K& K\rightarrow G_7&M_{K}\rightarrow M_{G_7}
\\\hline
\partial_x
&G_5=\langle r_1,rr_2r\rangle
&g\mapsto g
&(x',y')\mapsto (x',(y')^2)
\\\hline
\partial_y
&G_6=\langle r,r_1\rangle
&g\mapsto g
&(x',y')\mapsto ((y')^3,x')
\\\hline
12i\sqrt{3}\partial_x+\partial_y
&G_6=\langle r,r_1\rangle
&g\mapsto s^{-1}gs
&(x',y')\mapsto (-(y')^3+12i\sqrt{3}x',x')
\\\hline
\end{array}
\end{equation}
\caption{Covering maps for
$G_7=\langle r,r_1,r_2\rangle$. }\label{table:G7}
\end{table}

\begin{table}[h]
\begin{equation}\nonumber
\begin{array}{|c||c|c|c|}\hline
e&K&K\to G_{11}&M_{K}\rightarrow M_{G_{11}}
\\\hline
\partial_x
&G_{10}=\langle r_1,r_3^{-1}r_4r_3\rangle
&g\mapsto g
&(x',y')\mapsto (x',(y')^2)
\\\hline
\partial_y
&G_9=\langle r_3,r_4\rangle
&g\mapsto g
&(x',y')\mapsto ((y')^3,x')
\\\hline
\partial_x+\partial_y
&G_{14}=\langle r_1,r_4^{^2}r_3r_4^2\rangle
&g\mapsto g
&(x',y')\mapsto(x'+108(y')^4, x')
\\\hline
\end{array}
\end{equation}
\caption{Covering maps for
$G_{11}=\langle r_1,r_3,r_4\rangle$.
}\label{table:G11}
\end{table}
\begin{table}[h]
\begin{equation}\nonumber
\begin{array}{|c||c|c|c|}\hline
e&K&K\to G_{15}&M_{K}\rightarrow M_{G_{15}}
\\\hline
\partial_x
&G_{14}=\langle r_1,rr_3r\rangle&
g\mapsto g
&(x',y')\mapsto(x',(y')^2)
\\\hline
\end{array}
\end{equation}
\caption{A covering map for $G_{15}=\langle r, r_1,r_3 \rangle$.
}\label{table:G15}
\end{table}
\begin{table}[h]
\begin{equation}
\begin{array}{|c||c|c|c|}\hline
e&K&K\to G_{19}&M_{K}\rightarrow M_{G_{19}}
\\\hline
\partial_x
&G_{18}=\langle r_1^2,r_5\rangle
&g\mapsto g
&(x',y')\mapsto (x',(y')^2)
\\\hline
\partial_y
&G_{17}=\langle r,r_5\rangle
&g\mapsto g
&(x',y')\mapsto ((y')^3,x')
\\\hline
-60\sqrt{5}\partial_x+\partial_y
&G_{21}=\langle r,r_5 r_1 r_5^{-1}\rangle
&g\mapsto g
&(x',y')\mapsto(-60\sqrt{5} x'+(y')^5, x')
\\\hline
\end{array}\nonumber
\end{equation}
\caption{Covering maps for 
$G_{19}=\langle r,r_1,r_5\rangle$. }\label{table:G19}
\end{table}



\begin{thebibliography}{10}
\bibitem{Arsie-Lorenzoni2013}
Arsie, Alessandro and  Lorenzoni, Paolo,
{\it From the Darboux--Egorov system to bi-flat $F$-manifolds},
 J. Geom. Phys. {\bf 70} (2013), 98--116. 

\bibitem{Arsie-Lorenzoni2016}
\bysame,
{\it Complex reflection groups, logarithmic connections and bi-flat $F$-manifolds},
Lett. Math. Phys. {\bf 107} (2017), no.10, 1919--1961. 

\bibitem{Bessis2006}
Bessis, David,
{\it Finite complex reflection arrangements are $K(\pi,1)$},
Ann. of Math. (2) {\bf 181} (2015), no. 3, 809--904.

\bibitem{Dubrovin1993}
Dubrovin, Boris,
{\it Geometry of 2D topological field theories}, 
in Integrable systems and quantum groups (Montecatini Terme, 1993), 120--348, 
Lecture Notes in Math. {1620}, Springer, Berlin, 1996. 


\bibitem{Dubrovin1998}
\bysame, 
{\it Differential geometry of the space of orbits of a Coxeter group},
in Surveys in differential geometry: integrable systems, 181--211, 
Surv. Differ. Geom., 4, Int. Press, Boston, MA, 1998. 


\bibitem{Dubrovin2004}
\bysame, 
{\it On almost duality for Frobenius manifolds}, 
in Geometry, topology, and mathematical physics, 75--132,
Amer. Math. Soc. Transl. Ser. 2, 212, Amer. Math. Soc., Providence, RI, 2004. 

\bibitem{Hertling2002}
Hertling, Claus,
{\it Frobenius manifolds and moduli spaces for singularities}. 
Cambridge Tracts in Mathematics, 151. Cambridge University Press, Cambridge, 2002. x+270 pp.
 
\bibitem{KatoManoSekiguchi2014}
Kato, Mitsuo; Mano, Toshiyuki; Sekiguchi, Jiro,
{\it Flat structures without potentials}, 
Rev. Roumaine Math. Pures Appl. {\bf 60} (2015), no. 4, 481--505. 

\bibitem{KatoManoSekiguchi2015}
\bysame,
{\it Flat structure on the space of isomonodromic deformations},
arXiv:1511.01608v1 [math.CA]. 

\bibitem{LehrerTaylor}
Lehrer, Gustav I. and Taylor, Donald E.,
{\it Unitary reflection groups},
Australian Mathematical Society Lecture Series, 20. Cambridge University Press, Cambridge, 2009. viii+294 pp.

\bibitem{OrlikSolomon}
Orlik, Peter and Solomon, Louis,
{\it Unitary reflection groups and cohomology}, 
Invent. Math. {\bf 59} (1980), no. 1, 77--94.

\bibitem{OrlikTerao}
Orlik, Peter and Terao, Hiroaki,
{\it Arrangements of hyperplanes}, Grundlehren der Mathematischen Wissenschaften 
300. Springer-Verlag, Berlin, 1992. xviii+325 pp. 

\bibitem{Sabbah} 
Sabbah, Claude, 
{\it D\'eformations isomonodromiques et vari\'et\'es de Frobenius},
EDP Sciences, Les Ulis; CNRS \'{E}ditions, Paris, 2002. xvi+289 pp. 

\bibitem{SaitoSekiguchiYano}
Saito, Kyoji; Yano, Tamaki; Sekiguchi, Jiro,
{\it On a certain generator system of the ring of invariants of a finite reflection group}, 
Comm. Algebra {\bf 8} (1980), no. 4, 373--408. 

\bibitem{Saito1993}
Saito, Kyoji, 
{\it On a linear structure of the quotient variety by a finite reflexion group},
Publ. Res. Inst. Math. Sci. {\bf 29} (1993), no. 4, 535--579.
(Preprint version: RIMS-288, Kyoto Univ., Kyoto, 1979.)

\bibitem{ShephardTodd}
Shephard, G. C. and Todd, J. A.,
{\it Finite unitary reflection groups},
Canadian J. Math. {\bf 6} (1954). 274--304. 
\end{thebibliography}
\end{document}